\documentclass[12pt,reqno]{amsart}

\usepackage{amssymb,amsmath,amsbsy,eucal,amsfonts,mathrsfs,latexsym}
\usepackage{graphicx,verbatim,psfrag,epsfig,enumerate}
\usepackage{stmaryrd}
\usepackage{lscape}
\usepackage{url}
\usepackage{subfig}
\usepackage{multirow}
\usepackage{color}

\newtheorem{theorem}{Theorem}[section]
\newtheorem{lemma}[theorem]{Lemma}
\theoremstyle{definition}

\newtheorem{remark}{{\it Remark}}[section]

\numberwithin{equation}{section}

\newcommand{\lb}{\llbracket}
\newcommand{\rb}{\rrbracket}
\newcommand{\Lb}{\{\hspace{-4.0pt}\{}
\newcommand{\Rb}{\}\hspace{-4.0pt}\}}
\newcommand{\Bb}{{\boldsymbol{b}}}
\newcommand{\Bc}{{\boldsymbol{c}}}
\newcommand{\Bd}{{\boldsymbol{d}}}

\newcommand{\Bf}{{\boldsymbol{f}}}
\newcommand{\Bg}{{\boldsymbol{g}}}
\newcommand{\Bh}{{\boldsymbol{h}}}

\newcommand{\Bn}{{\boldsymbol{n}}}

\newcommand{\Bu}{{\boldsymbol{u}}}
\newcommand{\Bv}{{\boldsymbol{v}}}
\newcommand{\Bw}{{\boldsymbol{w}}}
\newcommand{\Bx}{{\boldsymbol{x}}}

\newcommand{\Bz}{{\boldsymbol{z}}}
\newcommand{\Bmu}{\boldsymbol{\mu}}

\newcommand{\BC}{{\boldsymbol{C}}}

\newcommand{\BK}{{\boldsymbol{K}}}
\newcommand{\BL}{{\boldsymbol{L}}}
\newcommand{\BM}{{\boldsymbol{M}}}

\newcommand{\BV}{{\boldsymbol{V}}}
\newcommand{\BW}{{\boldsymbol{W}}}

\newcommand{\BZ}{{\boldsymbol{Z}}}
\newcommand{\Bbeta}{{\boldsymbol{\beta}}}

\newcommand{\Ce}{{\mathcal E}}
\newcommand{\Cf}{{\mathcal F}}

\newcommand{\Cp}{{\mathcal P}}

\newcommand{\Ct}{{\mathcal T}}

\newcommand{\uhat}{\widehat{\boldsymbol{u}}_h}
\newcommand{\bhat}{{\widehat{\boldsymbol{b}}^t_h}}
\newcommand{\rhat}{{\widehat{{r}}_h}}
\newcommand{\vhat}{\widehat{\boldsymbol{v}}}
\newcommand{\chat}{{\widehat{\boldsymbol{c}}^t}}
\newcommand{\shat}{{\widehat{{s}}}}

\newcommand{\RmL}{{\mathrm{L}}}
\newcommand{\RmI}{{\mathrm{I}}}
\newcommand{\RmG}{{\mathrm{G}}}

\newcommand{\bint}[2]{( #1\,,\,#2 )_{\Ct_h}}
\newcommand{\bintEh}[2]{\langle #1\,,\,#2 \rangle_{\partial{\Ct_h}}}

\definecolor{red}{rgb}{0,0,0}
\newcommand\red[1]{\textcolor{red}{#1}}
\definecolor{blue}{rgb}{0,0,0}
\newcommand\blue[1]{\textcolor{blue}{#1}}

\begin{document}

\title[DG for incompressible MHD]{A Mixed DG method 
and an HDG method for incompressible magnetohydrodynamics}

\author{Weifeng Qiu}
\address{Department of Mathematics, City University of Hong Kong, 83 Tat Chee Avenue, Kowloon, Hong Kong, China}
\email{weifeqiu@cityu.edu.hk}
\thanks{
Weifeng Qiu is partially supported by a grant from the Research Grants 
Council of the Hong Kong Special Administrative Region, China 
(Project No. CityU 11302014). 
}

\author{Ke Shi}
\address{Department of Mathematics $\&$ Statistics, Old Dominion University, Norfolk, VA 23529, USA}
\email{kshi@odu.edu}
\thanks{
Ke Shi is partially supported by SRFP grant from the Research Foundation, Old Dominion University. 
As a convention the names of the authors are alphabetically ordered. 
Both authors contributed equally in this article. 
}

%\maketitle
%\date{\today}

\begin{abstract}
In this paper we propose and analyze a mixed DG method and an HDG method for the stationary Magnetohydrodynamics (MHD) 
equations with two types of boundary (or constraint) conditions. The mixed DG method is {{based on}} a recent work proposed by 
Houston et. al. {{in \cite{HSW09}}} for the linearized MHD. With two novel discrete Sobolev embedding type estimates for the discontinuous 
polynomials, we provide a priori error estimates for the method on the nonlinear MHD equations. In the smooth case, 
we have optimal convergence rate for the velocity, magnetic field and pressure in the energy norm, the Lagrange multiplier 
only has suboptimal convergence order. With the minimal regularity assumption on the exact solution, the approximation is 
optimal for all unknowns. To the best of our knowledge, this is the first a priori error estimates of DG methods for 
{{the}} nonlinear MHD equations. In addition, we also propose and analyze the first divergence-free HDG method for the problem with several unique features comparing with the mixed DG method.
\end{abstract}



\maketitle

 \section{Introduction}
MHD describes the interaction of electrically conducting fluids and electromagnetic fields  \cite{Davidson01, GerbeauBris2006, Muller01}. Examples of such magneto-fluids include plasmas, liquid metals, and salt water or electrolytes. We refer to \cite{GerbeauBris2006, Moreau, HughesYoung99} for a more comprehensive discussion on the applications of the MHD system. 
The physical model is based on two principles: first, the motion of a conducting fluid in the presence of a magnetic field induces an electric current which also interacts with the existing electromagnetic field. Second, the Lorentz force generated by the current and the magnetic field also affects the motion of the fluid. The governing equations of the stationary incompressible MHD system can be written as: 
\begin{subequations}
\label{mhd_eqs}
\begin{align}
\label{mhd_eq1}
-\nu \Delta\Bu + (\Bu\cdot \nabla) \Bu + \nabla p - \kappa (\nabla\times \Bb) \times \Bb & = \Bf \quad \text{ in } \Omega, \\
\label{mhd_eq2} 
\kappa \nu_{m} \nabla\times (\nabla\times \Bb) + \nabla r - \kappa \nabla \times (\Bu \times \Bb) & = \Bg \quad \text{ in } \Omega, \\
\label{mhd_eq3}
\nabla \cdot \Bu & = 0 \quad \text{ in } \Omega, \\
\label{mhd_eq4}
\nabla\cdot \Bb & = 0 \quad \text{ in } \Omega, \\
\label{mhd_eq5}
\Bu & = \boldsymbol{0} \quad \text{ on } \partial\Omega, \\
\label{mhd_eq8}
 \int_{\Omega} p \, d\Bx & = 0.
\end{align}
\end{subequations}
The domain $\Omega$ is a simply connected, bounded Lipschitz polyhedron in $\mathbb{R}^{3}$.  
We denote by $\Bn$ the unit outward normal vector on $\partial\Omega$. The unknowns are the velocity $\Bu$, the pressure $p$, magnetic field $\Bb$ and the Lagrange multiplier $r$ associated with the divergence constraint on the magnetic field $\Bb$. The functions $\Bf, \Bg$ are external force terms. 
These equations are characterized by three dimensionless parameters: the hydrodynamic Reynolds number $\text{Re} = \nu^{-1}$, the magnetic Reynolds number $\text{Rm} = \nu_{m}^{-1}$ and the coupling number $\kappa$. We refer to \cite{ArmeroSimo96, Davidson01, GerbeauBris2006} for further discussion of these parameters 
and their typical values. We consider two types of boundary (or constraint) conditions for the magnetic field $\Bb$ 
and the Lagrange multiplier $r$. The first type is 
\begin{subequations}
\label{conds_type1}
\begin{align}
\label{mhd_eq6}
\Bn \times \Bb & = \boldsymbol{0} \quad \text{ on } \partial\Omega, \\ 
\label{mhd_eq7}
r & = 0 \quad \text{ on } \partial\Omega. 
\end{align}
\end{subequations}
The second type is 
\begin{subequations}
\label{conds_type2}
\begin{align}
\label{mhd_eq9}
\Bb \cdot \Bn & = 0 \quad \text{ on } \partial\Omega, \\ 
\label{mhd_eq10}
\Bn \times (\nabla\times \Bb)  & = \boldsymbol{0} \quad \text{ on } \partial\Omega,\\ 
\label{mhd_eq11}
\int_{\Omega} r d\Bx & = 0.
\end{align}
\end{subequations}

Due to its significant role in applications, there have been many studies on the MHD equations see \cite{CaoWu10,Chae08, ChenMiaoZhang08, Badia2013, Banas2010, DongHeZhang, SalahSoulHab01,GerbeauBris2006,GreifLi2010,Gunzburger91,Prohl2008, MeirSchmidt99, Shoetzau2004, ZhangHeYang} and the references therein. Designing and analyzing numerical methods to solve 
this system is in general a challenging task due to multiple vector and scalar unknowns, to the various differential 
operators involved, and to the nonlinearities of the PDEs. To the best of the our knowledge, all existing finite element 
methods mentioned above for (\ref{mhd_eqs}) use conforming elements to approximate the magnetic field $\boldsymbol{b}$. 
As a consequence, local conservation does not hold for (\ref{mhd_eq4}). For the detailed explanation of 
importance of local conservation, we refer to \cite{DSW04}.

In this paper, we propose and analyze a mixed DG method for the stationary incompressible MHD 
with two types of boundary (or constraint) conditions (\ref{conds_type1}) and (\ref{conds_type2}),  
which provides optimal convergent approximation to the velocity, pressure and magnetic field. 
Due to the nature of DG methods, the local conservation for both velocity and magnetic field is preserved. 
Our method is based on the IP-DG scheme proposed in \cite{HSW09} for the linearized MHD equation. The bottle-neck 
to extend the analysis in \cite{HSW09} to nonlinear MHD system comes from the nonlinear coupling term between 
the fluid and magnetic field. Namely, a key ingredient in the analysis is a Sobolev embedding like estimate for 
the $L^3$-norm of the magnetic field. 
%For most existing conforming methods \cite{Shoetzau2004, GreifLi2010, ZhangHeYang}, this estimate is well established 
%see \cite{Monk2003, GiraultRaviart1986}. 
In Section 5 \& 6, we develop the estimates for the discrete nonconforming magnetic field with the first type of boundary 
(or constraint) conditions (\ref{conds_type1}) and the second type of boundary (or constraint) conditions (\ref{conds_type2}), 
respectively. In fact, these $L^3$-norm estimates of the discrete magnetic field help to show that our DG method has optimal 
approximation to all unknowns with the minimal regularity assumption on the exact solution. \cite{Shoetzau2004} is the first 
paper which tried to obtain optimal approximation under the minimal regularity assumption on the exact solution. However, 
according to \cite[Section~$1$]{ZhangHeYang}, the author of \cite{Shoetzau2004} failed to prove 
\cite[Proposition $3.2$]{Shoetzau2004} and \cite[Corollary $3.1$]{Shoetzau2004}, which are indispensable 
to give an error estimate. In \cite{ZhangHeYang}, a correct error analysis is given to show that 
with the first type of boundary (or constraint) conditions (\ref{conds_type1}), 
optimal convergence is achieved by the conforming method in \cite{Shoetzau2004} 
under the minimal regularity assumption. 
Up to our knowledge, with the second type of boundary (or constraint) conditions (\ref{conds_type2}), our DG method is 
the first numerical method shown to achieve optimal convergence of all unknowns under the minimal regularity assumption. 

Comparing with conforming mixed methods, DG approach has several attractive features such as local conservation, high order accuracy, {\em hp-}adaptivity, easy implementation. Nevertheless DG methods are also critized with much more degrees of freedoms comparing with conforming methods. This disadvantage becomes more severe in problems involving multiple vector and scalar unkowns such as MHD. In order to make the DG approach more competitive, in Section 9 we propose a new hybridizable DG method (HDG) for the MHD problem. Thanks to the nature of the HDG framework, the scheme can be hybridized so that the global degree of freedoms can be reduced significantly and is more efficient than existing mixed methods when high order polynomial spaces are employed. In addition, the proposed HDG method provides exactly divergence-free velocity field while maintaining all existing features of HDG framework. As a consequence, the errors of the velocity and magnetic fields are independent of the pressure. Violation of divergence-free constraint can cause large errors in practice even for stable elements, we refer a review paper \cite{JLMNR2016} for more discussions on this issue. 

The rest of the paper is organized as follows: in Section 2 we present the mixedDG scheme for the MHD system and introduce 
the notations and definitions; in Section 3 we present our main results; Section 4 provides several auxiliary results 
which needed for the proofs; Section 5-8 are the detail proofs for the main results; In Section, we propose and analyze a new HDG method for the MHD system; concluding remarks are in Section 10.

\section{Mixed Discontinuous Galerkin method}
To define the DG method for the problem, we adopt notations and norms in \cite{HSW09}. 
We consider a family of conforming triangulations $\Ct_{h}$ made of shape-regular tetrahedra. 
We denote by $\Cf_{h}^{I}$ the set of all interior faces of $\Ct_{h}$, and by $\Cf_{h}^{B}$ the set of 
all boundary faces. We define $\Cf_{h}:= \Cf_{h}^{I} \cup \Cf_{h}^{B}$. $h_{K}$ denotes the diameter of 
the element $K$, and $h_{F}$ is the diameter of the face $F$. The mesh size of $\Ct_{h}$ is defined as 
$h:= \max_{K\in \Ct_{h}}h_{K}$. We denote by $\Bn_{K}$ the unit outward normal vector on $\partial K$.  
We also introduce the average and jump operators. Let $F = \partial K\cap \partial K^{\prime}$ be 
an interior face shared by $K$ and $K^{\prime}$. Let $\phi$ be a generic piecewise smooth function 
(scalar-, vector- or tensor-valued). We define the average of $\phi$ on $F$ as 
$\Lb \phi \Rb := \frac{1}{2}(\phi + \phi^{\prime})$ where $\phi$ and $\phi^{\prime}$ denote the trace of 
$\phi$ from the interior of $K$ and $K^{\prime}$. Furthermore, let $u$ be a piecewise smooth function 
and $\Bu$ a piecewise smooth vector-valued field. Analogously, we define the following jumps on $F$:
\begin{align*}
& \lb u \rb  := u\Bn_{K} + u^{\prime} \Bn_{K^{\prime}},  
&\lb \Bu \rb := \Bu\otimes \Bn_{K}+ \Bu^{\prime}\otimes \Bn_{K^{\prime}},\\
&\lb \Bu \rb_{T} := \Bn_{K}\times \Bu + \Bn_{K^{\prime}}\times \Bu^{\prime},
& \lb \Bu \rb_{N} := \Bu\cdot \Bn_{K} + \Bu^{\prime}\cdot \Bn_{K^{\prime}}.
\end{align*}
On a boundary face $F= \partial K \cap \partial\Omega$, we set accordingly $\Lb \phi \Rb := \phi$, 
$\lb u\rb := u\Bn$, $\lb \Bu \rb := \Bu\otimes \Bn$, $\lb \Bu \rb_{T} := \Bn \times \Bu$ 
and $\lb \Bu\rb_{N} := \Bu\cdot \Bn$. 

Throughout this paper, we assume the integer $k\geq 1$. Here $P_{k}(\Ct_{h};\mathbb{R}^{3})$ denotes the space contains 
vector-valued piecewise polynomials of degree no more than $k$ on $\mathcal{T}_h$. Similarly, $P_{k}(\mathcal{T}_h)$ 
denotes the space contains piecewise polynomials of degree no more than $k$ on $\mathcal{T}_h$. In addition, standard 
inner product notations are used throughout the paper. Namely, $(f, g)_D = \int_D f g d \Bx$ for $D \in \mathbb{R}^3$ 
and $ \langle f, g \rangle_D := \int_D fg ds$ for $D \in \mathbb{R}^2$. We use the standard notations for Sobolev norms. 
In addition, we use the following notations and spaces:
\begin{align*}
H(\text{div}; \Omega) &:= \{\Bv \in L^2(\Omega; \mathbb{R}^{3}), \nabla\cdot \Bv \in L^2(\Omega)\}; \\
H(\text{div}^0; \Omega) &:= \{\Bv \in H(\text{div}, \nabla\cdot \Bv = 0\}; \\
H(\text{curl}; \Omega) &:= \{\Bc \in L^2(\Omega; \mathbb{R}^{3}), \nabla \times \Bc \in L^2(\Omega; \mathbb{R}^{3})\};\\
H(\text{curl}^0; \Omega) &:= \{\Bv \in H(\text{curl}, \nabla \times \Bv = 0\}; \\
H_0(\text{curl}; \Omega) &:= \{\Bc \in H(\text{curl}; \Omega), \Bn \times \Bc = 0 \; \text{on} \; \partial \Omega\};\\
\|\Bv\|_{L^2(\Ct_h)}&:= \big( \Sigma_{K \in \Ct_h} \|\Bv\|_{K}^{2}\big)^{\frac{1}{2}}, 
\quad \|v\|_{L^2(\Cf_h)}:= \big( \Sigma_{F \in \Cf_h} \|v\|_{L^2(F)}^{2} \big)^{\frac{1}{2}}.
\end{align*}

In order to define mixed discontinuous Galerkin method, we introduce the finite dimensional spaces:
\begin{align*}
\BV_{h} := P_{k}(\Ct_{h};\mathbb{R}^{3}),
\quad Q_{h} := P_{k-1}(\Ct_{h}) \cap L_{0}^{2}(\Omega), 
\quad \BC_{h} := P_{k}(\Ct_{h};\mathbb{R}^{3}),
\quad S_{h} := P_{k+1}(\Ct_{h}). 
\end{align*}
$\BV_{h}$ is for the approximation to the velocity field $\Bu$. $Q_{h}$ is for the approximation 
to the pressure $p$. $\BC_{h}$ is for the approximation to the magnetic field $\Bb$. 
$S_{h}$ is for the approximation to the Lagrange multiplier $r$ with the first type of boundary (or constraint) 
conditions (\ref{conds_type1}). $S_{h}\cap L_{0}^{2}(\Omega)$ is for the approximation to the Lagrange 
multiplier $r$ with the second type of boundary (constraint) conditions (\ref{conds_type2}).

\subsection{Mixed Discontinuous Galerkin method for the first type of boundary (or constraint) conditions}
The mixed DG method for the first type of boundary (or constraint) conditions (\ref{conds_type1}) for 
the magnetic field and the Lagrange multiplier seeks an approximation $(\Bu_{h}, \Bb_{h}, p_{h}, r_{h})\in \BV_{h}\times 
\BC_{h}\times Q_{h}\times S_{h}$ to the exact solution $(\Bu,\Bb, p, r)$ of (\ref{mhd_eqs}) with (\ref{conds_type1}). 
The method determines the approximate solution by requiring that it solves the following weak formulation:
\begin{subequations}
\label{DG_mhd}
\begin{align}
\label{DG_mhd_eq1}
A_{h}(\Bu_{h}, \Bv) + O_{h}(\Bbeta; \Bu_{h}, \Bv) + C_{h}(\Bd; \Bv, \Bb_{h}) + B_{h}(\Bv, p_{h}) & = (\Bf, \Bv)_{\Omega}, \\
\label{DG_mhd_eq2} 
M_{h}(\Bb_{h}, \Bc) - C_{h}(\Bd; \Bu_{h}, \Bc) + D_{h} (\Bc, r_{h}) & = (\Bg, \Bc)_{\Omega}, \\
\label{DG_mhd_eq3}
B_{h} (\Bu_{h}, q) &= 0,\\
\label{DG_mhd_eq4} 
D_{h}(\Bb_{h}, s) - J_{h}(r_{h}, s) &= 0,
\end{align}
\end{subequations}
for all $(\Bv, \Bc, q, s)\in \BV_{h}\times \BC_{h}\times Q_{h}\times S_{h}$. We put 
\begin{align}
\label{DG_input_data}
\Bbeta = \mathbb{P}(\Bu_{h}, \Lb \Bu_{h} \Rb),\qquad \Bd = \Bb_{h}. 
\end{align}
The postprocessing operator $\mathbb{P}$ from $H^1(\Ct_h;\mathbb{R}^{3})\times L^2(\Cf_h;\mathbb{R}^{3})$ into
\[
\{\Bv\in H(\text{div};\Omega):\Bv|_{K}\in
RT_{k}(K):=P_{k}(K;\mathbb{R}^{3})+\Bx P_{k}(K)\}
\]
is defined on the element $K$ by the following equations, see \cite{Cesmelioglu2016}:
\begin{subequations}
\label{post_process_op}
\begin{alignat}{2}
\label{post_process_op_eq1}
 (\mathbb{P}(\Bu_{h}, \Lb \Bu_{h} \Rb ) - \Bu_{h}, \Bv)_{K}  &=0&&\quad 
\forall \Bv\in P_{k-1}(K;\mathbb{R}^{3}),\\
\label{post_process_op_eq2}
 \langle (\mathbb{P}(\Bu_{h}, \Lb \Bu_{h} \Rb) - \Lb \Bu_{h} \Rb)\cdot\Bn, 
\lambda\rangle_{\partial K}&=0&&\quad \forall \lambda \in P_{k}(F), 
\mbox{ for each face $F$ of $K$}.
\end{alignat}
\end{subequations}
Here, the forms $A_{h}, O_{h}$ and $B_{h}$ are related to the discretization of the Navier-Stokes equations (fluid). 
The forms $M_{h}, D_{h}$ and $J_{h}$ are related to the discretization of the Maxwell equations (magnetic field). 
The form $C_{h}$ couples the Maxwell equations to the Navier-Stokes equations. 
These forms are defined in \cite[Section~$2.3$]{HSW09}. We write them below. 

First, the form $A_{h}$ is chosen as the standard interior penalty form 
\begin{align*}
A_{h}(\Bu, \Bv):= & \Sigma_{K\in\Ct_{h}} (\nu \nabla \Bu, \nabla \Bv)_{K} - \Sigma_{F\in \Cf_{h}}
\langle \Lb \nu \nabla \Bu\Rb, \lb \Bv\rb\rangle_{F}\\
 & \quad - {{\Sigma_{F\in \Cf_{h}} \langle \Lb \nu \nabla \Bv\Rb, \lb \Bu\rb\rangle_{F}  }}
 + \Sigma_{F\in \Cf_{h}}\frac{\nu a_{0}}{h_{F}} \langle \lb \Bu\rb, \lb \Bv\rb\rangle_{F}.
\end{align*}
The parameter $a_{0}>0$ is a sufficiently large stabilization parameter; see \cite[Proposition~$2.4$]{HSW09}. 
For the convective form, we take the usual upwind form defined by 
\begin{align*}
O_{h}(\Bbeta; \Bu, \Bv) := & \Sigma_{K\in \Ct_{h}} (\Bbeta \cdot \nabla \Bu_{h}, \Bv)_{K} 
+ \Sigma_{K\in \Ct_{h}}\langle (\Bbeta \cdot \Bn_{K}) (\Bu^{e} - \Bu), \Bv\rangle_{\partial K_{-}\backslash \Gamma_{-}}\\ 
&\quad - \langle (\Bbeta \cdot \Bn) \Bu, \Bv\rangle_{\Gamma_{-}}.
\end{align*}
Here, $\Bu^{e}$ is the value of the trace of $\Bu$ taken from the exterior of $K$, 
$\partial K_{-}:= \{\Bx\in \partial K: \Bbeta(\Bx)\cdot \Bn_{K}(\Bx) < 0\}$ 
and $\Gamma_{-} := \{ \Bx\in \partial \Omega: \Bbeta (\Bx)\cdot \Bn_{K}(\Bx) < 0 \}$. 
The form $B_{h}$ related to the divergence constraint on $\Bu$ is defined by 
\begin{align*}
B_{h} (\Bu, \Bv) := - \Sigma_{K\in \Ct_{h}} (\nabla \cdot \Bu, q)_{K} 
+\Sigma_{F\in\Cf_{h}} \langle \Lb q\Rb , \lb \Bu\rb_{N}\rangle_{F}.
\end{align*}

Next, we define the forms for the discretization of the Maxwell operator. 
The form $M_{h}$ for the $\text{curl-curl}$ operator is given by 
\begin{align*}
M_{h}(\Bb, \Bc) := & \Sigma_{K\in \Ct_{h}} (\kappa \nu_{m} \nabla \times \Bb, \nabla\times \Bc)_{K} 
- {{\Sigma_{F\in \Cf_{h}} \langle \Lb \kappa \nu_{m} \nabla\times \Bb \Rb,  \lb \Bc \rb_{T}\rangle_{F}}}\\
& \quad - \Sigma_{F\in \Cf_{h}} \langle \Lb \kappa \nu_{m} \nabla\times \Bc \Rb, \lb \Bb\rb_{T} \rangle_{F} 
+\Sigma_{F\in \Cf_{h}} \frac{\kappa\nu_{m}m_{0}}{h_{F}} \langle \lb \Bb\rb_T, \lb \Bc\rb_T \rangle_{F}.
\end{align*}
As for the diffusion form, the stabilization parameter $m_{0}>0$ must be chosen large enough; see \cite[Proposition~$2.4$]{HSW09}. 
The form $D_{h}$ for the divergence-free constraint on $\Bb$ is given by 
\begin{align*}
D_{h}(\Bb, s) := \Sigma_{K\in \Ct_{h}} (\Bb, \nabla s)_{K} - \Sigma_{F\in \Cf_{h}} \langle \Lb \Bb\Rb, \lb s\rb\rangle_{F}.
\end{align*}
The form $J_{h}$ is the stabilization term that ensures the $H^{1}$-conformity of the multiplier $r_{h}$ in a weak sense. 
It is given by 
\begin{align*}
J_{h}(r,s) := \Sigma_{F\in \Cf_{h}} \frac{s_{0}}{\kappa \nu_{m}h_{F}} \langle \lb r\rb, \lb s\rb\rangle_{F},
\end{align*}
with $s_{0}>0$ denoting a third stabilization parameter. 

Finally, the coupling form $C_{h}$ is defined by 
\begin{align*}
C_{h}(\Bd; \Bv, \Bb) := \Sigma_{K\in\Ct_{h}}  (\kappa(\Bv\times \Bd), \nabla \times \Bb)_{K} 
- \Sigma_{F\in \Cf_{h}^{I}} \langle \kappa \Lb \Bv\times \Bd\Rb, \lb \Bb\rb_{T}\rangle_{F}.
\end{align*}

\subsection{Mixed Discontinuous Galerkin method for the second type of boundary (or constraint) conditions}
We can obtain the mixed DG formulation for the second type of boundary conditions \eqref{conds_type2} with slight modification of the above method. Namely, we use the same spaces  for all unknowns except  the Lagrange multiplier {{for which we use $S_h \cap L^2_0(\Omega)$ instead}}. Now the formulation becomes: find $(\Bu_h, \Bb_h, p_h, r_h) \in \BV_{h}\times \BC_{h}\times Q_{h}\times S_{h}\cap L_{0}^{2}(\Omega)$ satisfying:
\begin{subequations}
\label{DG_mhd_type2}
\begin{align}
\label{DG_mhd_type2_eq1}
A_{h}(\Bu_{h}, \Bv) + O_{h}(\Bbeta; \Bu_{h}, \Bv) + C_{h}(\Bd; \Bv, \Bb_{h}) + B_{h}(\Bv, p_{h}) & = (\Bf, \Bv)_{\Omega}, \\
\label{DG_mhd_type2_eq2} 
M_{h}^{I}(\Bb_{h}, \Bc) - C_{h}(\Bd; \Bu_{h}, \Bc) + D_{h}^{I} (\Bc, r_{h}) & = (\Bg, \Bc)_{\Omega}, \\
\label{DG_mhd_type2_eq3}
B_{h} (\Bu_{h}, q) &= 0,\\
\label{DG_mhd_type2_eq4} 
D_{h}^{I}(\Bb_{h}, s) - J_{h}^{I}(r_{h}, s) &= 0,
\end{align}
\end{subequations}
for all $(\Bv, \Bc, q, s)\in \BV_{h}\times \BC_{h}\times Q_{h}\times S_{h}\cap L_{0}^{2}(\Omega)$. 
As (\ref{DG_input_data}), we put 
\begin{align*}
\Bbeta = \mathbb{P}(\Bu_{h}, \Lb \Bu_{h} \Rb),\qquad \Bd = \Bb_{h}. 
\end{align*}
where the forms $M_{h}^{I}$, $D_{h}^{I}$ and $J_{h}^{I}$ are defined as: 
\begin{align*}
M_{h}^{I}(\Bb, \Bc) := & \Sigma_{K\in \Ct_{h}} (\kappa \nu_{m} \nabla \times \Bb, \nabla\times \Bc)_{K} 
- {{\Sigma_{F\in \Cf_{h}^{I}} \langle \Lb \kappa \nu_{m} \nabla\times \Bb \Rb,  \lb \Bc \rb_{T}\rangle_{F}}}\\
& \quad - \Sigma_{F\in \Cf_{h}^{I}} \langle \Lb \kappa \nu_{m} \nabla\times \Bc \Rb, \lb \Bb\rb_{T} \rangle_{F} 
+\Sigma_{F\in \Cf_{h}^{I}} \frac{\kappa\nu_{m}m_{0}}{h_{F}} \langle \lb \Bb\rb_T, \lb \Bc\rb_T \rangle_{F}, \\
D_{h}^{I}(\Bb, s) := & \Sigma_{K\in \Ct_{h}} (\Bb, \nabla s)_{K} - \Sigma_{F\in \Cf_{h}^{I}} 
\langle \Lb \Bb\Rb, \lb s\rb\rangle_{F}, \\ 
J_{h}^{I}(r,s) := & \Sigma_{F\in \Cf_{h}^{I}} \frac{s_{0}}{\kappa \nu_{m}h_{F}} \langle \lb r\rb, \lb s\rb\rangle_{F}.
\end{align*}

\section{Main results}
We first present two novel discrete Sobolev embedding results which {{is the key ingredient}} for the analysis. {\blue{For any $\Bb_h \in \BC_h$, we define the {\em discrete divergence} of $\Bb_h$, denoted by $\nabla_h \cdot \Bb_h$ to be the unique function in $H^1_0(\Omega) \cap S_h$ satisfying:
\begin{equation}\label{discrete_div}
\bint{\nabla_h \cdot \Bb_h}{s} = -\bint{\Bb_h}{\nabla s} \quad \text{for all} \quad s \in H^1_0(\Omega) \cap S_h.
\end{equation}
Similarly, for the second type of boundary condition, we need to modify the above discrete divergence operator as: $\nabla^N_h \cdot \Bb_h$ is defined to be the unique function in $H^1(\Omega) \cap L^2_0(\Omega) \cap S_h$ satisfying:
\begin{equation}
\bint{\nabla^N_h \cdot \Bb_h}{s} = - \bint{\Bb_h}{\nabla s} \quad \text{for all $s \in H^1(\Omega) \cap L^2_0(\Omega) \cap S_h$.} 
\end{equation}
\begin{theorem}
\label{theorem_stability}
There is a positive constant $C$ such that for any $\Bb_h\in \BC_h$, we have  
\begin{align*}
\|\Bb_h \|_{\BL^3(\Omega)} \leq C \big( \|  h^{-\frac12}\llbracket\Bb_h \rrbracket_{T} \|_{L^2(\Cf_h)} 
+ \|\nabla \times \Bb_h \|_{L^2(\mathcal{T}_{h})} + \|\nabla_h \cdot \Bb_h\|_{L^2(\Ct_h)}\big).
\end{align*}
\end{theorem}
}}

In addition, we provide 
an analogue of Theorem \ref{theorem_stability} for the second type of boundary conditions as follows,

\begin{theorem}
\label{theorem_stability2}
There is a positive constant $C$ such that for any $\Bb_h\in \BC_h$, we have
\begin{align*}
\|\Bb_h \|_{\BL^3(\Omega)} \leq C \big( \|  h^{-\frac12}\llbracket\Bb_h \rrbracket_{T} \|_{L^2(\Cf_h^{I})} 
+ \|\nabla \times \Bb_h \|_{L^2(\mathcal{T}_{h})} + {\blue{\|\nabla^N_h \cdot \Bb_h\|_{L^2(\Omega)}}}\big).
\end{align*}
\end{theorem}

The norms that we are going to use in the error analysis are defined as follows:
\begin{align*}
\Vert \Bu \Vert_{V}^{2} &:= \Vert \nabla \Bu \Vert_{L^2(\mathcal{T}_h)}^{2} 
+ \Vert h^{-\frac12}  \lb \Bu\rb\Vert_{L^{2}(\Cf_h)}^{2}, \\
\Vert \Bb \Vert_{C}^{2} &:= \Vert \Bb \Vert_{L^{2}(\Omega)}^{2} 
+ \Vert \nabla\times \Bb\Vert_{L^{2}(\Ct_h)}^{2}
+  \Vert h^{-\frac12} \lb \Bb \rb_{T}\Vert_{L^{2}(\Cf_h)}^{2},\\
\Vert r\Vert_{S}^{2} &:= \Vert \nabla r \Vert_{L^{2}(\Ct_h)}^{2} 
+  \Vert h^{-\frac12} \lb r \rb \Vert_{L^{2}(\Cf_h)}^{2}.
\end{align*}
Finally, we use the standard $L^2$-norm for the pressure $p$. 
{\blue{
\begin{remark}\label{div_free_cond}
Notice that if a function $\Bc_h \in \BC_h$ satisfying
\begin{equation}\label{assumption_div_free}
\bint{\Bc_h}{\nabla s} = 0 \quad \text{for all $s \in H^1_0(\Omega) \cap S_h$},
\end{equation}
then we have $\nabla_h \cdot \Bc_h = 0$ and therefore by Theorem \ref{discrete_div}, we have
\begin{equation}\label{L3_divfree}
\|\Bc_h\|_{L^3(\Omega)} \le C \|\Bc_h\|_{C}.
\end{equation}
\end{remark}
}}

Now we are ready to present the well-posedness result:

The proofs of the above results are in Section \ref{proof_L3} and Section~\ref{sec_l3_type2}.

The norms that we are going to use in the error analysis are defined as follows:
\begin{align*}
\Vert \Bu \Vert_{V}^{2} &:= \Vert \nabla \Bu \Vert_{L^2(\mathcal{T}_h)}^{2} 
+ \Vert h^{-\frac12}  \lb \Bu\rb\Vert_{L^{2}(\Cf_h)}^{2}, \\
\Vert \Bb \Vert_{C}^{2} &:= \Vert \Bb \Vert_{L^{2}(\Omega)}^{2} 
+ \Vert \nabla\times \Bb\Vert_{L^{2}(\Ct_h)}^{2}
+  \Vert h^{-\frac12} \lb \Bb \rb_{T}\Vert_{L^{2}(\Cf_h)}^{2},\\
\Vert \Bb \Vert_{C^{I}}^{2} &:= \Vert \Bb \Vert_{L^{2}(\Omega)}^{2} 
+ \Vert \nabla\times \Bb\Vert_{L^{2}(\Ct_h)}^{2}
+  \Vert h^{-\frac12} \lb \Bb \rb_{T}\Vert_{L^{2}(\Cf^{I}_h)}^{2},\\
\Vert r\Vert_{S}^{2} &:= \Vert \nabla r \Vert_{L^{2}(\Ct_h)}^{2} 
+  \Vert h^{-\frac12} \lb r \rb \Vert_{L^{2}(\Cf_h)}^{2}, \\
\Vert r\Vert_{S^{I}}^{2} &:= \Vert \nabla r \Vert_{L^{2}(\Ct_h)}^{2} 
+  \Vert h^{-\frac12} \lb r \rb \Vert_{L^{2}(\Cf^{I}_h)}^{2}. 
\end{align*}
Finally, we use the standard $L^2$-norm for the pressure $p$. 

Now we are ready to present the well-posedness results for the method:
\begin{theorem}
\label{theorem_wellposed}
If the following quantities 
\begin{align}
\label{small_quantities}
\nu^{-2} \Vert \Bf \Vert_{L^{2}(\Omega)},\quad \nu^{-1}\nu_{m}^{-1} \Vert \Bf\Vert_{L^{2}(\Omega)}, \quad
\nu^{-\frac{3}{2}}\kappa^{-\frac{1}{2}}\nu_{m}^{-\frac{1}{2}} \Vert \Bg\Vert_{L^{2}(\Omega)},\quad 
\nu^{-\frac{1}{2}} \kappa^{-\frac{1}{2}}\nu_{m}^{-\frac{3}{2}} \Vert \Bg\Vert_{L^{2}(\Omega)}
\end{align}
are all small enough, the system \eqref{mhd_eqs} with the first type of boundary (or constraint) conditions (\ref{conds_type1}) 
has a unique weak solution $(\Bu, \Bb, p, r) \in H^1_0(\Omega; \mathbb{R}^{3}) 
\times H(\text{curl};\Omega) \times L^2_0(\Omega) \times H^1_0(\Omega)$ and the DG method (\ref{DG_mhd}) has a unique solution 
$(\Bu_{h}, \Bb_{h}, p_{h}, r_{h})
\in \BV_{h}\times \BC_{h}\times Q_{h} \times S_{h}$. Furthermore, 
\begin{subequations}
\label{energy_norm_bounds}
\begin{align}
\label{energy_norm_bound}
\nu^{\frac{1}{2}} \Vert \Bu\Vert_{V} + \kappa^{\frac{1}{2}} \nu_{m}^{\frac{1}{2}} 
\big( \Vert \Bb \Vert_{L^{3}(\Omega)} + \Vert \Bb \Vert_{C} \big)
\leq C \big( \nu^{-\frac{1}{2}} \Vert \Bf\Vert_{L^{2}(\Omega)}
 + \kappa^{-\frac{1}{2}}\nu_{m}^{-\frac{1}{2}} \Vert \Bg\Vert_{L^{2}(\Omega)} \big), \\
\label{discrete_energy_norm_bound}
\nu^{\frac{1}{2}} \Vert \Bu_{h}\Vert_{V} + \kappa^{\frac{1}{2}} \nu_{m}^{\frac{1}{2}} 
\big( \Vert \Bb_{h}\Vert_{L^{3}(\Omega)} + \Vert \Bb_{h}\Vert_{C} \big)
\leq C \big( \nu^{-\frac{1}{2}} \Vert \Bf\Vert_{L^{2}(\Omega)}
 + \kappa^{-\frac{1}{2}}\nu_{m}^{-\frac{1}{2}} \Vert \Bg\Vert_{L^{2}(\Omega)} \big).
\end{align} 
\end{subequations}
\end{theorem}

\begin{theorem}
\label{theorem_wellposed2}
If the following quantities 
\begin{align*}
\nu^{-2} \Vert \Bf \Vert_{L^{2}(\Omega)},\quad \nu^{-1}\nu_{m}^{-1} \Vert \Bf\Vert_{L^{2}(\Omega)}, \quad
\nu^{-\frac{3}{2}}\kappa^{-\frac{1}{2}}\nu_{m}^{-\frac{1}{2}} \Vert \Bg\Vert_{L^{2}(\Omega)},\quad 
\nu^{-\frac{1}{2}} \kappa^{-\frac{1}{2}}\nu_{m}^{-\frac{3}{2}} \Vert \Bg\Vert_{L^{2}(\Omega)}
\end{align*}
are all small enough, the system \eqref{mhd_eqs} with the second type of boundary (or constraint) conditions 
(\ref{conds_type2}) has a unique weak solution $(\Bu, \Bb, p, r) \in H^1_0(\Omega; \mathbb{R}^{3}) 
\times H(\text{curl};\Omega) \times L^2_0(\Omega) \times H^1(\Omega)\cap L_{2}^{0}(\Omega)$ and 
the DG method (\ref{DG_mhd_type2}) has a unique solution 
$(\Bu_{h}, \Bb_{h}, p_{h}, r_{h})
\in \BV_{h}\times \BC_{h}\times Q_{h} \times S_{h}\cap L_{2}^{0}(\Omega)$. Furthermore, 
\begin{align*}
\nu^{\frac{1}{2}} \Vert \Bu\Vert_{V} + \kappa^{\frac{1}{2}} \nu_{m}^{\frac{1}{2}} 
\big( \Vert \Bb \Vert_{L^{3}(\Omega)} + \Vert \Bb \Vert_{C^{I}} \big)
\leq C \big( \nu^{-\frac{1}{2}} \Vert \Bf\Vert_{L^{2}(\Omega)}
 + \kappa^{-\frac{1}{2}}\nu_{m}^{-\frac{1}{2}} \Vert \Bg\Vert_{L^{2}(\Omega)} \big), \\
\nu^{\frac{1}{2}} \Vert \Bu_{h}\Vert_{V} + \kappa^{\frac{1}{2}} \nu_{m}^{\frac{1}{2}} 
\big( \Vert \Bb_{h}\Vert_{L^{3}(\Omega)} + \Vert \Bb_{h}\Vert_{C^{I}} \big)
\leq C \big( \nu^{-\frac{1}{2}} \Vert \Bf\Vert_{L^{2}(\Omega)}
 + \kappa^{-\frac{1}{2}}\nu_{m}^{-\frac{1}{2}} \Vert \Bg\Vert_{L^{2}(\Omega)} \big).
\end{align*} 
\end{theorem}

The last two results are for the convergence of the numerical solutions. We make minimal regularity assumptions 
on the exact solutions, see Remark 4.1 in \cite{GreifLi2010}. Namely, we assume the exact solution 
$(\Bu, \Bb, p, r)$ of \eqref{mhd_eqs} possesses the smoothness
\begin{subequations}\label{regularity}
\begin{align}
&(\Bu, p) \; \in \; H^{\sigma + 1}(\Omega; \mathbb{R}^{3}) \times H^{\sigma}(\Omega), \\
& (\Bb, \nabla \times \Bb, r) \; \in \; H^{\sigma_m}(\Omega; \mathbb{R}^{3}) 
\times H^{\sigma_m}(\Omega; \mathbb{R}^{3}) \times H^{\sigma_m + 1}(\Omega),
\end{align}
\end{subequations}
for $\sigma, \sigma_m > \frac12$.

Now we are ready to state our main convergence results:
\begin{theorem}\label{errors_ub}
Let $(\Bu, \Bb, p, r)$ be the exact solution of the system \eqref{mhd_eqs} with the first type of boundary (or constraint) conditions (\ref{conds_type1}), and $(\Bu_{h}, \Bb_{h}, p_{h}, r_{h})$ be the solution of the DG method (\ref{DG_mhd}). 
With the same assumption as in Theorem \ref{theorem_wellposed}, in addition with the regularity assumption \eqref{regularity} 
and that $\dfrac{1}{\min(\nu, \nu_{m})}\Vert \Bu\Vert_{H^{1}(\Omega)}$ and 
$\dfrac{1}{\sqrt{\nu \kappa \nu_{m}}}\Vert \nabla\times \Bb \Vert_{L^{2}(\Omega)}$  
are small enough, then we have
\begin{align*}
\nu^{\frac12}\|\Bu - \Bu_h\|_V &+ \kappa^{\frac12} \nu_m^{\frac12}\|\Bb - \Bb_h\|_C + \|p - p_h\|_{L^2(\Omega)} + \|r - r_h\|_S \\
& \le \mathcal{C} h^{\text{min}\{k, \sigma, \sigma_m\}}\Big(\|\Bu\|_{H^{\sigma+1}(\Omega)} + \|p\|_{H^{\sigma}(\Omega)} 
+ \|\Bb\|_{H^{\sigma_m}(\Omega)} + \|\nabla \times \Bb\|_{H^{\sigma_m}(\Omega)} \\
&+ \|r\|_{H^{\sigma_m + 1}(\Omega)} + (\|\Bu\|_{H^{\sigma+1}(\Omega)} 
 + \|\nabla \times \Bb\|_{H^{\sigma_m}(\Omega)}) \|\Bb\|_{H^{\sigma_m}(\Omega)} \Big),
\end{align*}
here $\mathcal{C}$ depends on the physical parameters $\kappa, \nu, \nu_m$ and the external forces $\Bf, \Bg$ 
but is independent of mesh size $h$.
\end{theorem}

\begin{theorem}\label{errors_ub_type2}
Let $(\Bu, \Bb, p, r)$ be the exact solution of the system \eqref{mhd_eqs} with the second type of boundary (or constraint) conditions (\ref{conds_type2}), and $(\Bu_{h}, \Bb_{h}, p_{h}, r_{h})$ be the solution of the DG method (\ref{DG_mhd_type2}). 
With the same assumption as in Theorem \ref{theorem_wellposed2}, in addition with the regularity assumption \eqref{regularity} 
and that $\dfrac{1}{\min(\nu, \nu_{m})}\Vert \Bu\Vert_{H^{1}(\Omega)}$ and 
$\dfrac{1}{\sqrt{\nu \kappa \nu_{m}}}\Vert \nabla\times \Bb \Vert_{L^{2}(\Omega)}$  
are small enough, then we have
\begin{align*}
\nu^{\frac12}\|\Bu - \Bu_h\|_V &+ \kappa^{\frac12} \nu_m^{\frac12}\|\Bb - \Bb_h\|_{C^{I}}
 + \|p - p_h\|_{L^2(\Omega)} + \|r - r_h\|_{S^{I}} \\
& \le \mathcal{C} h^{\text{min}\{k, \sigma, \sigma_m\}}\Big(\|\Bu\|_{H^{\sigma+1}(\Omega)} + \|p\|_{H^{\sigma}(\Omega)} 
+ \|\Bb\|_{H^{\sigma_m}(\Omega)} + \|\nabla \times \Bb\|_{H^{\sigma_m}(\Omega)} \\
&+ \|r\|_{H^{\sigma_m + 1}(\Omega)} + (\|\Bu\|_{H^{\sigma+1}(\Omega)} 
 + \|\nabla \times \Bb\|_{H^{\sigma_m}(\Omega)}) \|\Bb\|_{H^{\sigma_m}(\Omega)} \Big),
\end{align*}
here $\mathcal{C}$ depends on the physical parameters $\kappa, \nu, \nu_m$ and the external forces $\Bf, \Bg$ 
but is independent of mesh size $h$.
\end{theorem}

\begin{remark}
The above results indicate that our method obtains same convergence rate as existing conforming methods \cite{ZhangHeYang,GreifLi2010}. Namely,  if the exact solution is sufficient smooth, we have optimal convergence for $\Bu, \Bb, p$ in the energy norms. The convergence rate for the Lagrange multiplier $r$ is suboptimal in the discrete $H^1$-norm. With minimal regularity assumption \eqref{regularity} the method is optimal for all unknowns.
\end{remark}

\section{Auxiliary results}

In this section we gather some auxiliary results needed to carry out the error estimates in the next section. 

\begin{lemma}
\label{lemma_proj}
For any $\Bv_{h} \in \BV_{h}$ such that $B_{h}(\Bv_{h}, q)=0$ for any $q\in Q_{h}$, 
\begin{subequations}
\label{proj_props}
\begin{align}
\label{proj_prop1}
& \nabla\cdot \mathbb{P}(\Bv_{h}, \Lb \Bv_{h} \Rb ) = 0 \quad \text{ in } \Omega,\\
\label{proj_prop2}
& \mathbb{P}(\Bv_{h}, \Lb \Bv_{h} \Rb ) \in \BV_{h}, \\
\label{proj_prop3}
& \Vert \mathbb{P}(\Bv_{h}, \Lb \Bv_{h} \Rb )\Vert_{V} \leq C \Vert \Bv_{h}\Vert_{V}, \\
\label{proj_prop4}
& \Vert \mathbb{P}(\Bv_{h}, \Lb \Bv_{h} \Rb )\Vert_{L^2(\Ct_{h})} \leq C \Vert \Bv_{h}\Vert_{L^2(\Ct_{h})},\\
\label{proj_prop5}
& \mathbb{P} (\Bv, \Bv |_{\Ce_{h}}) = \boldsymbol{\Pi}^{RT}\Bv \quad \forall \; \Bv \in H^{1}(\Omega;\mathbb{R}^{3}).  
\end{align}
\end{subequations}
Here, $\boldsymbol{\Pi}^{RT}$ is the $k$-th order Raviart-Thomas (RT) projection.
 In addition, for For any $\Bw_{h} \in \BV_{h}$ such that $B_{h}(\Bw_{h}, q)=0$ for any $q\in Q_{h}$ and $\boldsymbol{v}_h 
 \in V_h$, let $\widetilde{\Bw_h} := \mathbb{P}(\Bw_{h}, \Lb \Bw_{h} \Rb)$, we have
\[
O_{h}(\widetilde{\Bw_h}; \Bv_h, \Bv_h) = \frac12 \sum_{F \in \Cf^I_h} \langle |\widetilde{\Bw_h} \cdot \Bn| \lb \Bv_h \rb, 
\lb \Bv_h \rb \rangle_{F} +  \frac12 \sum_{F \in \Cf^B_h} \langle |\widetilde{\Bw_h} \cdot \Bn| \Bv_h , \Bv_h \rangle_{F}  \ge 0.
\]
\end{lemma}
We refer \cite{CockburnKanschatSchoetzauNS05} for the proof of above result. The next result is from \cite[Proposition~$2.4$]{HSW09}.

\begin{lemma}\label{coarsivity}
With sufficiently large parameters $a_0, m_0 > 0$, for any $\Bu_h, \Bb_h$, we have
\begin{align*}
 A_{h}(\Bu_h, \Bu_h) + M_{h}(\Bb_h, \Bb_h) \ge C \big( \nu \Vert \Bu_h \Vert_{V}^{2} 
+ \kappa \nu_{m} (\Vert \nabla \times \Bb_h\Vert_{L^2(\mathcal{T}_h)}^{2} +  \Vert h^{-\frac12} \lb \Bb_{h}\rb_{T}
\Vert_{L^2(\Cf_h)}^{2} ) \big).
\end{align*}
Here $C$ is independent of the mesh size, $\nu, \nu_m$ and $\kappa$.
\end{lemma}

\begin{lemma}
\label{lemma_Oh_cont}
There is a constant $C>0$ such that for any $(\Bd, \Bv, \Bc)\in \BC_{h}\times \BV_{h}\times \BC_{h}$ with 
$\Bd$ satisfying (\ref{assumption_div_free}), 
\begin{align*}
C_{h}(\Bd; \Bv, \Bc) \leq C \kappa \Vert \Bd\Vert_{C} \Vert \Bv\Vert_{V} \Vert \Bc\Vert_{C}. 
\end{align*}
\end{lemma}

\begin{proof}
We recall that 
\begin{align*}
C_{h}(\Bd; \Bv, \Bc) := \Sigma_{K\in\Ct_{h}}  (\kappa(\Bv\times \Bd), \nabla \times \Bc)_{K} 
- \Sigma_{F\in \Cf_{h}^{I}} \langle \kappa \Lb \Bv\times \Bd\Rb, \lb \Bc\rb_{T}\rangle_{F}.
\end{align*}
Obviously, 
\begin{align*}
\Sigma_{K\in\Ct_{h}}  (\kappa(\Bv\times \Bd), \nabla \times \Bc)_{K} 
\leq & C \kappa \Vert \Bd\Vert_{L^{3}(\Omega)} \Vert \Bv\Vert_{L^{6}(\Omega)} \Vert \nabla\times \Bc\Vert_{L^{2}(\Ct_{h})},\\
\Sigma_{F\in \Cf_{h}^{I}} \langle \kappa \Lb \Bv\times \Bd\Rb, \lb \Bc\rb_{T}\rangle_{F} 
\leq & C \kappa \Vert h^{\frac{1}{3}}\Bd\Vert_{L^{3}(\Cf_{h})} \Vert h^{\frac{1}{6}}\Bv\Vert_{L^{6}(\Cf_{h})} 
\Vert h^{-\frac{1}{2}} \lb \Bc\rb_{T}\Vert_{L^{2}(\Cf_{h})}.
\end{align*}
Due to discrete trace inequalities, we get that 
\begin{align*}
\Sigma_{F\in \Cf_{h}^{I}} \langle \kappa \Lb \Bv\times \Bd\Rb, \lb \Bc\rb_{T}\rangle_{F}  
\leq C \kappa \Vert \Bd\Vert_{L^{3}(\Omega)} \Vert \Bv\Vert_{L^{6}(\Omega)} 
\Vert h^{-\frac{1}{2}} \lb \Bc\rb_{T}\Vert_{L^{2}(\Cf_{h})}.
\end{align*}
Thus 
\begin{align*}
C_{h} (\Bd; \Bv, \Bc) \leq C \kappa \Vert d\Vert_{L^{3}(\Omega)} \Vert \Bv\Vert_{L^{6}(\Omega)} \Vert \Bc\Vert_{C}. 
\end{align*}
By (\ref{assumption_div_free}), Theorem~\ref{theorem_stability} and \cite[Theorem~$5.3$]{DiPietroErn}, we can conclude that 
the proof is complete. 
\end{proof}

We denote by $\boldsymbol{\Pi}_{N}$ the N{\'e}d{\'e}lec projection onto $H(\text{curl},\Omega)\cap \BC_{h}$. 
Let $\tilde{\sigma}>\frac{1}{2}$. We define two projections $\boldsymbol{\Pi}_{C}$ and $\boldsymbol{\Pi}_{C^{I}}$ from 
$H^{\tilde{\sigma}}(\text{curl}, \Omega):= \{ \Bc \in H^{\tilde{\sigma}}(\Omega; \mathbb{R}^{3}): 
\nabla\times \Bc \in H^{\tilde{\sigma}}(\Omega; \mathbb{R}^{3}) \}$ 
to $H(\text{curl},\Omega)\cap \BC_{h}$ by 
\begin{align}
\label{def_proj_C}
\boldsymbol{\Pi}_{C} \Bc := \boldsymbol{\Pi}_{N} \Bc - \nabla \phi_{h},
\quad \forall \Bc \in H^{\tilde{\sigma}}(\text{curl}, \Omega), \\
\nonumber
\boldsymbol{\Pi}_{C^{I}} \Bc := \boldsymbol{\Pi}_{N} \Bc - \nabla \tilde{\phi}_{h},
\quad \forall \Bc \in H^{\tilde{\sigma}}(\text{curl}, \Omega),
\end{align}
where $\phi_{h}\in H_{0}^{1}(\Omega)\cap P_{k+1}(\Ct_{h})$ 
and $\tilde{\phi}_{h}\in H^{1}(\Omega)\cap L_{0}^{2}(\Omega)\cap P_{k+1}(\Ct_{h})$ satisfying
\begin{align*}
(\nabla \phi_{h}, \nabla s)_{\Omega} = & (\boldsymbol{\Pi}_{N} \Bc, \nabla s)_{\Omega},
\quad \forall s \in H_{0}^{1}(\Omega)\cap P_{k+1}(\Ct_{h}), \\
(\nabla \tilde{\phi}_{h}, \nabla s)_{\Omega} = & (\boldsymbol{\Pi}_{N} \Bc, \nabla s)_{\Omega},
\quad \forall s \in H^{1}(\Omega)\cap P_{k+1}(\Ct_{h}).
\end{align*}

\begin{lemma}
\label{lemma_proj_C}
Let $\boldsymbol{\Pi}_{C}, \boldsymbol{\Pi}_{C^{I}}$ be the projections defined in (\ref{def_proj_C}). Then for any $\Bc\in H^{\tilde{\sigma}}(\text{curl}, \Omega)$ 
satisfying $\nabla\cdot \Bc = 0$ in $\Omega$, we get that  
\begin{subequations}
\label{proj_C_props}
\begin{align}
\label{proj_C_prop1}
& (\boldsymbol{\Pi}_{C} \Bc, \nabla s)_{\Omega} = (\Bc, \nabla s)_{\Omega},
\quad \forall s \in H_{0}^{1}(\Omega) \cap S_h, \\
\label{proj_C_prop2}
& \Vert \boldsymbol{\Pi}_{C} \Bc - \Bc\Vert_{L^{2}(\Omega)} \leq \Vert \boldsymbol{\Pi}_{N} \Bc - \Bc \Vert_{L^{2}(\Omega)}, \\
\label{proj_C_prop3}
& \Vert \nabla\times (\boldsymbol{\Pi}_{C} \Bc - \Bc)\Vert_{L^{2}(\Omega)} 
= \Vert \nabla\times (\boldsymbol{\Pi}_{N} \Bc - \Bc ) \Vert_{L^{2}(\Omega)}, \\
\label{proj_C_prop4}
& (\boldsymbol{\Pi}_{C^{I}} \Bc, \nabla s)_{\Omega} = 0,\quad \forall s \in H^{1}(\Omega) \cap S_h, \\
\label{proj_C_prop5}
& \Vert \boldsymbol{\Pi}_{C^{I}} \Bc - \Bc\Vert_{L^{2}(\Omega)} 
\leq \Vert \boldsymbol{\Pi}_{N} \Bc - \Bc \Vert_{L^{2}(\Omega)}, \\
\label{proj_C_prop6}
& \Vert \nabla\times (\boldsymbol{\Pi}_{C^{I}} \Bc - \Bc)\Vert_{L^{2}(\Omega)} 
= \Vert \nabla\times (\boldsymbol{\Pi}_{N} \Bc - \Bc ) \Vert_{L^{2}(\Omega)}.
\end{align}
\end{subequations}
\end{lemma}

\begin{proof}
It is easy to check that (\ref{proj_C_prop1}), (\ref{proj_C_prop3}), (\ref{proj_C_prop4}) 
and (\ref{proj_C_prop6}) hold. 

Since $\nabla\cdot \Bc = 0$, then $(\Bc, \nabla \phi_{h}) = 0$ where $\phi_{h}$ is introduced in (\ref{def_proj_C}). 
Then we get that $(\boldsymbol{\Pi}_{C} \Bc - \Bc, \nabla \phi_{h}) = 0$. Thus 
\begin{align*}
\Vert \boldsymbol{\Pi}_{N}\Bc - \Bc\Vert_{L^{2}(\Omega)}^{2} = \Vert \nabla \phi_{h} \Vert_{L^{2}(\Omega)}^{2} 
+ \Vert \boldsymbol{\Pi}_{C}\Bc - \Bc\Vert_{L^{2}(\Omega)}^{2},
\end{align*}
which implies (\ref{proj_C_prop2}) immediately. The proof of (\ref{proj_C_prop5}) 
is similar to that of (\ref{proj_C_prop2}). 
\end{proof}

\section{Proof of $L^{3}$-norm control of discrete magnetic field: Homogeneous tangential components boundary condition}
\label{proof_L3}
In this section, we prove Theorem~\ref{theorem_stability} which provides 
$L^{3}$-norm control of discrete magnetic field $\Bb_{h}$.
We begin by the following result: 

\begin{lemma}
\label{lemma_conforming_stability}
There is a positive constant $C$ such that for any $\tilde{\Bb}_{h} \in \BC_h\cap H_{0}(\text{curl}, \Omega)$, if 
\begin{align}
\label{weak_div_free_conforming}
(\tilde{\Bb}_{h}, \nabla s)_{\Omega}=0 \qquad \forall s \in H_{0}^{1}(\Omega)\cap P_{k+1}(\Ct_{h}),
\end{align}
then 
\begin{align*}
\Vert \tilde{\Bb}_{h}\Vert_{L^{3}(\Omega)} \leq C \Vert \nabla \times \tilde{\Bb}_{h} \Vert_{L^{2}(\Omega)}. 
\end{align*}
\end{lemma}
\begin{proof}
{\red{ We first recall the inverse inequality for discrete functions {\blue{\cite[Lemma 4.5.3~]{BrennerScott08}}}. For any $K \in \Ct_h$ and $v \in P_k(K)$, we have 
\begin{equation}\label{scaling_ineq}
\|v\|_{L^p(K)} \le C h_K^{\frac{3}{p} - \frac{3}{q}} \|v\|_{L^q(K)},
\end{equation}
for all $1 \le p,q \le \infty$, here $C$ is independent of $h_K$}}. 

We define {\red{the Hodge mapping $\Cp \tilde{\Bb}_{h} \in H_0({\rm curl},\Omega)$ \cite[(4.8)]{Hiptmair02} satisfying}}
\begin{equation}
\label{hodge_Pvh}
\nabla \times \Cp \tilde{\Bb}_{h}  = \nabla \times  \tilde{\Bb}_{h} , \quad \nabla \cdot \Cp \tilde{\Bb}_{h} = 0  \qquad \quad  \ \   {\rm in} \  \Omega.
\end{equation}

According to \cite[Theorem~$4.1$]{Hiptmair02}, there is $\delta\in (0,\frac{1}{2}]$ such that 
\begin{align*}
\Vert \Cp \tilde{\Bb}_{h}\Vert_{H^{\frac12+\delta}(\Omega)} \leq C \Vert \nabla\times \tilde{\Bb}_{h} \Vert_{L^{2}(\Omega)}. 
\end{align*}
According to \cite[Lemma~$4.5$]{Hiptmair02}, 
\begin{align*}
\Vert \tilde{\Bb}_{h} - \Cp \tilde{\Bb}_{h}\Vert_{L^{2}(\Omega)} 
\leq C h^{\frac12+\delta} \Vert \nabla \times \tilde{\Bb}_{h} \Vert_{L^{2}(\Omega)}. 
\end{align*}

We denote by $\boldsymbol{\Pi}_{V}$ the $L^{2}$-orthogonal projection onto $\BV_h = P_{k}(\Ct_{h};\mathbb{R}^{3})$. 
Since $\tilde{\Bb}_{h} \in P_{k}(\Ct_{h};\mathbb{R}^{3})$, then $\boldsymbol{\Pi}_{h} \tilde{\Bb}_{h} = \tilde{\Bb}_{h}$. So, we have that 
\begin{align*}
\Vert \tilde{\Bb}_{h}\Vert_{L^{3}(\Omega)} = \Vert \boldsymbol{\Pi}_{V} \tilde{\Bb}_{h}\Vert_{L^{3}(\Omega)} 
\leq \Vert \boldsymbol{\Pi}_{V} (\tilde{\Bb}_{h} - \Cp \tilde{\Bb}_{h})\Vert_{L^{3}(\Omega)} 
+ \Vert \boldsymbol{\Pi}_{V} (\Cp \tilde{\Bb}_{h})\Vert_{L^{3}(\Omega)}. 
\end{align*}
By \eqref{scaling_ineq}, we have that 
\begin{align*}
 \Vert \boldsymbol{\Pi}_{V} (\tilde{\Bb}_{h} - \Cp \tilde{\Bb}_{h})\Vert_{L^{3}(\Omega)}  
\leq & C h^{-\frac12} \Vert \boldsymbol{\Pi}_{V} (\tilde{\Bb}_{h} - \Cp \tilde{\Bb}_{h})\Vert_{L^{2}(\Omega)} 
\leq  C h^{-\frac12}\Vert \tilde{\Bb}_{h} - \Cp \tilde{\Bb}_{h} \Vert_{L^{2}(\Omega)} \\
\leq & C h^{\delta} \Vert \nabla \times \tilde{\Bb}_{h} \Vert_{L^{2}(\Omega)}. 
\end{align*}
{{On each $K \in \Ct_h$ for any $\Bv \in L^3(K)$, By \eqref{scaling_ineq} we have
\begin{equation*}
\Vert \boldsymbol{\Pi}_{V} \Bv\Vert_{L^{3}(K)} \leq C h_K^{-\frac12} \Vert \boldsymbol{\Pi}_{V} \Bv\Vert_{L^{2}(K)} \le C  h_K^{-\frac12} \Vert \Bv \Vert_{L^{2}(K)} \le C \Vert \Bv\Vert_{L^{3}(K)}, 
\end{equation*}
the last step in the above estimate is due to the Sobolev inequality. Consequently, we have
\begin{equation}\label{L3_estimate}
\|\boldsymbol{\Pi}_V \Bv\|_{L^3(\Omega)} \le C \|\Bv\|_{L^3(\Omega)} \quad \forall \Bv \in L^3(\Omega).
\end{equation}
}}
Thus we have that 
\begin{align*}
\Vert \boldsymbol{\Pi}_{V} (\Cp \tilde{\Bb}_{h})\Vert_{L^{3}(\Omega)} 
\leq C \Vert \Cp \tilde{\Bb}_{h} \Vert_{L^{3}(\Omega)} \leq C \Vert \Cp \tilde{\Bb}_{h} \Vert_{H^{\frac12+\delta}(\Omega)} 
\leq C \Vert \nabla \times \tilde{\Bb}_{h} \Vert_{L^{2}(\Omega)}.
\end{align*}
{{Combining above estimates we have that}}
\begin{align*}
\Vert \tilde{\Bb}_{h}\Vert_{L^{3}(\Omega)} \leq C \Vert \nabla \times \tilde{\Bb}_{h} \Vert_{L^{2}(\Omega)}. 
\end{align*}
\end{proof}
{\blue{Next we present an intermediate result:}} 
{\blue{
\begin{lemma}
\label{lemma_div0_stability}
There is a positive constant $C$ such that for any $\Bb_h\in \BC_h$ with $\nabla_h \cdot \Bb_h = 0$, we have  
\begin{align*}
\|\Bb_h \|_{\BL^3(\Omega)} \leq C \big( \|  h^{-\frac12}\llbracket\Bb_h \rrbracket_{T} \|_{L^2(\Cf_h)} 
+ \|\nabla \times \Bb_h \|_{L^2(\mathcal{T}_{h})}\big).
\end{align*}
\end{lemma}
}}
\begin{proof}
Due to \cite[Proposition~$4.5$]{Houston}, there is $\tilde{\Bb}_h \in \BC_h\cap H_{0}(\text{curl}, \Omega)$ such that
\begin{subequations}
\label{interpolation_ineqs}
\begin{align}
\label{interpolation_ineq1}
\|\Bb_h - \tilde{\Bb}_h\|_{L^2(\Omega)}  \leq & C{{\| h^{\frac12} \llbracket\Bb_h \rrbracket_T \|_{L^2(\Cf_h)}}}, \\
\label{interpolation_ineq2}
\|\nabla\times (\Bb_h - \tilde{\Bb}_h)\|_{L^2(\Ct_{h})}  \leq & C {{\| h^{-\frac12} \llbracket\Bb_h  \rrbracket_T \|_{L^2(\Cf_h)}}}.
\end{align}
\end{subequations}

According to \eqref{scaling_ineq} and (\ref{interpolation_ineq1}), we have that 
\begin{align}
\label{stab_ineq1}
\Vert \Bb_{h} -  \tilde{\Bb}_h \Vert_{L^{3}(\Omega)} \leq C {{\| h^{-\frac12} \llbracket\Bb_h \rrbracket_T \|_{L^2(\Cf_h)}}}. 
\end{align}

We define $\sigma_{h}\in H_{0}^{1}(\Omega)\cap P_{k+1}(\Ct_{h})$ by 
\begin{align*}
(\nabla \sigma_{h}, \nabla s)_{\Omega} = (\tilde{\Bb}_h, \nabla s)_{\Omega}\qquad \forall s \in H_{0}^{1}(\Omega)
\cap  P_{k+1}(\Ct_{h}).
\end{align*}
Due to {\blue{\eqref{discrete_div}, the assumption $\nabla_h \cdot \Bb_h = 0$}} and (\ref{interpolation_ineq1}), 
\begin{align*}
& (\nabla \sigma_{h}, \nabla \sigma_{h})_{\Omega} = (\tilde{\Bb}_h, \nabla \sigma_{h})_{\Omega} 
= (\tilde{\Bb}_h - \Bb_{h}, \nabla \sigma_{h})_{\Omega}\qquad \\
\leq & \Vert \Bb_h - \tilde{\Bb}_h \Vert_{L^{2}(\Omega)} \Vert \nabla \sigma_{h}\Vert_{L^{2}(\Omega)}
\leq C {{\| h^{\frac12} \llbracket\Bb_h \rrbracket_T \|_{L^2(\Cf_h)}}}\Vert \nabla \sigma_{h}\Vert_{L^{2}(\Omega)}.
\end{align*}
By \eqref{scaling_ineq}, we have that 
\begin{align}
\label{stab_ineq2}
\Vert \nabla \sigma_{h}\Vert_{L^{3}(\Omega)} \leq C h^{-\frac12} \Vert \nabla \sigma_{h}\Vert_{L^{2}(\Omega)} 
\leq C  {{\| h^{-\frac12} \llbracket\Bb_h  \rrbracket_T \|_{L^2(\Cf_h)}}}.
\end{align}
{{Next, by the definition of $\widetilde{\Bb_h}, \sigma_h$ we notice that}} 
\begin{align*}
& \tilde{\Bb}_h - \nabla \sigma_{h} \in \BC_h\cap H_{0}(\text{curl}, \Omega), \\
& (\tilde{\Bb}_h - \nabla \sigma_{h}, \nabla s)_{\Omega} = 0 \quad \forall s\in H_{0}^{1}(\Omega)\cap P_{k+1}(\Ct_{h}).
\end{align*}
Applying Lemma~\ref{lemma_conforming_stability} to $\tilde{\Bb}_h - \nabla \sigma_{h}$, we have that 
\begin{align}
\label{stab_ineq3}
& \Vert \tilde{\Bb}_h - \nabla \sigma_{h}\Vert_{L^{3}(\Omega)} \leq C 
\Vert \nabla\times (\tilde{\Bb}_h - \nabla \sigma_{h}) \Vert_{L^{2}(\Omega)} 
= C \Vert \nabla\times \tilde{\Bb}_h \Vert_{L^{2}(\Omega)}\\
\nonumber
\leq & C \big( \Vert \nabla\times (\tilde{\Bb}_h - \Bb_{h} ) \Vert_{L^{2}(\Ct_{h})} 
+ \Vert \nabla\times \Bb_{h}\Vert_{L^{2}(\Ct_{h})}\big) \\
\nonumber 
\leq &  {{C ( \|  h^{-\frac12}\llbracket\Bb_h\rrbracket_T \|_{L^2(\Cf_h)}
+ \|\nabla \times \Bb_h \|_{L^2(\mathcal{T}_{h})}).}}
\end{align}
The last inequality above is due to (\ref{interpolation_ineq2}). {{Finally the proof is complete by combining \eqref{stab_ineq1}, \eqref{stab_ineq2}, \eqref{stab_ineq3},}}
\begin{align*}
\|\Bb_h \|_{\BL^3(\Omega)} \leq C \big( {{\|  h^{-\frac12}\llbracket\Bb_h \rrbracket_T \|_{L^2(\Cf_h)} }}
+ \|\nabla \times \Bb_h \|_{L^2(\mathcal{T}_{h})}\big).
\end{align*}
\end{proof}

Now we are ready to prove Theorem~\ref{theorem_stability}. 

\begin{proof}(of Theorem \ref{theorem_stability})
Given $\Bb_h \in C_h$, notice that $\nabla_h \cdot \Bb_h \in H^1_0(\Omega) \cap S_h \in L^2(\Omega)$. We consider the auxiliary Poisson equation: Find $\phi$ satisfying:
\begin{subequations}\label{poisson}
\begin{align}
-\Delta \phi &= \nabla_h \cdot \Bb_h \quad \text{in $\Omega$}, \\
\phi &= 0 \quad \text{on $\partial \Omega$}
\end{align}
\end{subequations}
On a polygonal domain $\Omega$, we have the regularity result \cite{Hua1989}: there exists $\delta_0 > 0$ such that
\begin{equation}\label{poisson_reg}
\|\phi\|_{H^{\frac32 + \delta_0}(\Omega)} \le C \|\nabla_h \cdot \Bb_h\|_{L^2(\Omega)}.
\end{equation}
Let $\phi_h$ be the numerical solution of \eqref{poisson} in the Lagrange space $H^1_0(\Omega) \cap S_h$. i.e. It solves the system:
\begin{equation}\label{CG_phi}
\bint{\nabla \phi_h}{\nabla w} = \bint{\nabla_h \cdot \Bb_h}{w} \quad \text{for all $w \in H^1_0(\Omega) \cap S_h$}.
\end{equation}
This implies that that $\Bb_h - \nabla \phi_h \in C_h$ and $\nabla_h \cdot (\Bb_h - \nabla \phi_h) = 0$. By Lemma \ref{lemma_div0_stability}, we have
\begin{align}\label{bridge_L3}
\|\Bb_h - \nabla \phi_h\|_{L^3(\Omega)} &\le C ( \| \llbracket \Bb_h - \nabla \phi_h \rrbracket_T\|_{L^2(\Cf_h)} + \|\nabla \times (\Bb_h - \phi_h)\|_{L^2(\Ct_h)}) \\
\nonumber
& = C ( \| \llbracket \Bb_h \rrbracket_T\|_{L^2(\Cf_h)} + \|\nabla \times \Bb_h\|_{L^2(\Ct_h)}).
\end{align}
The last step is due the fact that $\nabla \phi_h \in C_h \cap H_0(\text{curl}; \Omega)$. Next we present a bound for $\|\nabla \phi_h\|_{L^3(\Omega)}$.
To this end, by triangle inequality, we have
\begin{align*}
\|\nabla \phi_h\|_{L^3(\Omega)} & = \|\Pi_V (\nabla \phi_h)\|_{L^3(\Omega)}  \le \| \Pi_V(\nabla \phi)\|_{L^3(\Omega)} + \|\Pi_V (\nabla (\phi - \phi_h))\|_{L^3(\Omega)},  \\
 \intertext{by \eqref{scaling_ineq} and \eqref{L3_estimate}, we further have}
 \|\nabla \phi_h\|_{L^3(\Omega)} & \le C (\| \nabla \phi \|_{L^3(\Omega)} + h^{-\frac12} \|\Pi_V (\nabla (\phi - \phi_h))\|_{L^2(\Omega)}) \\
& \le C (\| \nabla \phi \|_{L^3(\Omega)} + h^{-\frac12} \|\nabla (\phi - \phi_h)\|_{L^2(\Omega)}) \\
 & \le C (\|\phi\|_{H^{\frac32 + \delta_0}(\Omega)} + h^{\delta_0} \|\phi\|_{H^{\frac32 + \delta_0}(\Omega)}) \\
 &\le C \|\phi\|_{H^{\frac32 + \delta_0}(\Omega)} \le C \|\nabla_h \cdot \Bb_h\|_{L^2(\Omega)}.
\end{align*}
Here we used the approximation property of $\phi_h$ and the regularity property \eqref{poisson_reg}. Finally, the proof is complete by combining above estimate with \eqref{bridge_L3}.
\end{proof}

\section{Proof of $L^{3}$-norm control of discrete magnetic field: homogeneous normal component boundary condition} 
\label{sec_l3_type2}

In this section, we will present detailed proof for Theorem \ref{theorem_stability2} which plays a crucial role for the analysis for the second type of boundary condition. 
{\blue{Analogous to the proof of Theorem \ref{theorem_stability}. We first derive the estimate for functions in $\BC_h$ which are discretely divergence free:
\begin{lemma}
\label{L3_ineq_N}
%\label{lemma_div0_stability_N}
There is a positive constant $C$ such that for any $\Bb_h\in \BC_h$ with $\nabla^N_h \cdot \Bb_h = 0$, we have  
\begin{align*}
\|\Bb_h \|_{\BL^3(\Omega)} \leq C \big( \|  h^{-\frac12}\llbracket\Bb_h \rrbracket_{T} \|_{L^2(\Cf^I_h)} 
+ \|\nabla \times \Bb_h \|_{L^2(\mathcal{T}_{h})}\big).
\end{align*}
\end{lemma}
}}

Similar as for the Lemma \ref{lemma_div0_stability}, 
we begin by the following Lemma~\ref{lemma_conforming_stability2}, which is similar to \cite[Lemma~$3.6$]{BLi1}. 
{\red{We}} provide the proof of Lemma~\ref{lemma_conforming_stability2} 
in Appendix~\ref{appendix1}.

\begin{lemma}
\label{lemma_conforming_stability2}
There is a positive constant $C$ such that for any $\tilde{\Bb}_{h} \in \BC_h\cap H(\text{curl}, \Omega)$, if 
\begin{align}
\label{weak_div_free_conforming2}
(\tilde{\Bb}_{h}, \nabla s)_{\Omega}=0 \qquad \forall s \in H^{1}(\Omega)\cap P_{k+1}(\Ct_{h}),
\end{align}
then 
\begin{align*}
\Vert \tilde{\Bb}_{h}\Vert_{L^{3}(\Omega)} \leq C \Vert \nabla \times \tilde{\Bb}_{h} \Vert_{L^{2}(\Omega)}. 
\end{align*}
\end{lemma}

We also need Lemma~\ref{lemma_curl_interpolation}, which is similar to \cite[Proposition~$4.5$]{Houston}. 
The proof is presented in Appendix B.

\begin{lemma}
\label{lemma_curl_interpolation}
There is a positive constant $C$ such that for any $\Bb_{h} \in \BC_{h}$, there is 
$\tilde{\Bb}_{h}\in H(\text{curl},\Omega)\cap \BC_{h}$ satisfying 
\begin{align*}
\Vert \Bb_{h} - \tilde{\Bb}_{h} \Vert_{L^{2}(\Omega)} \leq C \|  h^{\frac12}\llbracket\Bb_h \rrbracket_{T} \|_{L^2(\Cf_h^{I})}. 
\end{align*}
\end{lemma}

With above two lemmas, the proof of Lemma \ref{L3_ineq_N} is almost the same as that of 
Lemma~\ref{lemma_div0_stability}. We only need to use 
Lemma~\ref{lemma_conforming_stability2} and Lemma~\ref{lemma_curl_interpolation} to replace 
Lemma~\ref{lemma_conforming_stability} and \cite[Proposition~$4.5$]{Houston} in the proof of 
Theorem~\ref{theorem_stability}. {\blue{Finally, the proof of Theorem \ref{theorem_stability2} can be obtained by replacing 
the auxillary Poisson problem \eqref{poisson} with homogeneous Neumann boundary condition 
\begin{subequations}
\label{poisson_neumann}
\begin{align}
-\Delta \phi &= \nabla_{h}^{N} \cdot \Bb_h \quad \text{in $\Omega$}, \\
\nabla\phi \cdot \Bn &= 0 \quad \text{on $\partial \Omega$}
\end{align}
\end{subequations}
in the proofs in the Theorem \ref{theorem_stability}.}} 
According to \cite[Corollary $3.9$]{Dauge1}, there exists $\delta_{0}>0$ such that 
\begin{equation}\label{poisson_reg_neumann}
\|\phi\|_{H^{\frac32 + \delta_0}(\Omega)} \le C \|\nabla_{h}^{N} \cdot \Bb_h\|_{L^2(\Omega)}.
\end{equation}
So, the proof of Theorem \ref{theorem_stability2} can be carried out in the same way as 
that of Theorem \ref{theorem_stability}.

\section{Proof of the existence, uniqueness and boundedness of the approximate solution}

In this section, we prove Theorem~\ref{theorem_wellposed} on the existence, uniqueness and boundedness of 
the approximate solution of the DG method. We skip the proof of Theorem~\ref{theorem_wellposed2} since 
it is almost the same as that of Theorem~\ref{theorem_wellposed} (we only need to use Theorem~\ref{theorem_stability2} 
to replace Theorem~\ref{theorem_stability}).  
The counterpart for the exact solution was provided in \cite{GreifLi2010}. 
We first define a mapping $\mathcal{F}$ on 
\begin{align}
\label{def_Zh}
\BZ_{h} := \{ (\Bv, \Bc)\in \BV_{h} \times \BC_{h}:  B_{h}(\Bv, q) = D_{h} (\Bc, s) = 0, 
\quad \forall (q, s) \in Q_{h}\times (H_{0}^{1}(\Omega) \cap S_{h})  \},
\end{align}
We will show that the mapping is a contraction on a subset of $\BZ_h$ and apply the Brower fixed point theorem for 
the existence of the solution. Finally the uniqueness follows easily.

\subsection*{Step 1: Definition of the operator $\mathcal{F}$} 
We start by defining $\mathcal{F}$. For $(\Bbeta_{h}, \Bd_{h})\in \BZ_{h}$, 
we take $\mathcal{F}(\Bbeta_{h}, \Bd_{h})$ to be the component $(\Bu_{h}, \Bb_{h})$ of the solution 
$(\Bu_{h}, \Bb_{h}, p_{h}, r_{h})\in \BV_{h}\times \BC_{h}\times Q_{h}\times S_{h}$ of 
 \begin{align}
 \label{DG_linear_mhd}
 A_{h}(\Bu_{h}, \Bv) + O_{h}(\mathbb{P}(\Bbeta_{h}, \Lb \Bbeta_{h} \Rb ); \Bu_{h}, \Bv) 
 + C_{h}(\Bd_{h}; \Bv, \Bb_{h}) + B_{h}(\Bv, p_{h}) & = (\Bf, \Bv)_{\Omega}, \\
 \nonumber
M_{h}(\Bb_{h}, \Bc) - C_{h}(\Bd_{h}; \Bu_{h}, \Bc) + D_{h} (\Bc, r_{h}) & = (\Bg, \Bc)_{\Omega}, \\
\nonumber
B_{h} (\Bu_{h}, q) &= 0,\\
\nonumber
D_{h}(\Bb_{h}, s) - J_{h}(r_{h}, s) &= 0,
 \end{align}
for all $(\Bv, \Bc, q, s)\in \BV_{h}\times \BC_{h}\times Q_{h}\times S_{h}$. The above system is the original mixed DG scheme for the linearized MHD equations in \cite{HSW09}, we refer \cite{HSW09} for the existence and uniqueness of the solutions.

\subsection*{Step 2: Proof of the upper bound of the approximate solution} 
Next, we establish the boundedness result of the mapping $\mathcal{F}$. 
We take $\Bv = \Bu_{h}$, $\Bc = \Bb_{h}$, $q = - p_{h}$ and $s = - r_{h}$ in (\ref{DG_linear_mhd}). 
We have that 
\begin{align*}
A_{h} (\Bu_{h}, \Bu_{h}) + O_{h}(\mathbb{P}(\Bbeta_{h}, \Lb \Bbeta_{h} \Rb); \Bu_{h}, \Bu_{h}) 
+ M_{h}(\Bb_{h}, \Bb_{h}) + J_{h} (r_{h}, r_{h}) = (\Bf, \Bu_{h})_{\Omega} + (\Bg, \Bb_{h})_{\Omega}.  
\end{align*}
By Lemma \ref{coarsivity} and Lemma \ref{lemma_proj},  
\begin{align*}
\nu \Vert \Bu_{h}\Vert_{V}^{2} + \kappa \nu_{m} \big( \Vert \nabla \times \Bb_{h}\Vert_{\mathcal{T}_h}^{2}
+   \Vert h^{-\frac12} \lb \Bb_{h}\rb_{T}\Vert_{L^{2}(\Cf_h)}^{2}\big) 
\leq  C \big( (\Bf, \Bu_{h})_{\Omega} + (\Bg, \Bb_{h})_{\Omega} \big). 
\end{align*}
We notice that $D_{h}(\Bb_{h}, s) = 0$ for any $s \in S_{h}$. Then by Theorem~\ref{theorem_stability}, 
\begin{align*}
\nu \Vert \Bu_{h}\Vert_{V}^{2} + \kappa \nu_{m} \big(\Vert \Bb_{h}\Vert_{L^{3}(\Omega)}^{2} + \Vert \Bb_{h}\Vert_{C}^{2} \big)
\leq  C \big( (\Bf, \Bu_{h})_{\Omega} + (\Bg, \Bb_{h})_{\Omega} \big). 
\end{align*}
As a sequence, we get that 
\begin{align*}
\nu^{\frac{1}{2}} \Vert \Bu_{h}\Vert_{V} + \kappa^{\frac{1}{2}} \nu_{m}^{\frac{1}{2}}
\big(\Vert \Bb_{h}\Vert_{L^{3}(\Omega)}+ \Vert \Bb_{h}\Vert_{C}\big)
\leq C \big( \nu^{-\frac{1}{2}} \Vert \Bf\Vert_{L^{2}(\Omega)}
 + \kappa^{-\frac{1}{2}}\nu_{m}^{-\frac{1}{2}} \Vert \Bg\Vert_{L^{2}(\Omega)} \big).
\end{align*}
This proves the stability result (\ref{energy_norm_bound}) of Theorem~\ref{theorem_wellposed}. Additionally, it also shows that $\mathcal{F}$ maps $\BK_{h}$ into $\BK_{h}$, where 
\begin{align*}
\BK_{h} := \{ (\Bv, \Bc)\in \BZ_{h}:  \nu^{\frac{1}{2}} \Vert \Bu_{h}\Vert_{V} 
+\kappa^{\frac{1}{2}} \nu_{m}^{\frac{1}{2}} \Vert \Bb_{h}\Vert_{C}
\leq C \big( \nu^{-\frac{1}{2}} \Vert \Bf\Vert_{L^{2}(\Omega)}
 + \kappa^{-\frac{1}{2}}\nu_{m}^{-\frac{1}{2}} \Vert \Bg\Vert_{L^{2}(\Omega)} \big)\}.
\end{align*}

\subsection*{Step 3: the operator $\mathcal{F}$ is a contraction on $\BK_{h}$} 
To prove this, let $(\Bbeta_{1}, \Bd_{1}), (\Bbeta_{2}, \Bd_{2}) \in \BK_{h}$ and set 
$(\Bu_{1}, \Bb_{1}) := \mathcal{F}(\Bbeta_{1}, \Bd_{1})$ and 
$(\Bu_{2}, \Bb_{2}) := \mathcal{F}(\Bbeta_{2}, \Bd_{2})$. 
By definition, there exist $p_{1}, p_{2}\in Q_{h}$, $r_{1}, r_{2} \in S_{h}$ such that 
both $(\Bu_{1}, \Bb_{1}, p_{1}, r_{1})$ and $(\Bu_{2}, \Bb_{2}, p_{2}, r_{2})$ satisfy (\ref{DG_linear_mhd}). 

We set $\delta_{\Bu}:= \Bu_{1} - \Bu_{2}$, $\delta_{\Bb}:= \Bb_{1} - \Bb_{2}$, $\delta_{p}:= p_{1} - p_{2}$ 
and $\delta_{r} := r_{1} - r_{2}$. We get that 
\begin{align*}
& A_{h}(\delta_{\Bu}, \Bv) + B_{h}(\Bv, \delta_{p}) + M_{h}(\delta_{\Bb}, \Bv) + D_{h}(\Bc, \delta_{r}) 
- B_{h}(\delta_{\Bu}, q) - D_{h}(\delta_{\Bb}, s) + J_{h} (\delta_{r}, s)\\ 
&\qquad + O_{h}(\mathbb{P}(\Bbeta_{1}, \Lb \Bbeta_{1}\Rb); \Bu_{1}, \Bv)
 - O_{h}(\mathbb{P}(\Bbeta_{2}, \Lb \Bbeta_{2}\Rb); \Bu_{2}, \Bv) \\
 & \qquad + \big( C_{h}(\Bd_{1}; \Bv, \Bb_{1}) - C_{h}(\Bd_{2}; \Bv, \Bb_{2}) \big) 
 - \big( C_{h}(\Bd_{1}; \Bu_{1}, \Bc) - C_{h}(\Bd_{2}; \Bu_{2}, \Bc) \big) = 0
\end{align*} 
for any $(\Bv, \Bc, q, s) \in \BV_{h}\times \BC_{h}\times Q_{h} \times S_{h}$. 
Taking $(\Bv, \Bc, q, s):= (\delta_{\Bu}, \delta_{\Bc}, \delta_{p}, \delta_{r})$, we obtain 
\begin{align*}
& A_{h}(\delta_{\Bu}, \delta_{\Bu}) + M_{h}(\delta_{\Bb}, \delta_{\Bb}) + J_{h}(\delta_{r}, \delta_{r})\\
= & O_{h}(\mathbb{P}(\Bbeta_{2}, \Lb \Bbeta_{2}\Rb); \Bu_{2}, \delta_{\Bu}) 
- O_{h}(\mathbb{P}(\Bbeta_{1}, \Lb \Bbeta_{1}\Rb); \Bu_{1}, \delta_{\Bu})  \\
&\quad + \big( C_{h}(\Bd_{2}; \delta_{\Bu}, \Bb_{2}) -  C_{h}(\Bd_{1}; \delta_{\Bu}, \Bb_{1})\big) 
+ \big(C_{h}(\Bd_{1}; \Bu_{1}, \delta_{\Bb}) - C_{h}(\Bd_{2}; \Bu_{2}, \delta_{\Bb}) \big)\\
= & O_{h}(\mathbb{P}(\Bbeta_{2}, \Lb \Bbeta_{2}\Rb); \Bu_{2}, \delta_{\Bu}) 
- O_{h}(\mathbb{P}(\Bbeta_{1}, \Lb \Bbeta_{1}\Rb); \Bu_{1}, \delta_{\Bu})  \\
&\quad  - C_{h}(\Bd_{1} - \Bd_{2}; \delta_{\Bu}, \Bb_{2}) 
+ C_{h}(\Bd_{1}-\Bd_{2}; \Bu_{2}, \delta_{\Bb}) := I_{1} + I_{2} + I_{3}.
\end{align*}

It is easy to see that 
\begin{align*}
I_{1} & = - O_{h}(\mathbb{P}(\Bbeta_{2}, \Lb \Bbeta_{2}\Rb); \delta_{\Bu}, \delta_{\Bu}) 
+ O_{h}(\mathbb{P}(\Bbeta_{2}, \Lb \Bbeta_{2}\Rb) - \mathbb{P}(\Bbeta_{1}, \Lb \Bbeta_{1}\Rb); \Bu_{1}, \delta_{\Bu})\\
& \leq O_{h}(\mathbb{P}(\Bbeta_{2}, \Lb \Bbeta_{2}\Rb) - \mathbb{P}(\Bbeta_{1}, \Lb \Bbeta_{1}\Rb); \Bu_{1}, \delta_{\Bu}),
\end{align*}
since $O_{h}(\mathbb{P}(\Bbeta_{2}, \Lb \Bbeta_{2}\Rb); \delta_{\Bu}, \delta_{\Bu}) \geq 0$ due to Lemma \ref{lemma_proj}. 
According to \cite[Theorem~$5.3$]{DiPietroErn} (see also \cite[Proposition~$4.5$]{Karakashian}), 
discrete trace inequality and (\ref{proj_prop3}), we get that
\begin{align}
\label{contraction_ineq1}
I_{1} \leq C \Vert \Bbeta_{1} - \Bbeta_{2}\Vert_{V} \Vert \Bu_{1}\Vert_{V} \Vert \delta_{\Bu}\Vert_{V}.
\end{align}
Since $\Bd_{1} - \Bd_{2}$ satisfies (\ref{assumption_div_free}), by Lemma~\ref{lemma_Oh_cont}, we get that 
\begin{align}
\label{contraction_ineq2} 
I_{2} \leq & C \kappa \Vert \Bd_{1} - \Bd_{2}\Vert_{C} \Vert  \delta_{\Bu}\Vert_{V} \Vert \Bb_{2}\Vert_{C},\\
\label{contraction_ineq3}
I_{3} \leq & C \kappa \Vert \Bd_{1} - \Bd_{2}\Vert_{C} \Vert \Bu_{2}\Vert_{V} \Vert \delta_{\Bb}\Vert_{C}.
\end{align}

On the other hand, by Lemma \ref{coarsivity} we get that 
\begin{align*}
C \big( \nu \Vert \delta_{\Bu}\Vert_{V}^{2} 
+ \kappa \nu_{m} (\Vert \nabla \times \delta_\Bb \Vert_{L^2(\mathcal{T}_h)}^{2}
+   \Vert h^{-\frac12} \lb \delta_\Bb \rb_{T}\Vert_{L^{2}(\Cf_h)}^{2}\big) \leq A_{h}(\delta_{\Bu}, \delta_{\Bu}) 
+ M_{h}(\delta_{\Bb}, \delta_{\Bb}).
\end{align*}
In addition, since $\delta_{\Bb}$ satisfies \eqref{assumption_div_free}, by Theorem \ref{theorem_stability} and a simple scaling argument, we have
\[
\|\delta_{\Bb}\|_{\mathcal{T}_h} \le C \|\delta_{\Bb}\|_{L^3(\Omega)} \le C  (\Vert \nabla \times \delta_\Bb \Vert_{L^2(\mathcal{T}_h)}^{2}
+   \Vert h^{-\frac12} \lb \delta_\Bb \rb_{T}\Vert_{L^{2}(\Cf_h)}^{2}).
\]
The above two inequalities implies that
\begin{align*}
C \big( \nu \Vert \delta_{\Bu}\Vert_{V}^{2} 
+ \kappa \nu_{m} \|\delta_{\Bb}\|^2_C \big) \leq A_{h}(\delta_{\Bu}, \delta_{\Bu}) + M_{h}(\delta_{\Bb}, \delta_{\Bb}).
\end{align*}
Combining the above inequality with \eqref{contraction_ineq1}-\eqref{contraction_ineq3}, we get that 
\begin{align*}
\nu \Vert \delta_{\Bu}\Vert_{V}^{2} + \kappa \nu_{m} \Vert \delta_{\Bb}\Vert_{C}^{2} 
\leq & C \big( (\nu^{-2}\Vert \Bu_{1}\Vert_{V}^{2}) \cdot \nu \Vert \Bbeta_{1} - \Bbeta_{2}\Vert_{V}^{2} \\
&\qquad + (\kappa \nu^{-1} \nu_{m}^{-1} \Vert \Bb_{2}\Vert_{C}^{2} + \nu_{m}^{-2}\Vert \Bu_{2}\Vert_{V}^{2})
\cdot \kappa \nu_{m} \Vert \Bd_{1} - \Bd_{2}\Vert_{C}^{2}\big). 
\end{align*}

By virtue of (\ref{energy_norm_bound}),  $\mathcal{F}$ is a contraction on $\BK_{h}$ if the following quantities 
\begin{align*}
\nu^{-2} \Vert \Bf \Vert_{L^{2}(\Omega)},\quad \nu^{-1}\nu_{m}^{-1} \Vert \Bf\Vert_{L^{2}(\Omega)}, \quad
\nu^{-\frac{3}{2}}\kappa^{-\frac{1}{2}}\nu_{m}^{-\frac{1}{2}} \Vert \Bg\Vert_{L^{2}(\Omega)},\quad 
\nu^{-\frac{1}{2}} \kappa^{-\frac{1}{2}}\nu_{m}^{-\frac{3}{2}} \Vert \Bg\Vert_{L^{2}(\Omega)}
\end{align*}
are all small enough such that 
\begin{align}
\label{contraction_ineq4}
\nu^{-2}\Vert \Bu_{1}\Vert_{V}^{2} \leq \rho, \qquad 
\kappa \nu^{-1} \nu_{m}^{-1} \Vert \Bb_{2}\Vert_{C}^{2} + \nu_{m}^{-2}\Vert \Bu_{2}\Vert_{V}^{2} \leq \rho,
\end{align}
for some constant $\rho \in [0, 1)$. As a consequence, by the Brower's fixed point theorem, $\mathcal{F}$ has a unique fixed point in $\BK_{h}$ which is
a solution of \eqref{DG_mhd}. 

The uniqueness of the solution is trivial since if $\Bu_h, \Bb_h, p_h, r_h$ is a solution of $\eqref{DG_mhd}$, by \eqref{energy_norm_bound} $\Bu_h, \Bb_h$ must be a fixed point of $\mathcal{F}$ in $\BK_{h}$ which is unique.

\section{Proof of the error estimates}\label{estimate_DG}

In this section, we prove the error estimates of Theorem \ref{errors_ub}. We skip the proof of Theorem~\ref{errors_ub_type2} 
since it is almost the same as that of Theorem~\ref{errors_ub} (we only need to use Theorem~\ref{theorem_stability2} 
and the projection $\boldsymbol{\Pi}_{C^{I}}$ defined in (\ref{def_proj_C}) to replace Theorem~\ref{theorem_stability} 
and $\boldsymbol{\Pi}_{C}$). To do that, we proceed in the following steps to 
give estimates of the projection of the approximation errors, 
\begin{align*}
e^{\Bu}:= \boldsymbol{\Pi}^{RT} \Bu - \Bu_{h}, \quad e^{\Bb} := \boldsymbol{\Pi}_{C} \Bb - \Bb_{h},\quad 
e^{p} := \Pi_{Q} p - p_{h},\quad e^{r}:= \Pi_{S} r - r_{h},
\end{align*}
where $\boldsymbol{\Pi}^{RT}$ is the $k$-th order Raviart-Thomas (RT) projection, 
$\boldsymbol{\Pi}_{C}$ is defined in (\ref{def_proj_C}), $\Pi_{Q}$ is the $L^{2}$-orthogonal projection onto $Q_{h}$ 
and $\Pi_{S}$ is the Lagrange interpolation onto $H_{0}^{1}(\Omega)\cap S_{h}$. 
Since $\nabla\cdot \Bu = 0$, then $\nabla\cdot \boldsymbol{\Pi}^{RT} \Bu = 0$ due to the communication between 
$\boldsymbol{\Pi}^{RT}$ and the divergence operator. 
Thus $\boldsymbol{\Pi}^{RT} \Bu \in \BV_{h}$ such that $e^{\Bu}\in \BV_{h}$. 
Furthermore, due to (\ref{proj_C_prop1}) and (\ref{DG_mhd_eq4}), $e^{\Bb}$ satisfies (\ref{assumption_div_free}). 
So we have:
\begin{align}
\label{bound_eb}
\Vert e^{\Bb}\Vert_{L^{3}(\Omega)} \leq C \big( \Vert \nabla\times e^{\Bb} \Vert_{L^{2}(\Ct_{h})} 
+ \Vert h^{-\frac{1}{2}} \lb e^{\Bb}\rb_{T} \Vert_{L^{2}(\Cf_{h})} \big) \le C \|e^{\Bb}\|_C. 
\end{align}

\subsection{Estimates for $\Bu - \Bu_h$, $\Bb - \Bb_h$.}

\subsection*{Step 1: The error equations}
We  start our error analysis by obtaining the equations satisfied by the
projections of the errors.

\begin{lemma}
\label{lemma_error_eqs}
The projection of the error $(e^{\Bu}, e^{\Bb}, e^{p}, e^{r})$ satisfies 
\begin{subequations}
\label{error_eqs}
\begin{align}
\label{error_eq1}
A_{h}(e^{\Bu}, \Bv) +B_{h}(\Bv, e^{p}) = & A_{h} (\boldsymbol{\Pi}^{RT}\Bu - \Bu, \Bv) 
+ \Sigma_{F\in \Cf_{h}} \langle \Lb \Pi_{Q}p - p\Rb, \lb \Bv \rb_{N}\rangle_{F}\\
\nonumber
& \quad +O_{h}(\mathbb{P}(\Bu_{h}, \Lb \Bu_{h}\Rb); \Bu_{h}, \Bv) - O_{h}(\Bu; \Bu, \Bv) \\
\nonumber
&\quad + C_{h}(\Bb_{h}; \Bv, \Bb_{h}) - C_{h}(\Bb; \Bv, \Bb),\\
\label{error_eq2}
M_{h}(e^{\Bb}, \Bc) + D_{h}(\Bc, e^{r}) = & M_{h}(\boldsymbol{\Pi}_{C}\Bb - \Bb, \Bc) 
+ (\Bc, \nabla (\Pi_{S}r -r) )_{\Omega}\\
\nonumber 
&\quad +C_{h}(\Bb; \Bu, \Bc) - C_{h}(\Bb_{h}; \Bu_{h}, \Bc),\\
\label{error_eq3}
B_{h}(e^{\Bu}, q) = & 0,\\
\label{error_eq4}
D_{h}(e^{\Bb}, s) - J_{h}(e^{r},s) = & D_{h}(\boldsymbol{\Pi}_{C}\Bb - \Bb, s), 
\end{align}
\end{subequations}
for all $(\Bv, \Bc, q, r) \in \BV_{h}\times \BC_{h}\times Q_{h}\times S_{h}$. 
\end{lemma}

\begin{proof}
Notice that the exact solution $(\Bu, \Bb, p, r)$ of the equations (\ref{mhd_eqs}) satisfies 
\begin{align*}
A_{h}(\Bu, \Bv) + O_{h}(\Bu; \Bu, \Bv) + C_{h}(\Bb; \Bv, \Bb) + B_{h}(\Bv, p) & = (\Bf, \Bv)_{\Omega}, \\
M_{h}(\Bb, \Bc) - C_{h}(\Bb; \Bu, \Bc) + D_{h} (\Bc, r) & = (\Bg, \Bc)_{\Omega}, \\
B_{h} (\Bu, q) &= 0,\\
D_{h}(\Bb, s) - J_{h}(r, s) &= 0,
\end{align*}
for all $(\Bv, \Bc, q, s)\in \BV_{h}\times \BC_{h}\times Q_{h}\times S_{h}$.
By the definition of $\boldsymbol{\Pi}^{RT}$, $\Pi_{Q}$ and $\Pi_{S}$, we have that for any 
$(\Bv, \Bc, q, r) \in \BV_{h}\times \BC_{h}\times Q_{h}\times S_{h}$, 
\begin{align*}
& A_{h}(\boldsymbol{\Pi}^{RT}\Bu, \Bv) + O_{h}(\Bu; \Bu, \Bv) + C_{h}(\Bb; \Bv, \Bb) + B_{h}(\Bv, \Pi_{Q}p) = 
(\Bf, \Bv)_{\Omega}\\
&\qquad \qquad + A_{h} (\boldsymbol{\Pi}^{RT}\Bu - \Bu, \Bv) 
+ \Sigma_{F\in \Cf_{h}} \langle \Lb \Pi_{Q}p - p\Rb, \lb \Bv \rb_{N}\rangle_{F}, \\
& M_{h}(\boldsymbol{\Pi}_{C}\Bb, \Bc) - C_{h}(\Bb; \Bu, \Bc) + D_{h} (\Bc, \Pi_{S}r) = (\Bg, \Bc)_{\Omega} 
+ M_{h}(\boldsymbol{\Pi}_{C}\Bb - \Bb, \Bc)  + (\Bc, \nabla (\Pi_{S}r -r) )_{\Omega}, \\
& B_{h} (\boldsymbol{\Pi}^{RT}\Bu, q) = 0,\\
& D_{h}(\boldsymbol{\Pi}_{C}\Bb, s) - J_{h}(\Pi_{S}r, s) = D_{h}(\boldsymbol{\Pi}_{C}\Bb - \Bb, s).
\end{align*}
Subtracting the above equations by (\ref{DG_mhd}) gives the result.  
\end{proof}

\subsection*{Step 2: The energy identity derived from the error equations} 
Now we derive the energy identity which states as follows:

\begin{lemma}
We have the energy identity:
\begin{align}
\label{energy_identity}
& A_{h}(e^{\Bu}, e^{\Bu}) + M_{h}(e^{\Bb}, e^{\Bb}) 
+ O_{h}(\mathbb{P}(\Bu_{h}, \Lb \Bu_{h}\Rb); e^{\Bu}, e^{\Bu}) + J_{h} (e^{r}, e^{r}) \\
\nonumber 
= & A_{h} (\boldsymbol{\Pi}^{RT}\Bu - \Bu, e^{\Bu}) 
+ \Sigma_{F\in \Cf_{h}} \langle \Lb \Pi_{Q}p - p\Rb, \lb e^{\Bu} \rb_{N}\rangle_{F} 
+ M_{h}(\boldsymbol{\Pi}_{C}\Bb - \Bb, e^{\Bb})\\ 
\nonumber 
& \quad  + (e^{\Bb}, \nabla (\Pi_{S}r -r) )_{\Omega}  - D_{h}(\boldsymbol{\Pi}_{C}\Bb - \Bb, e^{r}) \\  
\nonumber 
& \quad + \big( O_{h}(\mathbb{P}(\Bu_{h}, \Lb \Bu_{h}\Rb); \boldsymbol{\Pi}^{RT}\Bu - \Bu, e^{\Bu})  
- O_{h} (\mathbb{P}(e^{\Bu}, \Lb e^{\Bu}\Rb) ; \Bu, e^{\Bu})
 + O_{h} (\boldsymbol{\Pi}^{RT} \Bu - \Bu; \Bu, e^{\Bu})  \big) \\
\nonumber 
& \quad + \big(C_{h} (\Bb_{h}; e^{\Bu}, \boldsymbol{\Pi}_{C}\Bb - \Bb) 
- C_{h}(\Bb_{h}; \boldsymbol{\Pi}^{RT}\Bu - \Bu, e^{\Bb})\big) \\
\nonumber 
& \quad + \big( -  C_{h}( e^{\Bb}; e^{\Bu}, \Bb) + C_{h} (e^{\Bb}; \Bu, e^{\Bb})\big) \\ 
\nonumber 
& \quad + \big( C_{h}(\boldsymbol{\Pi}_{C}\Bb - \Bb; e^{\Bu}, \Bb)
 - C_{h} (\boldsymbol{\Pi}_{C}\Bb - \Bb; \Bu, e^{\Bb})\big) \\
 := & T_{1} +\cdots + T_{9}.
\end{align}
\end{lemma}

\begin{proof}
By taking $\Bv:= e^{\Bu}$, $\Bc:=e^{\Bb}$, $q:= -e^{p}$ and $s:= -e^{r}$ in the error equations \eqref{error_eqs} and 
adding all equations, 
we get that 
\begin{align*}
& A_{h}(e^{\Bu}, e^{\Bu}) + M_{h}(e^{\Bb}, e^{\Bb}) + J_{h} (e^{r}, e^{r}) \\
= & A_{h} (\boldsymbol{\Pi}^{RT}\Bu - \Bu, e^{\Bu}) 
+ \Sigma_{F\in \Cf_{h}} \langle \Lb \Pi_{Q}p - p\Rb, \lb e^{\Bu} \rb_{N}\rangle_{F} 
+ M_{h}(\boldsymbol{\Pi}_{C}\Bb - \Bb, e^{\Bb})\\
& \quad  + (e^{\Bb}, \nabla (\Pi_{S}r -r) )_{\Omega}  - D_{h}(\boldsymbol{\Pi}_{C}\Bb - \Bb, e^{r}) \\
& \quad + \big( O_{h}(\mathbb{P}(\Bu_{h}, \Lb \Bu_{h}\Rb); \Bu_{h}, e^{\Bu}) - O_{h}(\Bu; \Bu, e^{\Bu}) \big) \\ 
& \quad + \big( C_{h}(\Bb_{h}; e^{\Bu}, \Bb_{h}) - C_{h}(\Bb; e^{\Bu}, \Bb) \big) 
 + \big( C_{h}(\Bb; \Bu, e^{\Bb}) - C_{h}(\Bb_{h}; \Bu_{h}, e^{\Bb}) \big).
\end{align*}

we denote by 
\begin{align*}
I := & \big( O_{h}(\mathbb{P}(\Bu_{h}, \Lb \Bu_{h}\Rb); \Bu_{h}, e^{\Bu}) - O_{h}(\Bu; \Bu, e^{\Bu}) \big) 
 + \big( C_{h}(\Bb_{h}; e^{\Bu}, \Bb_{h}) - C_{h}(\Bb; e^{\Bu}, \Bb) \big) \\
 & \quad + \big( C_{h}(\Bb; \Bu, e^{\Bb}) - C_{h}(\Bb_{h}; \Bu_{h}, e^{\Bb}) \big).
\end{align*}
Then we get that 
\begin{align*}
I = & \big( O_{h}(\mathbb{P}(\Bu_{h}, \Lb \Bu_{h}\Rb); \Bu_{h} - \Bu, e^{\Bu})  
+ C_{h} (\Bb_{h}; e^{\Bu}, \Bb_{h} - \Bb) - C_{h}(\Bb_{h}; \Bu_{h} - \Bu, e^{\Bb})\big) \\
&\quad + \big( O_{h} (\mathbb{P}(\Bu_{h}, \Lb \Bu_{h}\Rb) - \Bu; \Bu, e^{\Bu}) 
+ C_{h}(\Bb_{h} - \Bb; e^{\Bu}, \Bb) - C_{h} (\Bb_{h} - \Bb; \Bu, e^{\Bb})\big)\\
= & \big( -O_{h}(\mathbb{P}(\Bu_{h}, \Lb \Bu_{h}\Rb); e^{\Bu}, e^{\Bu})  
- C_{h} (\Bb_{h}; e^{\Bu},  e^{\Bb}) + C_{h}(\Bb_{h}; e^{\Bu}, e^{\Bb})\big)\\
&\quad + \big( O_{h}(\mathbb{P}(\Bu_{h}, \Lb \Bu_{h}\Rb); \boldsymbol{\Pi}^{RT}\Bu - \Bu, e^{\Bu})  
+ C_{h} (\Bb_{h}; e^{\Bu}, \boldsymbol{\Pi}_{C}\Bb - \Bb) - C_{h}(\Bb_{h}; \boldsymbol{\Pi}^{RT}\Bu - \Bu, e^{\Bb})\big) \\
& \quad + \big( - O_{h} (\boldsymbol{\Pi}^{RT}\Bu - \mathbb{P}(\Bu_{h}, \Lb \Bu_{h}\Rb) ; \Bu, e^{\Bu}) 
-  C_{h}( e^{\Bb}; e^{\Bu}, \Bb) + C_{h} (e^{\Bb}; \Bu, e^{\Bb})\big) \\
& \quad + \big( O_{h} (\boldsymbol{\Pi}^{RT} \Bu - \Bu; \Bu, e^{\Bu}) 
+ C_{h}(\boldsymbol{\Pi}_{C}\Bb - \Bb; e^{\Bu}, \Bb) - C_{h} (\boldsymbol{\Pi}_{C}\Bb - \Bb; \Bu, e^{\Bb})\big).
\end{align*}
By (\ref{proj_prop5}), we have that 
\begin{align*}
I = & \big( O_{h}(\mathbb{P}(\Bu_{h}, \Lb \Bu_{h}\Rb); \boldsymbol{\Pi}^{RT}\Bu - \Bu, e^{\Bu})  
+ C_{h} (\Bb_{h}; e^{\Bu}, \boldsymbol{\Pi}_{C}\Bb - \Bb) - C_{h}(\Bb_{h}; \boldsymbol{\Pi}^{RT}\Bu - \Bu, e^{\Bb})\big) \\
& \quad + \big( - O_{h} (\mathbb{P}(e^{\Bu}, \Lb e^{\Bu}\Rb) ; \Bu, e^{\Bu}) 
-  C_{h}( e^{\Bb}; e^{\Bu}, \Bb) + C_{h} (e^{\Bb}; \Bu, e^{\Bb})\big) \\
& \quad + \big( O_{h} (\boldsymbol{\Pi}^{RT} \Bu - \Bu; \Bu, e^{\Bu}) 
+ C_{h}(\boldsymbol{\Pi}_{C}\Bb - \Bb; e^{\Bu}, \Bb) - C_{h} (\boldsymbol{\Pi}_{C}\Bb - \Bb; \Bu, e^{\Bb})\big)\\
& \quad  -O_{h}(\mathbb{P}(\Bu_{h}, \Lb \Bu_{h}\Rb); e^{\Bu}, e^{\Bu}).  
\end{align*}
Thus we obtain (\ref{energy_identity}). 
\end{proof}

\subsection*{Step 3: Lower bound of the left hand side of the energy identity}
Now we provide the lower bound of the left hand side of the energy identity (\ref{energy_identity}). 

\begin{lemma}
There is a positive constant $\gamma_{0}$ independent of mesh size such that 
\begin{align}
\label{lower_bound}
& \gamma_{0} \big( \nu \Vert e^{\Bu} \Vert_{V}^{2} + \kappa \nu_{m} \Vert e^{\Bb}\Vert_{C}^{2} 
+ \frac{s_{0}}{\kappa \nu_{m}} h^{-1} \Vert \lb e^{r}\rb \Vert_{L^{2}(\Cf_h)}^{2}\big) \\
\nonumber
\leq & A_{h}(e^{\Bu}, e^{\Bu}) + M_{h}(e^{\Bb}, e^{\Bb}) 
+ O_{h}(\mathbb{P}(\Bu_{h}, \Lb \Bu_{h}\Rb); e^{\Bu}, e^{\Bu}) + J_{h} (e^{r}, e^{r}).
\end{align}
\end{lemma}
\begin{proof}
By Lemma \ref{lemma_proj}, we get that $O_{h}(\mathbb{P}(\Bu_{h}, \Lb \Bu_{h}\Rb); e^{\Bu}, e^{\Bu}) \geq 0$. 
By (\ref{bound_eb}) and Lemma \ref{coarsivity}, we get that 
\begin{align*}
C( \nu \|e^{\Bu}\|^2_V + \kappa \nu_{m} \Vert e^{\Bb} \Vert_{C}^{2}) \leq C M_{h} (e^{\Bb}, e^{\Bb}). 
\end{align*}
Thus we obtain (\ref{lower_bound}). 
\end{proof}

\subsection*{Step 4: Upper bound of the right hand side of the energy identity}
Now we provide upper bound of the right hand side of the energy identity (\ref{energy_identity}). 
We denote by $\boldsymbol{\Pi}_{V}$ the $L^{2}$-orthogonal projection onto $\BV_{h}$. 
Since $\BV_{h} = \BC_{h}$, $\boldsymbol{\Pi}_{V}$ is also the $L^{2}$-orthogonal projection onto $\BC_{h}$. We bound $T_1 - T_9$ as follows:

For $T_1, T_2$, we apply Cauchy-Schwarz inequality with the approximation properties of the projections to have:
\begin{align*}
T_{1} \leq & C \nu \Vert \boldsymbol{\Pi}^{RT}\Bu - \Bu \Vert_{V} \Vert e^{\Bu}\Vert_{V} 
\leq C \nu h^{\min (k, \sigma)} \Vert \Bu\Vert_{H^{1+\sigma}(\Omega)} \Vert e^{\Bu}\Vert_{V},\\
T_{2} \leq & C \Vert h^{\frac{1}{2}} (\Pi_{Q} p - p ) \Vert_{L^{2}(\Cf_{h})} \Vert h^{-\frac{1}{2}} \lb e^{\Bu}\rb_{N} \Vert_{L^{2}(\Cf_{h})}
\leq C h^{\min (k, \sigma)} \Vert p\Vert_{H^{\sigma}(\Omega)}  \Vert h^{-\frac{1}{2}} \lb e^{\Bu}\rb_{N} \Vert_{L^{2}(\Cf_{h})}.\\
\intertext{Similarly, for $T_4$ we have:}
T_{4} \leq & C h^{\min (k, \sigma_{m})} \Vert r\Vert_{H^{1+\sigma_{m}}(\Omega)} \Vert e^{\Bb}\Vert_{C}. 
\end{align*}

For $T_3$, since $\lb \boldsymbol{\Pi}_{C} \Bb\rb_{T} = \lb \Bb\rb_{T} = \boldsymbol{0}$ on $\Ce_{h}$, we have
\begin{align*}
T_{3} = & \Sigma_{K\in \Ct_{h}} (\kappa \nu_{m} \nabla\times (\boldsymbol{\Pi}_{C} \Bb - \Bb), \nabla\times e^{\Bb})_{K} 
- \Sigma_{F\in \Cf_{h}} \langle \Lb \kappa \nu_{m} \nabla \times (\boldsymbol{\Pi}_{C} \Bb - \Bb) \Rb, \lb e^{\Bb}\rb_{T} \rangle_{F} \\
= & \Sigma_{K\in \Ct_{h}} (\kappa \nu_{m} \nabla\times (\boldsymbol{\Pi}_{C} \Bb - \Bb), \nabla\times e^{\Bb})_{K} 
- \Sigma_{F\in \Cf_{h}} \langle \Lb \kappa \nu_{m}
 (\nabla \times \boldsymbol{\Pi}_{C} \Bb - \boldsymbol{\Pi}_{V} \nabla\times \Bb) \Rb, \lb e^{\Bb}\rb_{T} \rangle_{F} \\
&\quad - \Sigma_{F\in \Cf_{h}} \langle \Lb \kappa \nu_{m}  (\boldsymbol{\Pi}_{V} \nabla \times\Bb - \nabla \times\Bb) \Rb, 
\lb e^{\Bb}\rb_{T} \rangle_{F} \\
\leq & C \kappa \nu_{m} \big( \Vert \nabla\times (\boldsymbol{\Pi}_{C} \Bb - \Bb)\Vert_{L^{2}(\Omega)}
 \Vert \nabla\times e^{\Bb}\Vert_{L^{2}(\Omega)}\\
&\quad  + \Vert \nabla \times \boldsymbol{\Pi}_{C} \Bb - \boldsymbol{\Pi}_{V} \nabla\times \Bb \Vert_{L^{2}(\Omega)} 
 \Vert h^{-\frac{1}{2}} \lb e^{\Bb} \rb_{T} \Vert_{L^{2}(\Cf_{h})} \\
& \quad  + \Vert h^{\frac{1}{2}} (\boldsymbol{\Pi}_{V} \nabla \times\Bb - \nabla \times\Bb)\Vert_{L^{2}(\Cf_{h})} 
 \Vert h^{-\frac{1}{2}} \lb e^{\Bb} \rb_{T} \Vert_{L^{2}(\Cf_{h})} \big) \\
\leq & C \kappa \nu_{m} h^{\min (k, \sigma_{m})} \Vert \nabla\times \Bb\Vert_{H^{\sigma_{m}}(\Omega)} \Vert e^{\Bb}\Vert_{C}. 
\end{align*}
The last inequality above is due to (\ref{proj_C_prop3}). 

With respect to the term $T_{5}$, we choose $\tilde{e}^{r}\in H_{0}^{1}(\Omega) \cap S_{h}$ 
(see \cite[Theorem~$2.2$ and Theorem~$2.3$]{Karakashian03}) satisfying 
\begin{equation}\label{er_tilde}
\Vert \nabla (e^{r} - \tilde{e}^{r}) \Vert_{L^{2}(\Ct_{h})} \leq C \Vert h^{-\frac{1}{2}} \lb e^{r}\rb \Vert_{L^{2}(\Cf_{h})}.
\end{equation}
By the definition of $D_{h}$, the fact that $\nabla\cdot \Bb = 0$ and (\ref{proj_C_prop1}), we have that
\begin{align*}
D_{h}(\boldsymbol{\Pi}_{C} \Bb - \Bb, \tilde{e}^{r}) = (\boldsymbol{\Pi}_{C} \Bb - \Bb, \nabla \tilde{e}^{r})_{\Omega} = 0.
\end{align*}
On the other hand, 
\begin{align*}
T_5 =& D_{h}(\boldsymbol{\Pi}_{C} \Bb - \Bb, e^{r} - \tilde{e}^{r}) \\
\leq & C \big( \Vert \boldsymbol{\Pi}_{C} \Bb - \Bb\Vert_{L^{2}(\Omega)} 
+ \Vert h^{\frac{1}{2}} (\boldsymbol{\Pi}_{C} \Bb - \Bb) \Vert_{L^{2}(\Cf_{h})} \big) 
\Vert h^{-\frac{1}{2}} \lb e^{r}\rb \Vert_{L^{2}(\Cf_{h})}\\
\leq & C \big( \Vert \boldsymbol{\Pi}_{C} \Bb - \Bb\Vert_{L^{2}(\Omega)} 
+ \Vert h^{\frac{1}{2}} (\boldsymbol{\Pi}_{C} \Bb - \boldsymbol{\Pi}_{V}\Bb) \Vert_{L^{2}(\Cf_{h})} 
+ \Vert h^{\frac{1}{2}} (\boldsymbol{\Pi}_{V} \Bb - \Bb) \Vert_{L^{2}(\Cf_{h})}\big) 
\Vert h^{-\frac{1}{2}} \lb e^{r}\rb \Vert_{L^{2}(\Cf_{h})} \\
\leq & C \big( \Vert \boldsymbol{\Pi}_{C} \Bb - \Bb\Vert_{L^{2}(\Omega)} 
+ \Vert h^{\frac{1}{2}} (\boldsymbol{\Pi}_{V} \Bb - \Bb) \Vert_{L^{2}(\Cf_{h})} \big) 
\Vert h^{-\frac{1}{2}} \lb e^{r}\rb \Vert_{L^{2}(\Cf_{h})}\\
\leq & C h^{\min (k, \sigma_{m})} \Vert \Bb \Vert_{H^{\sigma_{m}}(\Omega)}
\Vert h^{-\frac{1}{2}} \lb e^{r}\rb \Vert_{L^{2}(\Cf_{h})}. 
\end{align*}
The last inequality above is due to (\ref{proj_C_prop2}). Thus we get that 
\begin{align*}
T_{5} \leq  C h^{\min (k, \sigma_{m})} \Vert \Bb \Vert_{H^{\sigma_{m}}(\Omega)}
\Vert h^{-\frac{1}{2}} \lb e^{r}\rb \Vert_{L^{2}(\Cf_{h})}.
\end{align*}

For the convection term $T_6$, according to \cite[Proposition~$4.2$]{CockburnKanschatSchoetzauNS05}, we get that 
\begin{align*}
T_{6} & \le C (\|\mathbb{P}(\Bu_{h}, \Lb \Bu_{h}\Rb)\|_V \|\boldsymbol{\Pi}^{RT} \Bu - \Bu\|_V + \|\mathbb{P}(e^{\Bu}, \Lb e^{\Bu}\Rb)\|_V \|\Bu\|_V + \|\boldsymbol{\Pi}^{RT} \Bu - \Bu\|_V \|\Bu\|_V) \|e^{\Bu}\|_V \\
&\leq C \big( h^{\min(k,\sigma)}\Vert \Bu\Vert_{H^{1+\sigma}(\Omega)} (\Vert \Bu\Vert_{H^{1}(\Omega)} 
+ \Vert \Bu_{h}\Vert_{V})  + \Vert \Bu\Vert_{H^{1}(\Omega)} \Vert e^{\Bu}\Vert_{V} \big) \Vert e^{\Bu} \Vert_{V}.
\end{align*}
The last inequality is due to \eqref{proj_prop3} and the approximation property of $\boldsymbol{\Pi}^{RT}$. 

The estimates for the last three terms are more involved which we state as following lemmas.
\begin{lemma}
We have
\begin{align}
\label{bound_T7} 
T_{7} \leq C \kappa \Vert \Bb_{h}\Vert_{C} 
\big( h^{\min(k,\sigma)} \Vert \Bu\Vert_{H^{1+\sigma}(\Omega)} \Vert e^{\Bb}\Vert_{C} 
+ h^{\min(k,\sigma_{m})} \Vert \nabla \times \Bb\Vert_{H^{\sigma_{m}}(\Omega)} \Vert e^{\Bu}\Vert_{V}\big).
\end{align}
\end{lemma}

\begin{proof}
We recall that 
\begin{align*}
T_{7} = & C_{h} (\Bb_{h}; e^{\Bu}, \boldsymbol{\Pi}_{C}\Bb - \Bb) - C_{h}(\Bb_{h}; \boldsymbol{\Pi}^{RT}\Bu - \Bu, e^{\Bb}) \\
= & \Sigma_{K\in\Ct_{h}}  (\kappa(e^{\Bu}\times \Bb_{h}), \nabla \times (\boldsymbol{\Pi}_{C}\Bb - \Bb))_{K} 
- \Sigma_{F\in \Cf_{h}^{I}} \langle \kappa \Lb e^{\Bu}\times \Bb_{h}\Rb, \lb \boldsymbol{\Pi}_{C}\Bb - \Bb\rb_{T}\rangle_{F} \\
&\quad + \Sigma_{K\in\Ct_{h}}  (\kappa((\boldsymbol{\Pi}^{RT}\Bu - \Bu)\times \Bb_{h}), \nabla \times e^{\Bb})_{K} 
- \Sigma_{F\in \Cf_{h}^{I}} \langle \kappa \Lb (\boldsymbol{\Pi}^{RT}\Bu - \Bu) \times \Bb_{h}\Rb, \lb e^{\Bb}\rb_{T}\rangle_{F}.
\end{align*}
Due to the construction of $\boldsymbol{\Pi}_{C}$ in (\ref{def_proj_C}), $\lb \boldsymbol{\Pi}_{C} \Bb\rb_{T} = \boldsymbol{0}$ 
on $\Ce_{h}$. Thus we have that 
\begin{align*}
T_{7} = & - \Sigma_{F\in \Cf_{h}^{I}} \langle \kappa \Lb (\boldsymbol{\Pi}^{RT}\Bu - \Bu) \times \Bb_{h}\Rb, \lb e^{\Bb}\rb_{T}\rangle_{F} \\
&\quad + \Sigma_{K\in\Ct_{h}}  (\kappa(e^{\Bu}\times \Bb_{h}), \nabla \times (\boldsymbol{\Pi}_{C}\Bb - \Bb))_{K} 
+ \Sigma_{K\in\Ct_{h}}  (\kappa((\boldsymbol{\Pi}^{RT}\Bu - \Bu)\times \Bb_{h}), \nabla \times e^{\Bb})_{K}\\
:= & T_{71} + T_{72} + T_{73}. 
\end{align*}
By generalized H{\"{o}}lder's inequality, (\ref{discrete_energy_norm_bound}) and scaling argument to $\Bb_{h}$, we have
\begin{align*}
T_{71} \leq & C \kappa \Sigma_{K \in \Ct_{h}} h \Vert \Bb_{h}\Vert_{L^{\infty}(\partial K)}
\Vert h^{-\frac{1}{2}} (\boldsymbol{\Pi}^{RT}\Bu - \Bu)\Vert_{L^{2}(\partial K)}
 \Vert h^{-\frac{1}{2}} \lb e^{\Bb} \rb_{T} \Vert_{L^{2}(\partial K)}\\
\leq & C \kappa \Sigma_{K \in \Ct_{h}} \Vert \Bb_{h}\Vert_{L^{3}( K)}
\Vert h^{-\frac{1}{2}} (\boldsymbol{\Pi}^{RT}\Bu - \Bu)\Vert_{L^{2}(\partial K)}
 \Vert h^{-\frac{1}{2}} \lb e^{\Bb} \rb_{T} \Vert_{L^{2}(\partial K)}\\
 \leq & C \kappa \Sigma_{K \in \Ct_{h}} \Vert \Bb_{h}\Vert_{L^{3}( K)}
\big( \Vert h^{-\frac{1}{2}} (\boldsymbol{\Pi}^{RT}\Bu - \boldsymbol{\Pi}_{V}\Bu)\Vert_{L^{2}(\partial K)}
 + \Vert h^{-\frac{1}{2}} (\boldsymbol{\Pi}_{V}\Bu - \Bu)\Vert_{L^{2}(\partial K)} \big)
 \Vert h^{-\frac{1}{2}} \lb e^{\Bb} \rb_{T} \Vert_{L^{2}(\partial K)}\\
\leq & C \kappa \Vert \Bb_{h}\Vert_{L^{3}( \Omega)} \big( 
\Vert h^{-1} (\boldsymbol{\Pi}^{RT}\Bu - \boldsymbol{\Pi}_{V}\Bu )\Vert_{L^{2}(\Omega)}
+\Vert h^{-\frac{1}{2}} (\boldsymbol{\Pi}_{V}\Bu - \Bu)\Vert_{L^{2}(\Cf_{h})} \big)
 \Vert h^{-\frac{1}{2}} \lb e^{\Bb} \rb_{T} \Vert_{L^{2}(\Cf_{h})}\\ 
%\leq & C \kappa \Vert \Bb_{h}\Vert_{L^{3}( \Omega)} \Vert h^{-\frac{1}{2}} 
%(\boldsymbol{\Pi}^{RT}\Bu - \Bu)\Vert_{L^{2}(\Cf_{h})} 
% \Vert h^{-\frac{1}{2}} \lb e^{\Bb} \rb_{T} \Vert_{L^{2}(\Cf_{h})}\\
 \leq & C \kappa h^{\min(k, \sigma)} \Vert \Bu\Vert_{H^{1+\sigma}(\Omega)} \Vert \Bb_{h}\Vert_{C}
  \Vert h^{-\frac{1}{2}} \lb e^{\Bb} \rb_{T} \Vert_{L^{2}(\Cf_{h})}. 
\end{align*}
By generalized H{\"{o}}lder's inequality, Sobolev imbedding, \cite[Theorem~$5.3$]{DiPietroErn}, 
(\ref{discrete_energy_norm_bound}) and (\ref{proj_C_prop3}), 
\begin{align*}
T_{72} \leq & C \kappa \Vert e^{\Bu}\Vert_{L^6(\Omega)} \Vert \Bb_{h}\Vert_{L^3(\Omega)} \Vert \nabla \times 
(\boldsymbol{\Pi}_{C}\Bb - \Bb) \Vert_{L^{2}(\Omega)} \le C \kappa \Vert e^{\Bu}\Vert_{V} \Vert \Bb_{h}\Vert_{C} \Vert \nabla \times 
(\boldsymbol{\Pi}_{C}\Bb - \Bb) \Vert_{L^{2}(\Omega)} \\
\leq & C \kappa h^{\min(k, \sigma_{m})} \Vert \nabla\times \Bb\Vert_{H^{\sigma_{m}}(\Omega)} 
\Vert \Bb_{h}\Vert_{C} \Vert e^{\Bu}\Vert_{V}, \\
T_{73} \le &C \kappa \|\boldsymbol{\Pi}^{RT}\Bu - \Bu\|_{L^6(\Omega)} \|\Bb_h\|_{L^3(\Omega)} 
\|\nabla \times \Bb_h\|_{L^2(\Ct_h)} \\
 \leq  & C \kappa \big( \Vert h^{-1} (\boldsymbol{\Pi}^{RT}\Bu - \Bu )\Vert_{L^{2}(\Omega)} +
\Vert \boldsymbol{\Pi}^{RT}\Bu - \Bu\Vert_{H^{1}(\Ct_{h})} \big) \Vert \Bb_{h}\Vert_{C} \Vert e^{b}\Vert_{C} \\
 \leq  & C \kappa \big( \Vert h^{-1} (\boldsymbol{\Pi}^{RT}\Bu - \Bu )\Vert_{L^{2}(\Omega)} +
\Vert \boldsymbol{\Pi}^{RT}\Bu - \boldsymbol{\Pi}_{V}\Bu\Vert_{H^{1}(\Ct_{h})} 
+ \Vert \boldsymbol{\Pi}_{V}\Bu - \Bu\Vert_{H^{1}(\Ct_{h})}\big) \Vert \Bb_{h}\Vert_{C} \Vert e^{b}\Vert_{C} \\
 \leq  & C \kappa \big( \Vert h^{-1} (\boldsymbol{\Pi}^{RT}\Bu - \Bu )\Vert_{L^{2}(\Omega)} +
\Vert h^{-1}(\boldsymbol{\Pi}^{RT}\Bu - \boldsymbol{\Pi}_{V}\Bu )\Vert_{L^{2}(\Omega)} 
+ \Vert \boldsymbol{\Pi}_{V}\Bu - \Bu\Vert_{H^{1}(\Ct_{h})}\big) \Vert \Bb_{h}\Vert_{C} \Vert e^{b}\Vert_{C} \\
\leq & C \kappa h^{\min(k,\sigma)}\Vert \Bu \Vert_{H^{1+\sigma}(\Omega)} \Vert \Bb_{h} \Vert_{C} \Vert e^{\Bb}\Vert_{C}. 
\end{align*}

Thus we obtain (\ref{bound_T7}).
\end{proof}

\begin{lemma}
\begin{align}
\label{bound_T8}
T_{8} \leq C \kappa \big(\Vert \Bu \Vert_{H^{1}(\Omega)}  \Vert e^{\Bb}\Vert_{C}^{2} 
+ \Vert \nabla\times \Bb\Vert_{L^{2}(\Omega)} \Vert e^{\Bu} \Vert_{V} \Vert e^{\Bb} \Vert_{C} \big).
\end{align}
\end{lemma}

\begin{proof}
We recall that 
\begin{align*}
T_{8} = & -  C_{h}( e^{\Bb}; e^{\Bu}, \Bb) + C_{h} (e^{\Bb}; \Bu, e^{\Bb}) \\
= & -\Sigma_{K\in\Ct_{h}}  (\kappa(e^{\Bu}\times e^{\Bb}), \nabla \times \Bb)_{K} 
+ \Sigma_{F\in \Cf_{h}^{I}} \langle \kappa \Lb e^{\Bu}\times e^{\Bb}\Rb, \lb \Bb\rb_{T}\rangle_{F}\\
&\quad + \Sigma_{K\in\Ct_{h}}  (\kappa(\Bu \times e^{\Bb}), \nabla \times e^{\Bb})_{K} 
- \Sigma_{F\in \Cf_{h}^{I}} \langle \kappa \Lb \Bu \times e^{\Bb}\Rb, \lb e^{\Bb}\rb_{T}\rangle_{F}.
\end{align*}
Since $\lb \Bb\rb_{T} = 0$ on $\Cf_{h}^{I}$, we have that 
\begin{align*}
T_{8} = & - \Sigma_{F\in \Cf_{h}^{I}} \langle \kappa \Lb \Bu \times e^{\Bb}\Rb, \lb e^{\Bb}\rb_{T}\rangle_{F}\\
& \quad  -\Sigma_{K\in\Ct_{h}}  (\kappa(e^{\Bu}\times e^{\Bb}), \nabla \times \Bb)_{K} 
+ \Sigma_{K\in\Ct_{h}}  (\kappa(\Bu \times e^{\Bb}), \nabla \times e^{\Bb})_{K}\\
 := & T_{81} + T_{82} + T_{83}. 
\end{align*} 

Obviously, 
\begin{align*}
T_{81} = & - \Sigma_{K \in \Ct_{h}} \langle \kappa (\Bu\times e^{\Bb}), \lb e^{\Bb}\rb_{T} \Vert \rangle_{\partial K \cap \Cf_{h}^{I}} \\
= & - \Sigma_{K \in \Ct_{h}} \langle \kappa ( \boldsymbol{\Pi}_{V}\Bu\times e^{\Bb}), \lb e^{\Bb}\rb_{T}
\rangle_{\partial K \cap \Cf_{h}^{I}}  
+ \Sigma_{K \in \Ct_{h}} \langle \kappa  (\boldsymbol{\Pi}_{V}\Bu - \Bu)\times e^{\Bb}, \lb e^{\Bb}\rb_{T} 
 \rangle_{\partial K \cap \Cf_{h}^{I}}. 
\end{align*}
By scaling argument, it is easy to see that $\Vert \boldsymbol{\Pi}_{V} \Bu\Vert_{L^{6}(\Omega)}
\leq C \Vert \Bu\Vert_{L^{6}(\Omega)} \leq C \Vert \Bu\Vert_{H^{1}(\Omega)}$. So by discrete trace inequality and (\ref{bound_eb}), 
\begin{align*}
& - \Sigma_{K \in \Ct_{h}} \langle \kappa ( \boldsymbol{\Pi}_{V}\Bu\times e^{\Bb}), \lb e^{\Bb}\rb_{T}
\rangle_{\partial K \cap \Cf_{h}^{I}} \\
\leq & C \kappa \Vert \boldsymbol{\Pi}_{V} \Bu \Vert_{L^{6}(\Omega)} \Vert e^{\Bb}\Vert_{L^{3}(\Omega)}  
\Vert h^{-\frac{1}{2}} \lb e^{\Bb} \rb \Vert_{L^{2}(\Cf_{h})}
\leq  C \kappa \Vert \Bu\Vert_{H^{1}(\Omega)} \Vert e^{\Bb}\Vert_{C}^{2}. 
\end{align*}
On the other hand, by scaling argument and discrete trace inequality to $e^{\Bb}$, we get that 
\begin{align*}
& \Sigma_{K \in \Ct_{h}} \langle \kappa  (\boldsymbol{\Pi}_{V}\Bu - \Bu)\times e^{\Bb}, \lb e^{\Bb}\rb_{T} 
\rangle_{\partial K \cap \Cf_{h}^{I}} \\ 
\leq & C \kappa \Sigma_{K \in \Ct_{h}} \Vert h^{-\frac{1}{2}}(\boldsymbol{\Pi}_{V}\Bu - \Bu)\Vert_{L^{2}(\partial K)} 
h\Vert  e^{\Bb} \Vert_{L^{\infty}(\partial K)} \Vert h^{-\frac{1}{2}} \lb e^{\Bb} \rb_{T} \Vert_{L^{2}(\partial K)}\\
\leq & C \kappa \Sigma_{K \in \Ct_{h}} \Vert  e^{\Bb} \Vert_{L^{3}( K)}
\Vert h^{-\frac{1}{2}}(\boldsymbol{\Pi}_{V}\Bu - \Bu)\Vert_{L^{2}(\partial K)} 
 \Vert h^{-\frac{1}{2}} \lb e^{\Bb} \rb_{T} \Vert_{L^{2}(\partial K)}\\
\leq & C \kappa \Vert e^{\Bb} \Vert_{L^{3}(\Omega)} 
\Vert h^{-\frac{1}{2}}(\boldsymbol{\Pi}_{V}\Bu - \Bu )\Vert_{L^{2}(\Cf_{h})}
\Vert h^{-\frac{1}{2}} \lb e^{\Bb} \rb_{T} \Vert_{L^{2}(\Cf_{h})}\\
\leq & C \kappa \Vert e^{\Bb} \Vert_{L^{3}(\Omega)} 
\big(\Vert h^{-1}(\boldsymbol{\Pi}_{V}\Bu - \Bu )\Vert_{L^{2}(\Omega)} 
+ \Vert \boldsymbol{\Pi}_{V}\Bu - \Bu \Vert_{H^{1}(\Ct_{h})} \big)
\Vert h^{-\frac{1}{2}} \lb e^{\Bb} \rb_{T} \Vert_{L^{2}(\Cf_{h})}\\
\leq & C \kappa \Vert \Bu\Vert_{H^{1}(\Omega)} \Vert e^{\Bb}\Vert_{C}^{2}.  
\end{align*}
We utilized (\ref{bound_eb}) in the last inequality above. Thus we get that 
\begin{align*}
T_{81} \leq C \kappa \Vert \Bu \Vert_{H^{1}(\Omega)} \Vert e^{\Bb}\Vert_{C}^{2}. 
\end{align*}

By generalized H{\"{o}}lder's inequality, Sobolev imbedding, \cite[Theorem~$5.3$]{DiPietroErn} and 
(\ref{bound_eb}), we get that
\begin{align*}
T_{82} \leq & C \kappa \Vert \nabla\times \Bb\Vert_{L^{2}(\Omega)} \Vert e^{\Bu} 
\Vert_{L^{6}(\Omega)} \Vert e^{\Bb} \Vert_{L^{3}(\Omega)} 
\leq C \kappa \Vert \nabla\times \Bb\Vert_{L^{2}(\Omega)} \Vert e^{\Bu} 
\Vert_{V} \Vert e^{\Bb} \Vert_{C},\\ 
T_{83} \leq & C \kappa \Vert \Bu\Vert_{L^{6}(\Omega)} \Vert e^{\Bb}\Vert_{L^{3}(\Omega)} 
\Vert \nabla\times e^{\Bb} \Vert_{L^{2}(\Omega)} 
\leq C \kappa \Vert \Bu\Vert_{H^{1}(\Omega)} \Vert e^{\Bb}\Vert_{C}^{2}. 
\end{align*}

Thus we obtain (\ref{bound_T8}). 
\end{proof}

\begin{lemma}
\begin{align}
\label{bound_T9}
T_{9} \leq & C \kappa h^{\min (k, \sigma_{m})} \Vert \Bb\Vert_{H^{\sigma_{m}}(\Omega)} 
\big( \Vert \nabla\times \Bb\Vert_{H^{\sigma_{m}}(\Omega)} \Vert e^{\Bu} \Vert_{V} 
+ \Vert \Bu\Vert_{H^{1+\sigma}(\Omega)} \Vert e^{\Bb}\Vert_{C} \big).
\end{align}
\end{lemma}

\begin{proof}
We recall that 
\begin{align*}
T_{9} = & C_{h}(\boldsymbol{\Pi}_{C}\Bb - \Bb; e^{\Bu}, \Bb) - C_{h} (\boldsymbol{\Pi}_{C}\Bb - \Bb; \Bu, e^{\Bb}) \\
 = & \Sigma_{K\in\Ct_{h}}  (\kappa(e^{\Bu}\times (\boldsymbol{\Pi}_{C}\Bb - \Bb)), \nabla \times \Bb)_{K} 
- \Sigma_{F\in \Cf_{h}^{I}} \langle \kappa \Lb e^{\Bu} \times (\boldsymbol{\Pi}_{C}\Bb - \Bb) \Rb, \lb \Bb\rb_{T}\rangle_{F}\\
& \quad - \Sigma_{K\in\Ct_{h}}  (\kappa(\Bu\times (\boldsymbol{\Pi}_{C}\Bb - \Bb)), \nabla \times e^{\Bb})_{K} 
+ \Sigma_{F\in \Cf_{h}^{I}} \langle \kappa \Lb \Bu\times (\boldsymbol{\Pi}_{C}\Bb - \Bb)\Rb, \lb e^{\Bb} \rb_{T}\rangle_{F}.
\end{align*}
Since $\lb \Bb \rb_{T} = \boldsymbol{0}$ on $\Cf_{h}^{I}$, we have that 
\begin{align*}
T_{9} =  &\Sigma_{F\in \Cf_{h}^{I}} \langle \kappa \Lb \Bu\times (\boldsymbol{\Pi}_{C}\Bb - \Bb)\Rb, \lb e^{\Bb} \rb_{T}\rangle_{F}\\
& \quad +\Sigma_{K\in\Ct_{h}}  (\kappa(e^{\Bu}\times (\boldsymbol{\Pi}_{C}\Bb - \Bb)), \nabla \times \Bb)_{K}  
- \Sigma_{K\in\Ct_{h}}  (\kappa(\Bu\times (\boldsymbol{\Pi}_{C}\Bb - \Bb)), \nabla \times e^{\Bb})_{K}\\
:= & T_{91} + T_{92} + T_{93}.
\end{align*}

By discrete trace inequality and (\ref{proj_C_prop2}), 
\begin{align*}
T_{91} = & \Sigma_{F\in \Cf_{h}^{I}} \langle \kappa \Lb \Bu\times (\boldsymbol{\Pi}_{C}\Bb - \boldsymbol{\Pi}_{V}\Bb)\Rb, 
\lb e^{\Bb} \rb_{T}\rangle_{F} 
+ \Sigma_{F\in \Cf_{h}^{I}} \langle \kappa \Lb \Bu\times (\boldsymbol{\Pi}_{V}\Bb - \Bb)\Rb, \lb e^{\Bb} \rb_{T}\rangle_{F}\\
\leq & C \kappa \Vert \Bu\Vert_{L^{\infty}(\Omega)} \big( \Vert h^{1/2}
 (\boldsymbol{\Pi}_{C}\Bb - \boldsymbol{\Pi}_{V}\Bb)\Vert_{L^{2}(\Cf_{h})} 
 + \Vert h^{\frac{1}{2}} (\boldsymbol{\Pi}_{V}\Bb - \Bb) \Vert_{L^{2}(\Ce_{h})}\big) 
 \Vert h^{-\frac{1}{2}}\lb e^{\Bb} \rb_{T} \Vert_{L^{2}(\Cf_{h})}\\
 \leq & C \kappa \Vert \Bu\Vert_{L^{\infty}(\Omega)} \big( \Vert 
 \boldsymbol{\Pi}_{C}\Bb - \boldsymbol{\Pi}_{V}\Bb \Vert_{L^{2}(\Omega)} 
 + \Vert h^{\frac{1}{2}} (\boldsymbol{\Pi}_{V}\Bb - \Bb) \Vert_{L^{2}(\Cf_{h})}\big) 
 \Vert h^{-\frac{1}{2}}\lb e^{\Bb} \rb_{T} \Vert_{L^{2}(\Cf_{h})}\\
 \leq & C \kappa h^{\min(k,\sigma_{m})} \Vert \Bu\Vert_{L^{\infty}(\Omega)} \Vert \Bb\Vert_{H^{\sigma_{m}}(\Omega)}
  \Vert h^{-\frac{1}{2}}\lb e^{\Bb} \rb_{T} \Vert_{L^{2}(\Cf_{h})}.
\end{align*}

By (\ref{proj_C_prop3}), it is east to see that 
\begin{align*}
& T_{92}+ T_{93} \\
\leq & C \kappa \big( \Vert e^{\Bu}\Vert_{L^{6}(\Omega)} \Vert \boldsymbol{\Pi}_{C}\Bb - \Bb \Vert_{L^{2}(\Omega)} 
\Vert \nabla\times \Bb\Vert_{L^{3}(\Omega)} + \Vert \Bu\Vert_{L^{\infty}(\Omega)} 
\Vert \boldsymbol{\Pi}_{C}\Bb - \Bb\Vert_{L^{2}(\Omega)} \Vert \nabla\times e^{\Bb} \Vert_{L^{2}(\Omega)}\big) \\
\leq & C \kappa \big( \Vert e^{\Bu}\Vert_{V} \Vert \boldsymbol{\Pi}_{C}\Bb - \Bb \Vert_{L^{2}(\Omega)} 
\Vert \nabla\times \Bb\Vert_{H^{\sigma_{m}}(\Omega)} + \Vert \Bu\Vert_{L^{\infty}(\Omega)} 
\Vert \boldsymbol{\Pi}_{C}\Bb - \Bb\Vert_{L^{2}(\Omega)} \Vert e^{\Bb} \Vert_{C}\big) \\
\leq & C \kappa h^{\min (k, \sigma_{m})} \Vert \Bb\Vert_{H^{\sigma_{m}}(\Omega)} 
\big( \Vert \nabla\times \Bb\Vert_{H^{\sigma_{m}}(\Omega)} \Vert e^{\Bu} \Vert_{V} 
+ \Vert \Bu\Vert_{L^{\infty}(\Omega)} \Vert e^{\Bb}\Vert_{C} \big).
\end{align*}

Thus we get that 
\begin{align*}
T_{9} \leq & C \kappa h^{\min (k, \sigma_{m})} \Vert \Bb\Vert_{H^{\sigma_{m}}(\Omega)} 
\big(  \Vert \Bu\Vert_{L^{\infty}(\Omega)}   \Vert h^{-\frac{1}{2}}\lb e^{\Bb} \rb_{T} \Vert_{L^{2}(\Cf_{h})} \\
&\qquad +\Vert \nabla\times \Bb\Vert_{H^{\sigma_{m}}(\Omega)} \Vert e^{\Bu} \Vert_{V} 
+ \Vert \Bu\Vert_{L^{\infty}(\Omega)} \Vert e^{\Bb}\Vert_{C} \big).
\end{align*} 
Since $\Vert \Bu\Vert_{L^{\infty}(\Omega)} \leq C \Vert \Bu \Vert_{H^{1+\sigma}(\Omega)}$, we obtain (\ref{bound_T9}).
\end{proof}

%\begin{remark}
%In the above estimate, we could improve the control of $T_{92}, T_{93}$ if we have the embedding $\|\boldsymbol{\Pi}_C \Bb - \Bb\|_{L^3(\Omega)} \le C \|\boldsymbol{\Pi}_C \Bb - \Bb\|_C$. Notice here $\boldsymbol{\Pi}_C \Bb$ only satisfies \eqref{assumption_div_free} (discrete divergence-free) where $\Bb$ is strongly divergence-free. Although the embedding is valid for each individual term, the embedding cannot be applied for the difference trivially.
%\end{remark}

Finally, the estimate of $\Bu - \Bu_h, \Bb - \Bb_h$ in Theorem \ref{errors_ub} can be obtained by combining the estimates for $T_1 - T_9$ together with Theorem \ref{theorem_wellposed} and the assumption that $\dfrac{1}{\min(\nu, \nu_{m})}\Vert \Bu\Vert_{H^{1}(\Omega)}$ and 
$\dfrac{1}{\sqrt{\nu \kappa \nu_{m}}}\Vert \nabla\times \Bb \Vert_{L^{2}(\Omega)}$ are small enough. 

\subsection*{Step 5: Estimates for $r - r_h$}

By a triangle inequality we have
\[
\|r - r_h\|_S \le \|e^r\|_S + \|r - \Pi_S r\|_S,
\]
with the approximation property of the interpolant we have
\[
\|r - \Pi_S r\|_S \le C h^{\text{min}\{k+1, \sigma_m+1\}}\|r\|_{H^{\sigma_m + 1}(\Omega)}.
\]
Therefore it suffice to bound $\|e^r\|_S$. To this end, by \eqref{er_tilde}, considering $\tilde{e}^r \in H^1_0(\Omega) \cap S_h$ we have
\begin{equation}\label{estimate_r1}
\|\tilde{e}^r\|_S = \|\nabla \tilde{e}^r\|_{L^2(\Ct_h)} \le \|e^r\|_S + \|e^r - \tilde{e}^r\|_S \le \|\nabla e^r\|_{L^2(\Ct_h)} + C \Vert h^{-\frac{1}{2}} \lb e^{r}\rb \Vert_{L^{2}(\Cf_{h})}.
\end{equation}
The last term is in the lower bound in Lemma \ref{lower_bound} which has the same estimate as $\Bu - \Bu_h$ in Theorem \ref{errors_ub}. For $\|\nabla e^r\|_{L^2(\Ct_h)}$, taking $\Bc = \nabla \tilde{e}^r$ in the error equation \eqref{error_eq2} we have:
\begin{align}\label{error_r2}
M_{h}(e^{\Bb}, \nabla \tilde{e}^r) + D_{h}(\nabla \tilde{e}^r, e^{r}) = & M_{h}(\boldsymbol{\Pi}_{C}\Bb - \Bb, \nabla \tilde{e}^r) 
+ (\nabla \tilde{e}^r, \nabla (\Pi_{S}r -r) )_{\Omega}\\ 
\nonumber
&\quad + C_{h}(\Bb; \Bu, \nabla \tilde{e}^r) - C_{h}(\Bb_{h}; \Bu_{h}, \nabla \tilde{e}^r).
\end{align}
By the definition of $D_h(\cdot, \cdot)$ and the fact that $\tilde{e}^r \in H^1_0(\Omega) \cap S_h$ we have:
\begin{align*}
D_h(\nabla \tilde{e}^r, e^r) &= \Sigma_{K \in \Ct_h} (\nabla \tilde{e}^r, \nabla e^r)_K - \Sigma_{F \in \Cf_h}\langle \Lb \nabla \tilde{e}^r \Rb \, , \, \lb e^r \rb \rangle_F, \\
& = \|\nabla e^r\|^2_{L^2(\Ct_h)} + \Sigma_{K \in \Ct_h} (\nabla (\tilde{e}^r - e^r), \nabla e^r)_K - \Sigma_{F \in \Cf_h}\langle \Lb \nabla \tilde{e}^r \Rb \, , \, \lb e^r \rb \rangle_F.
\end{align*}
On the other hand, due to the fact that $\tilde{e}^r \in H^1_0(\Omega) \cap S_h$, with the definition of $C_h(\cdot; \cdot, \cdot)$ we have
\[
 C_{h}(\Bb; \Bu, \nabla \tilde{e}^r) = C_{h}(\Bb_{h}; \Bu_{h}, \nabla \tilde{e}^r) = 0.
\]
Inserting above two identities into \eqref{error_r2}, rearranging terms we have
\begin{align*}
\|\nabla e^r\|^2_{L^2(\Ct_h)} &=  M_{h}(\Bb_h - \Bb, \nabla \tilde{e}^r) 
+ (\nabla \tilde{e}^r, \nabla (\Pi_{S}r -r) )_{\Omega} \\
& -\Sigma_{K \in \Ct_h} (\nabla (\tilde{e}^r - e^r), \nabla e^r)_K - \Sigma_{F \in \Cf_h}\langle \Lb \nabla \tilde{e}^r \Rb \, , \, \lb e^r \rb \rangle_F \\
\intertext{by the definition of $M_h(\cdot, \cdot)$ and $\tilde{e}^r \in H^1_0(\Omega) \cap S_h$, we have $M_{h}(\Bb_h - \Bb, \nabla \tilde{e}^r) = 0$. For the rest terms we apply Cauchy Schwarz inequality to obtain:}
\|\nabla e^r\|^2_{L^2(\Ct_h)} &\le \|\nabla \tilde{e}^r\|_{L^2(\Omega)} \|\nabla (\Pi_{S}r -r)  \|_{L^2(\Omega)} + \|\nabla (\tilde{e}^r - e^r)\|_{L^2(\Ct_h)} \|\nabla e^r\|_{L^2(\Ct_h)} \\
&+ \|h^{\frac12} \Lb \nabla \tilde{e}^r \Rb\|_{L^2(\Cf_h)} \|h^{-\frac12} \lb e^r \rb\|_{L^2(\Cf_h)} \\
&\le C \|\nabla \tilde{e}^r\|_{L^2(\Ct_h)} (h^{\text{min}\{k+1, \sigma_m+1\}}\|r\|_{H^{\sigma_m + 1}(\Omega)} + \|h^{-\frac12} \lb e^r \rb\|_{L^2(\Cf_h)}),
\end{align*}
the last step is due to a discrete trace inequality and approximation property of $\Pi_S$. Finally the estimate is complete by \eqref{estimate_r1}, Lemma \ref{lower_bound} and estimates for $\Bu - \Bu_h, \Bb - \Bb_h$.

\subsection*{Step 6: Estimates for $p - p_h$} For the pressure we apply a standard $inf-sup$ argument, see \cite{BrezziFortin91, HSW09, CockburnShi12}. Since $e^p \in L^2(\Omega)$, there exists a $\Bw \in H^1_0(\Omega: \mathbb{R}^3)$ satisfies
\begin{equation}\label{infsup}
\|e^p\|_{L^2{\Omega}} \le C \frac{(e^p, \nabla \cdot \Bw)_{\Omega}}{\|\Bw\|_{H^1(\Omega)}}.
\end{equation}
It suffices to estimate the term on the right hand side. To this end, let $\boldsymbol{\Pi}^{\text{BDM}} \Bw$ denote the BDM projection of $\Bw$ into $\Bv_h \cap H(\text{div}; \Omega)$. Due to the orthogonal property of the BDM projection, we have
\begin{align*}
(e^p, \nabla \cdot \Bw)_\Omega &= (e^p, \nabla \cdot \boldsymbol{\Pi}^{\text{BDM}}\Bw)_{\Omega}+ (e^p, \nabla \cdot (\Bw - \boldsymbol{\Pi}^{\text{BDM}}\Bw))_{\Omega} \\
	& = (e^p, \nabla \cdot \boldsymbol{\Pi}^{\text{BDM}}\Bw)_{\Omega}.
\end{align*}
Notice that since $\boldsymbol{\Pi}^{\text{BDM}}\Bw \in H(\text{div}; \Omega) \rightarrow \lb \boldsymbol{\Pi}^{\text{BDM}}\Bw \rb_N = 0$ on $\Ce_h$, we have
\begin{align*}
B_h(\boldsymbol{\Pi}^{\text{BDM}}\Bw, e^p) &= - \Sigma_{K \in \Ct_h} (e^p, \nabla \cdot \boldsymbol{\Pi}^{\text{BDM}}\Bw)_{K} + \Sigma_{F\in\Cf_{h}} \langle \Lb e^p\Rb , \lb \boldsymbol{\Pi}^{\text{BDM}}\Bw \rb_{N}\rangle_{F} \\
&= - (e^p, \nabla \cdot \boldsymbol{\Pi}^{\text{BDM}}\Bw)_{\Omega}.
\end{align*}
Taking $\Bv = \boldsymbol{\Pi}^{\text{BDM}}\Bw$ in the error equation \eqref{error_eq1}, after rearranging terms, we arrive at:
\begin{align*}
(e^p, \nabla \cdot \boldsymbol{\Pi}^{\text{BDM}}\Bw)_{\Omega} & = A_h(\Bu -\Bu_h, \boldsymbol{\Pi}^{\text{BDM}}\Bw) \\
	& + \big(O_h(\Bu; \Bu, \boldsymbol{\Pi}^{\text{BDM}}\Bw)) - O_h(\mathbb{P}(\Bu_h, \Lb \Bu_h \Rb); \Bu_h, \boldsymbol{\Pi}^{\text{BDM}}\Bw)\big) \\
	& + \big( C_h(\Bb; \boldsymbol{\Pi}^{\text{BDM}}\Bw, \Bb) - C_h(\Bb_h; \boldsymbol{\Pi}^{\text{BDM}}\Bw, \Bb_h) \big) \\
	& := Y_1 + Y_2 + Y_3.
\end{align*}
For $Y_1$ we simply apply Cauchy-Schwarz inequality and discrete trace inequality to have
\begin{equation}\label{Y1}
Y_1 \le C \|\Bu - \Bu_h\|_V \|\boldsymbol{\Pi}^{\text{BDM}}\Bw\|_V.
\end{equation}
For $Y_2$, with a similar estimate as $T_6$ we can have
\begin{equation}\label{Y2}
Y_2 \le C \|\Bu - \Bu_h\|_V (\|\Bu\|_V + \|\Bu_h\|_V) \|\boldsymbol{\Pi}^{\text{BDM}}\Bw\|_V.
\end{equation}
For $Y_3$, we can apply similar estimates as for $T_7 - T_9$ in Lemma \ref{bound_T7} - Lemma \ref{bound_T9} to have
\begin{equation}\label{Y3}
Y_2 \le C \|\Bb - \Bb_h\|_C \|\boldsymbol{\Pi}^{\text{BDM}}\Bw\|_V (\|\Bb_h\|_C + \|\nabla \times \Bb\|_{H^{\sigma_m}(\Omega)}).
\end{equation}
Finally, it is well-known that we have
\begin{equation}\label{BDM}
\|\boldsymbol{\Pi}^{\text{BDM}}\Bw\|_V \le \|\Bw\|_{H^1(\Omega)}.
\end{equation}

If we combine the estimates \eqref{infsup} - \eqref{BDM}, we have
\[
\|e^p\|_{L^2(\Omega)} \le C \|\Bu - \Bu_h\|_V (1 + \|\Bu\|_V + \|\Bu_h\|_V) + C  \|\Bb - \Bb_h\|_C (\|\Bb_h\|_C + \|\nabla \times \Bb\|_{H^{\sigma_m}(\Omega)}).
\]
Finally, the estimate for $p - p_h$ is complete by the above estimate together with estimates for $\Bu - \Bu_h, \Bb - \Bb_h$ and Theorem \ref{theorem_wellposed}.

\section{A divergence-free HDG method for MHD}
In this section, we present the hybridizable discontinuous Galerkin (HDG) method for the MHD problem \eqref{mhd_eqs}. 
There are two main advantages comparing with the mixed DG method proposed in the previous sections: it provides exactly 
divergence-free velocity field. This feature makes HDG methods more robust in the sense that the convergence of the 
velocity and magnetic fields is independent of the pressure. Secondly, like all existing HDG methods for different 
problems, the only global unknowns for the system are the traces of some of the interior unknowns. 
%In the case of 3D problems involving multiple vector and scalar unknowns, 
The hybridization can significantly reduce the size of the global system. To present the HDG scheme, we first write 
the original problem \eqref{mhd_eqs} into a system of first order equations by introducing two more auxiliary 
unknowns $\mathrm{L}, \boldsymbol{w}$:

\begin{subequations}
\label{mhd_v2}
\begin{align}
\label{mhd_v2_eq1}
\mathrm{L} - \nabla \boldsymbol{u} & = 0, \quad \text{in} \Omega, \\
\label{mhd_v2_eq2}
-\nu \nabla\cdot\mathrm{L} + (\Bu\cdot \nabla) \Bu + \nabla p - \kappa (\nabla\times \Bb) \times \Bb & = \Bf \quad \text{ in } \Omega, \\
\label{mhd_v2_eq3}
\Bw - \nabla \times \boldsymbol{b} & = 0, \quad \text{in} \Omega, \\
\label{mhd_v2_eq4} 
\kappa \nu_{m} \nabla\times \Bw + \nabla r - \kappa \nabla \times (\Bu \times \Bb) & = \Bg \quad \text{ in } \Omega, \\
\label{mhd_v2_eq5}
\nabla \cdot \Bu & = 0 \quad \text{ in } \Omega, \\
\label{mhd_v2_eq6}
\nabla\cdot \Bb & = 0 \quad \text{ in } \Omega, \\
\label{mhd_v2_eq7}
\Bu & = \boldsymbol{0} \quad \text{ on } \partial\Omega, \\
\label{mhd_v2_eq8}
 \int_{\Omega} p \, d\Bx & = 0, \\
 \intertext{for the sake of implicity, we only consider the first type of boundary conditions:}
\label{mhd_v2_eq9}
\Bn \times \Bb  = \boldsymbol{0} \text{ and } r & = 0 \quad \text{ on } \partial\Omega. \\ 
 \intertext{In fact, the second type of boundary (constraint) conditions:}
\label{mhd_v2_eq10}
\Bb \cdot \Bn  = 0 \text{ and } \Bn \times (\nabla\times \Bb)  & = \boldsymbol{0} \quad \text{ on } \partial\Omega, \\  
\label{mhd_v2_eq11}
\int_{\Omega} r d\Bx & = 0, \\
 \intertext{can be treated and analyzed in a similar manner.}
\nonumber
\end{align}
\end{subequations}
{ Besides using $(\Bu_{h}, p_{h}, \Bb_{h}, r_{h}) \in \BV_h \times Q_h \times \BC_h \times S_h$ to approximate 
$(\Bu, p, \Bb, r)$ in the interior of each element,}  
we also need to introduce new unknowns $({\blue{\mathrm{L}_h}}, \Bw_h, \lambda_h, \uhat, \bhat, \rhat)$ for the HDG formulation. 
{\red{For a vector variable $\boldsymbol{d}$, on any face $F \in \mathcal{F}_h$ we define 
$\boldsymbol{d}^t|_F = \boldsymbol{d} - (\Bd \cdot \Bn_F) \Bn_F$ to be its tangential component on $F$.}}
Namely, $({\blue{\mathrm{L}_h}}, \Bw_h)$ is the approximation of $({\blue{\mathrm{L}_h}}, \Bw) = (\nabla \Bu, \nabla \times \Bb)$, 
$(\uhat, \bhat, \rhat)$ are numerical traces approximating $(\Bu, \Bb^t, r)$ on {\red{$\mathcal{F}_h$}} 
and $\lambda_h$ is another Lagrange multiplier which approximates 0 on {\red{$\partial \Ct_h$}}. 
We use the same spaces {\blue{$\BV_h \times Q_h \times \BC_h \times S_h$}} \red{ for $(\Bu_{h}, p_{h}, \Bb_{h}, r_{h})$ } as in Section 2. 
In addition, we need following spaces for  
the additional unknowns:
\begin{align*}
\mathrm{G}_h &:= P_{k}(\Ct_{h};\mathbb{R}^{3\times 3}), \quad  \BW_h := P_{k}(\Ct_{h};\mathbb{R}^{3}), 
\quad \Lambda_h := P_k(\partial \Ct_h), \\
\BM_h &:= \{\Bmu \in P_{k}(\Cf_h; \mathbb{R}^3): \, \Bmu |_{\partial\Omega} = \boldsymbol{0}\}, \\
{\red{\BM_h^{T} }} & ~ {\red{:= \{\Bmu^{t} \in \BM_h: \, \Bn \cdot \Bmu^{t}|_{F} = 0 \text{ for all $F \in \Cf_h$}, \Bmu^t|_{\partial \Omega} = 0 \}}}, \\
N_h &:= \{\shat \in P_{k+1}(\Cf_h): \, \shat|_{\partial \Omega} = 0\}.
\end{align*}

{\red{ Here we define the space $P_k(\partial \Ct_h) := \{w \in L^2(\partial \Ct_h)| w|_F \in P_k(F), \text{for all } 
F \in \partial K, K \in \Ct_h\}$.}} This space, together with the Lagrange multiplier $\lambda_h$ was introduced 
in \cite{CockburnSayas2014, FuJinQiu2017} for HDG methods with exact divergence-free velocity fields. It is worth 
to mention that on each interior face $F$ shared by two elements, $\lambda_h$ is ``double-valued'' and eventually 
it approximates $0$. {\red{Finally, we adopt the standard integral notation for HDG methods \cite{CockburnShi12}, 
for scalar-valued functions $\phi$ and $\psi$, we write
\[
 (\phi,\psi)_{\mathcal{T}_h}:=\sum_{K\in\mathcal{T}_h}(\phi,\psi)_K, \, \, \langle\phi,\psi\rangle_{\partial\mathcal{T}_h}
 :=\sum_{K\in\mathcal{T}_h}\langle\phi,\psi \rangle_{\partial K}.
\]
The above notation also applies for vector-valued functions.}}

Following the standard procedure of devising HDG methods \cite{CockburnShi12}, we first first multiplying the test functions $(\RmG,\Bv, \Bz, \Bc, q, s) \in \mathrm{G}_h \times \BV_h \times \BW_h \times \BC_h \times Q_h \times S_h$ in \eqref{mhd_v2_eq1} - \eqref{mhd_v2_eq6} and applying integrating by parts to have:
\begin{align*}
(\RmL, \RmG)_{\Ct_h} + (\Bu, \nabla \cdot \RmG)_{\Ct_h} - \langle \Bu, \RmG \Bn \rangle_{\partial \Ct_h} & = 0, \\
(\nu \RmL - p \RmI - \Bu \otimes \Bu, \nabla \Bv)_{\Ct_h} - \langle \nu \RmL \Bn - p \Bn - (\Bu \otimes \Bu) \Bn \, , \, \Bv \rangle_{\partial \Ct_h}  \\
+ \bint{\kappa \nabla \times \Bb}{ \Bv \times \Bb}  & = \bint{\Bf}{\Bv}, \\
\bint{\Bw}{\Bz} - \bint{\Bb}{\nabla \times \Bz} - \bintEh{\Bb^{t}}{\Bz \times \Bn} & = 0, \\
\bint{\kappa \nu_m \Bw}{\nabla \times \Bc} + \bint{\nabla r}{\Bc}  
 - \bint{\kappa \Bu \times \Bb}{\nabla \times \Bc} &\\
 - \bintEh{\kappa \nu_m \Bw \times \Bn + \kappa \Bn \times (\Bu \times \Bb)}{\Bc^t} & = \bint{\Bg}{\Bc}, \\
\bint{\nabla\cdot\Bu}{q} & = 0, \\
- \bint{\Bb}{\nabla s} + \bintEh{\Bb \cdot \Bn }{s} & = 0. 
\end{align*}
%Here the $\widehat{u}$ represents the trace of $u$ over $\Ce_h$ and similar convention applies to all unknowns with $\widehat{(\cdot)}$. 
Next we replace the exact solutions with the numerical approximations. The HDG method then seeks the approximation 
$(\mathrm{L}_h, \Bu_h, p_h, \Bb_h, \Bw_h, r_h, \lambda_h, \uhat, \bhat, \rhat)$ in the finite dimensional space
$\mathrm{G}_h \times \BV_h \times Q_h \times \BC_h \times \BW_h \times S_h \times \Lambda_h \times \BM_h \times \BM^T_h \times N_h$ {\blue{satisfying}}:
\begin{align*}
(\RmL_h, \RmG)_{\Ct_h} + (\Bu_h, \nabla \cdot \RmG)_{\Ct_h} - \langle \uhat, \RmG \Bn \rangle_{\partial \Ct_h} & = 0, \\
(\nu \RmL_h - p_h \RmI - \Bu_h \otimes \Bu_h, \nabla \Bv)_{\Ct_h} - \langle \widehat{\nu \RmL_h \Bn - p_h \Bn - (\uhat \otimes \Bu_h) \Bn} \, , \, \Bv \rangle_{\partial \Ct_h}  \\
+ \bint{\kappa \nabla\times \Bb_h}{ \Bv \times \Bb_h} 
+ \bintEh{\kappa (\Bb_{h}^{t} - \widehat{\Bb}_h^t)}{\Bn \times (\Bv \times \Bb_h)} & = \bint{\Bf}{\Bv}, \\
\bint{\Bw_h}{\Bz} - \bint{\Bb_h}{\nabla \times \Bz} - \bintEh{\widehat{\Bb}_h^t}{\Bz \times \Bn} & = 0, \\
\bint{\kappa \nu_m \Bw_h}{\nabla \times \Bc} + \bint{\nabla r_h}{\Bc} - \bintEh{r_{h}-\widehat{r}_h}{\Bc \cdot \Bn} 
 - \bint{\kappa \Bu_h \times \Bb_h}{\nabla \times \Bc} &\\ 
 - \bintEh{\kappa \nu_m \widehat{\Bw}_h \times \Bn + \kappa \Bn \times (\Bu_h \times \Bb_h)}{\Bc^t} & = \bint{\Bg}{\Bc}, \\
\bint{\nabla\cdot\Bu_h}{q} - \bintEh{(\Bu_{h} - \widehat{\Bu}_h) \cdot \Bn}{q} & = 0, \\
- \bint{\Bb_h}{\nabla s} + \bintEh{\widehat{\Bb}_{h}^{n} \cdot \Bn }{s}  & = 0. 
\end{align*}
for all $(\RmG,\Bv, \Bz, \Bc, q, s) \in \mathrm{G}_h \times \BV_h \times \BW_h \times \BC_h \times Q_h \times S_h$. The above system is closed by enforcing the continuity conditions for the fluxes which results to the so-called {\em transimission conditions:}
\begin{align*}
\langle  \widehat{\nu \RmL_h \Bn - p_h \Bn - (\uhat \otimes \Bu_h) \Bn} \, , \, \vhat \rangle_{\partial \Ct_h} & = 0, \\
\bintEh{\kappa \nu_m \widehat{\Bw}_{h} \times \Bn + \kappa \Bn \times (\Bu_h \times \Bb_h)}{\chat} &= 0,\\
\bintEh{\widehat{\Bb}_{h}^{n} \cdot \Bn}{\shat} & = 0, 
\end{align*}
for all $(\vhat, \chat, \shat) \in \BM_h \times \BM^T_h \times N_h$. To complete the HDG scheme, we need to define the above three numerical fluxes on $\partial \Ct_h$. To this end, for $\widehat{\Bw}_h, \widehat{\Bb}_h \cdot \Bn$ which are related with magnetic fields, we adopt the choices used in \cite{ChenQiuShi2018}:
\begin{align*}
\widehat{\Bw}_h&= \Bw_h - \frac{1}{h} (\Bn \times (\Bb_h^t - \bhat)),\\
\widehat{\Bb}_{h}^{n} \cdot \Bn& = \Bb_h \cdot \Bn + \frac{1}{h}(r_h - \rhat).
\end{align*}
With the above fluxes, the HDG scheme is complete. 
Motivated by the work in \cite{Cesmelioglu2016, CockburnSayas2014, FuJinQiu2017}, we define the convection flux to be 
\[
\widehat{\nu \RmL_h \Bn - p_h \Bn - (\uhat \otimes \Bu_h) \Bn} = \nu \RmL_h \Bn - p_h \Bn - (\uhat \otimes \Bu_h) \Bn
 - \max (\Bu_{h}\cdot \Bn, 0) (\Bu_h - \uhat) + \lambda_h \Bn.
\]
We recall that $\lambda_h \in \Lambda_h$ is a Lagrange multiplier which approximates 0 on {\red{$\partial \Ct_h$}}.
Finally, the constraint equation due to the Lagrange multiplier is given by:
\[
\bintEh{(\Bu_h - \uhat) \cdot \Bn}{\eta} = 0 \quad \text{for all} \quad \eta \in \Lambda_h.
\]
This completes the HDG scheme for the problem. Hence, we can formulate our HDG method as to seek $(\mathrm{L}_h, \Bu_h, p_h, \Bb_h, \Bw_h, r_h, \lambda_h, \uhat, \bhat, \rhat)$ in the finite dimensional space
$\mathrm{G}_h \times \BV_h \times Q_h \times \BC_h \times \BW_h \times S_h \times \Lambda_h \times \BM_h \times \BM^T_h \times N_h$ {\blue{satisfying}}
\begin{subequations}\label{HDG_eqs}
\begin{align}
\label{HDG_eq1}
(\RmL_h, \RmG)_{\Ct_h} + (\Bu_h, \nabla \cdot \RmG)_{\Ct_h} - \langle \uhat, \RmG \Bn \rangle_{\partial \Ct_h} & = 0, \\
\label{HDG_eq2}
(\nu \RmL_h - p_h \RmI - \Bu_h \otimes \Bbeta, \nabla \Bv)_{\Ct_h} - \langle \nu \RmL_h \Bn - p_h \Bn - (\uhat \otimes \Bbeta) \Bn - S_u(\Bu_h - \uhat) &+ \lambda_h \Bn \, , \, \Bv \rangle_{\partial \Ct_h}  \\
\nonumber
+ \bint{\kappa \nabla \times \Bb_h}{\Bv \times \Bd} + \bintEh{\kappa (\Bb^t_h - \bhat)}{\Bn \times (\Bv \times \Bd)} & = \bint{\Bf}{\Bv}, \\
\label{HDG_eq3}
\bint{\Bw_h}{\Bz} - \bint{\nabla \times \Bb_h}{\Bz} + \bintEh{\Bb^t_h - \bhat}{\Bz} & = 0, \\
\label{HDG_eq4}
\bint{\kappa \nu_m \Bw_h}{\nabla \times \Bc} - \bintEh{\kappa \nu_m (\Bw_h \times \Bn - \frac{1}{h} (\Bb^t_h - \bhat))}{\Bc^t}
 + \bint{\nabla r_h}{\Bc} &\\
\nonumber
- \bintEh{r_h - \rhat}{\Bc \cdot \Bn} - \bint{\kappa \Bu_h \times \Bd}{\nabla \times \Bc}
 - \bintEh{\kappa \Bn\times (\Bu_h \times \Bd)}{\Bc^t} & = \bint{\Bg}{\Bc}, \\
\label{HDG_eq5}
-\bint{\nabla \cdot \Bu_h}{q} + \bintEh{(\Bu_h - \uhat) \cdot \Bn}{q} & = 0, \\
\label{HDG_eq6}
- \bint{\Bb_h}{\nabla s} + \bintEh{\Bb_h \cdot \Bn + \frac{1}{h}(r_h - \rhat)}{s} & = 0, \\
\label{HDG_eq7}
\bintEh{(\Bu_h - \uhat) \cdot \Bn}{\eta} & = 0, \\
\intertext{together with the {\em transmission conditions}:}
\label{HDG_eq8}
\langle \nu \RmL_h \Bn - p_h \Bn - (\uhat \otimes \Bbeta) \Bn - S_u(\Bu_h - \uhat)  + \lambda_h \Bn\, , \, \vhat \rangle_{\partial \Ct_h} & = 0, \\
\label{HDG_eq9}
\bintEh{\kappa \nu_m (\Bw_h \times \Bn - \frac{1}{h}(\Bb_h^t - \bhat)) + \kappa \Bn \times (\Bu_h \times \Bd) }{\chat} &= 0,\\
\label{HDG_eq10}
\bintEh{\Bb_h \cdot \Bn + \frac{1}{h}(r_h - \rhat)}{\shat} & = 0, 
\end{align}
\end{subequations}
for all $(\RmG,\Bv,q, \Bc, \Bz, s, \eta, \vhat, \chat, \shat) \in \mathrm{G}_h \times \BV_h \times Q_h \times \BC_h \times \BW_h \times S_h \times \Lambda_h \times \BM_h \times \BM^T_h \times N_h$. On each face $F \in \partial K$ for each $K \in \Ct_h$, the stablization parameter $S_u$ is defined as: (see \cite{Cesmelioglu2016, QiuShi2016})
\[
S_u = \max \{ \Bbeta \cdot \Bn, 0\}.
\]
Finally, we set the vector fields $\Bbeta, \Bd$ as:
\begin{equation}
\label{HDG_eq11}
\Bbeta = \Bu_h, \quad \Bd = \Bb_h.
\end{equation}

\begin{remark}\label{div-free}
From \eqref{HDG_eq5} we can see that $\Bu_h \in H(\text{div}; \Omega)$. Further, since $\nabla \cdot \Bu_h \in Q_h$ we can take $q = \nabla \cdot \Bu_h$ in \eqref{HDG_eq7} to conclude that $\Bu_h$ is exactly divergence-free. 
\end{remark}

\begin{remark}
In practice, to solve the numerical solution of the above nonlinear system we need to apply Picard iteration. Namely, each iteration we need to solve a linearized system by setting $(\Bbeta, \Bd) = (\Bu_h^{n-1}, \Bb^{n-1}_h)$ where $(\Bu_h^{n-1}, \Bb^{n-1}_h)$ is the corresponding solution from the previous iteration. Thanks to the exactly divergence-free feature, we don't need to construct a divergence-free convection field $\Bbeta$ as for DG scheme in \eqref{post_process_op}. 
\end{remark}

For the purpose of presenting the error estimates, we first rewrite the scheme into a more compact form. To this end, we group \eqref{HDG_eq1}, \eqref{HDG_eq2}, \eqref{HDG_eq7} and \eqref{HDG_eq8} together; group \eqref{HDG_eq3}, \eqref{HDG_eq4} and \eqref{HDG_eq9} together; group \eqref{HDG_eq6} and \eqref{HDG_eq10} together, we can rewrite the system as: seeking $(\mathrm{L}_h, \Bu_h, p_h, \Bb_h, \Bw_h, r_h, \lambda_h, \uhat, \bhat, \rhat) \in \mathrm{G}_h \times \BV_h \times Q_h \times \BC_h \times \BW_h \times S_h \times \Lambda_h \times \BM_h \times \BM^T_h \times N_h$ satisfies:
\begin{subequations}\label{HDG_red}
\begin{align}
\label{HDG_red1}
\bint{\RmL_h}{\RmG} + B_h(\RmG; (\Bu_h, \uhat)) - B_h(\nu \RmL; (\Bv, \vhat)) + D_h(p_h; (\Bv, \vhat)) &\\
\nonumber
- I_h(\lambda_h; (\Bv, \vhat)) + I_h(\eta; (\Bu_h, \uhat))+ O_h(\Bu_h; (\Bu_h, \uhat), (\Bv, \vhat)) + C_h(\Bb_h; \Bv, (\Bb_h, \bhat)) & = \bint{\Bf}{\Bv}, \\
\label{HDG_red2}
\bint{\Bw_h}{\Bz} - M_h(\Bz; (\Bb_h, \bhat)) + M_h(\kappa \nu_m \Bw_h; (\Bc, \chat))  - C_h(\Bb_h; \Bu_h, (\Bc, \chat))& \\
\nonumber
- N_h(\Bc; (r_h, \rhat)) + \frac{1}{h} \bintEh{\kappa\nu_{m}(\Bb_h^t - \bhat)}{\Bc^t - \chat}& = \bint{\Bg}{\Bc}, \\
\label{HDG_red3}
-D_h(q;(\Bu_h, \uhat))&= 0,\\
\label{HDG_red4}
N_h(\Bb_h; (s, \shat)) + \frac{1}{h}\bintEh{r_h - \rhat}{s - \shat} &= 0,
\end{align}
\end{subequations}
for all $(\RmG,\Bv,q, \Bc, \Bz, s, \eta, \vhat, \chat, \shat) \in \mathrm{G}_h \times \BV_h \times Q_h \times \BC_h \times \BW_h \times S_h \times \Lambda_h \times \BM_h \times \BM^T_h \times N_h$. Here the operators are defined as:
\begin{align*}
B_h(\RmG; (\Bv, \vhat)) &:= -\bint{\nabla \Bv}{\RmG} + \bintEh{\Bv - \vhat}{\RmG \Bn}; \\
D_h(q; (\Bv, \vhat)) &:= -\bint{\nabla \cdot \Bv}{q} + \bintEh{(\Bv - \vhat) \cdot \Bn}{q}; \\
I_h(\eta; (\Bv, \vhat)) &:= \bintEh{\eta}{(\Bv - \vhat) \cdot \Bn};\\
M_h(\Bz; (\Bc, \chat)) &:= \bint{\nabla \times \Bc}{\Bz} - \bintEh{\Bc^t - \chat}{\Bz {{ \times \Bn }}}; \\
N_h(\Bc; (s, \shat)) & := - \bint{\nabla s}{\Bc} + \bintEh{s - \shat}{\Bc \cdot \Bn}; \\
O_h(\Bbeta; (\Bu, \widehat{\Bu}), (\Bv, \vhat)) &:= {{ - \bint{\Bu \otimes \Bbeta}{\nabla \Bv} + \bintEh{(\widehat{\Bu} \otimes \Bbeta) \Bn + S_u(\Bu - \widehat{\Bu})}{\Bv - \vhat};}}\\
C_h(\Bd; \Bv, (\Bc, \chat))&:= \bint{{{\kappa}}\nabla \times \Bc}{\Bv \times \Bd} + \bintEh{{{\kappa}}(\Bc^t - \chat)}{{\Bn \times} (\Bv \times \Bd)}.
\end{align*}

\subsection{Well-posedness and hybridization}
The well-posedness of the scheme \eqref{HDG_eqs} can be validated by two steps: first we show that the linearized scheme ($\Bbeta, \Bd$ given data) is well-posed. Then we can apply a similar Brower fixed point argument as in Section 7 for the DG scheme. Here we only present the first step in detail. The second step is very similar as in Section 7 and we leave it for readers. We define the norms used in the analysis:
\begin{subequations}\label{HDG_norms}
\begin{align}
\|\eta\|^2_{L^2(\partial \Ct_h)} &:= \sum_{K \in \Ct_h} \sum_{F \in \partial K} \|\eta\|^2_{L^2(F)}, \\
\|(\Bu, \widehat{\Bu})\|^2_{1,h} &:= \|\nabla \Bu\|^2_{L^2(\Ct_h)} + h^{-1} \|(\Bu - \widehat{\Bu})\|^2_{L^2(\partial \Ct_h)}, \\
\|(\Bc, \widehat{\Bc}^t)\|^2_{C} &:= \|\nabla \times \Bc\|^2_{L^2(\Ct_h)} + h^{-1}  \|(\Bc^t - \widehat{\Bc}^t)\|^2_{L^2(\partial \Ct_h)},\\
\|(s, \shat)\|^2_{1, h} &:= \|\nabla s\|^2_{L^2(\Ct_h)} + h^{-1} \|s - \shat\|^2_{L^2(\partial \Ct_h)}.
\end{align}
\end{subequations}

\begin{theorem}\label{well-post_HDG}
If $(\Bbeta, \Bd)$ in the system \eqref{HDG_eqs} are given data and they satisfy $\Bbeta \in H(\text{div}^0; \Omega)$, $\Bd \in L^2(\Omega; \mathbb{R}^3)$ and $\Bd|_{\partial \Ct_h} \in L^2(\partial \Ct_h; \mathbb{R}^3)$, then the linear system \eqref{HDG_eqs} has a unique solution $(\mathrm{L}_h, \Bu_h, p_h, \Bb_h, \Bw_h, r_h, \lambda_h, \uhat, \bhat, \rhat) \in \mathrm{G}_h \times \BV_h \times Q_h \times \BC_h \times \BW_h \times S_h \times \Lambda_h \times \BM_h \times \BM^T_h \times N_h$. In addition, we have $\Bu_h \in H(\text{div}^0; \Omega)$.
\end{theorem} 
\begin{proof}
With $(\Bbeta, \Bd)$ given, the system \eqref{HDG_eqs} is a square linear system. Therefore, it is suffice to show that we only have zero solution if the right hand side $(\Bf, \Bg) = ({\bf{0, 0}})$. Let $(\mathrm{L}_h, \Bu_h, p_h, \Bb_h, \Bw_h, r_h, \lambda_h, \uhat, \bhat, \rhat)$ be a solution, next we show that all components vanish under this assumption.

By Remark \ref{div-free} we know that $\Bu_h \in H(\text{div}^0, \Omega)$ and $(\Bu_h - \uhat) \cdot \Bn = 0$ on $\Cf_h$. Therefore, by Theorem 2.1 in \cite{FuJinQiu2017} we have
\begin{equation}\label{H1ineq}
\|(\Bu_h, \uhat)\|_{1,h} \le C \|\RmL_h\|_{L^2(\Ct_h)}.
\end{equation}

Taking $(\RmG,\Bv,q, \Bc, \Bz, s, \eta, \vhat, \chat, \shat) = (\nu \mathrm{L}_h, \Bu_h, p_h, \Bb_h, \kappa \nu_m \Bw_h, r_h, \lambda_h, \uhat, \bhat, \rhat)$ in \eqref{HDG_eqs} and adding all the equations, after some algebraic manipulation of the terms, we obtain the energy identity as:
\begin{align*}
\nu \|\RmL_h\|^2_{L^2(\Ct_h)} &+ \bintEh{(S_u - \frac12 \Bbeta \cdot \Bn) (\Bu_h - \uhat)}{\Bu_h - \uhat} + \kappa \nu_m \|\Bw_h\|^2_{L^2(\Ct_h)} \\
& + \frac{\kappa \nu_m}{h} \bintEh{\Bb^t_h - \bhat}{\Bb^t_h - \bhat} + \frac{1}{h} \bintEh{r_h - \rhat}{r_h - \rhat} = 0.
\end{align*}
This implies that
\begin{equation}\label{firststep}
\RmL_h = 0, \quad \Bw_h = {\bf 0}, \quad \Bb^t|_{\Cf_h} = \bhat, \quad r_h|_{\Cf_h} = \rhat.
\end{equation}
Together with \eqref{H1ineq} we have
\[
\Bu_h = {\bf 0}, \quad \uhat = {\bf 0}.
\]
Next, by \eqref{firststep} and \eqref{HDG_eq3} we have $\Bb_h \in H_0(\text{curl}^0; \Omega)$. The \eqref{HDG_eq6}, \eqref{HDG_eq10} reduces to
\[
\bint{\nabla \cdot \Bb_h}{s} = 0, \quad \bintEh{\Bh_h \cdot \Bn}{\shat} = 0,
\]
for all $(s, \shat) \in S_h \times N_h$. This implies that $\Bb_h \in H_0(\text{div}^0, \Omega)$. Therefore, by Lemma \ref{lemma_conforming_stability} we have $\Bb_h = 0$, and then $\bhat = 0$ by \eqref{firststep}.

Now the equation \eqref{HDG_eq4} reduces to:
\[
\bint{\nabla r_h}{\Bc} = 0, \quad \text{for all} \quad \Bc \in \BC_h,
\]
Simply taking $\Bc = \nabla r_h$ implies that $\nabla r_h = 0$. This shows that $r_h$ is piecewise constant over $\Ct_h$. By the fact that $r_h = \rhat$ on $\Cf_h$ and $\rhat$ vanishes on $\partial \Omega$ we conclude that 
\[
r_h = 0, \quad \rhat = 0.
\]
Finally, we need to show that $p_h, \lambda_h$ also vanish. To this end, we use the equations \eqref{HDG_eq2}, \eqref{HDG_eq8} which reduce to:
\begin{subequations}
\begin{align}
\label{reduce_1}
- \bint{p_h}{\nabla \cdot \Bv} + \bintEh{p_h \Bn - \lambda_h \Bn}{\Bv} &= 0,\\
\label{reduce_2}
\bintEh{p_h \Bn - \lambda_h \Bn}{\vhat} &= 0.
\end{align}
\end{subequations}
for all $(\Bv, \vhat) \in \BV_h \times \BM_h$. Since $p_h \in L^2_0(\Omega)$, there exists a $\boldsymbol{\gamma} \in H(\text{div}; \Omega)$ such that
\[
\nabla \cdot \boldsymbol{\gamma} = p_h \quad \text{in $\Omega$}, \quad \boldsymbol{\gamma} \cdot \Bn = 0 \quad \text{on $\partial \Omega$}.
\]
Taking $\Bv = \boldsymbol{\Pi}^{\text{BDM}} \boldsymbol{\gamma}$ in \eqref{reduce_1} we have
\[
\bint{p_h}{\nabla \cdot \boldsymbol{\Pi}^{\text{BDM}} \boldsymbol{\gamma}} + \bintEh{p_h \Bn - \lambda_h \Bn}{\boldsymbol{\Pi}^{\text{BDM}} \boldsymbol{\gamma}} = 0,
\]
by the property of the BDM-projection and \eqref{reduce_2} we have $p_h = 0$. Now \eqref{reduce_1} becomes:
\[
\bintEh{\lambda_h}{\Bv \cdot \Bn} = 0 \quad \text{for all} \quad \Bv \in \BV_h.
\]
On each element $K$ we can always find a $\Bv \in P_k(K; \mathbb{R}^3)$ such that $\Bv \cdot \Bn|_{\partial K} = \lambda_h$ since this is part of the degree of freedoms of BDM element of degree k. With this $\Bv$ in the above equation we can conclude that $\lambda_h = 0$. This completes the proof. 
\end{proof}

Like all HDG methods, the above scheme can be hybridized so that the only globally coupled unknowns are $(\uhat, \bhat, \rhat, \bar{p}_h)$ where $\bar{p}_h$ approximates the {\em average} pressure within each element. To be more specific, let us decompose the space $Q_h = Q_h^{\perp} \oplus \bar{Q}_h$, where
\[
Q^{\perp}_h :=\{q \in Q_h: q|_K \in L_{0}^{2}(K), \; \forall\; K \in \Ct_h\}, \quad \bar{Q}_h = \{q \in Q_h: q|_K \in P_0(K), \; \forall \; K \in \Ct_h\}. 
\]
Therefore, we can write $p_h = p^{\perp}_h + \bar{p}_h$. In practice, we use \eqref{HDG_eq1} - \eqref{HDG_eq7} as local solvers so that $\RmL_h, \Bu_h,p^{\perp}_h, \Bb_h, \Bw_h, \lambda_h$ are elinimated from the glocal system. The global unknowns are $(\uhat, \bhat, \rhat, \bar{p}_h)$ coupled by the transmission conditions \eqref{HDG_eq8} - \eqref{HDG_eq10}. In other word, the size of the global system is the dimension of the space $\BM_h \times \BM^{T}_h \times N_h \times \bar{Q}_h$. The hybridized system reduces the global degrees of freedom significantly, which in turn to make the method more efficient and compatetive with mixed DG and conforming mixed methods. For more details on hybridization of the method we refer to \cite{CockburnShi14, FuJinQiu2017} and references therein.

\subsection{Error equations}
In this section, we will derive the error equations based on \eqref{HDG_red}. We begin by some notations that will be used in the analysis. For a generic unknown $\mathcal{U}$ with its numerical counterpart $\mathcal{U}_h$ we can split the error as:
\begin{align*}
& \mathcal{U} - \mathcal{U}_h = \delta_{\mathcal{U}} + e^{\mathcal{U}}, \\
\text{where} \quad & 
\delta_{\mathcal{U}} := \mathcal{U} - {\Pi} \mathcal{U}, \quad e^{\mathcal{U}}:= {\Pi}  \mathcal{U} - \mathcal{U}_h.
\end{align*}
Here $\Pi \mathcal{U}$ is some projection/interpolant of $\mathcal{U}$ into the discrete polynomial space which we specify next. For $(\Bu, p, \Bc, r)$ we use the same projections $(\boldsymbol{\Pi}^{RT}, \Pi_Q, \boldsymbol{\Pi}_C, \Pi_S)$ as for the mixed DG method in Section \ref{estimate_DG}. For the unknowns $(\RmL, \Bw)$ we use the standard $L^2-$projections $(\mathrm{\Pi}_{{G}} \RmL, \boldsymbol{\Pi}_{W} \Bw)$. For the boundary unknowns $(\Bu|_{\Cf_h}, \Bb^t|_{\Cf_h}, r|_{\Cf_h})$ we simply use the traces of the projections so that on $\Cf_h$ we define:
\begin{align*}
&e^{\widehat{\Bu}}:= \boldsymbol{\Pi}^{RT} \Bu - \uhat, \quad e^{\widehat{\Bb}, t} := (\boldsymbol{\Pi}_{C} \Bb)^t - \bhat,\quad e^{\widehat{r}}:= \Pi_{S} r - \rhat, \\
&\delta_{\widehat{\Bu}}:= \Bu - \boldsymbol{\Pi}^{RT} \Bu, \quad \delta^t_{\widehat{\Bb}} := \Bb^t - (\boldsymbol{\Pi}_{C} \Bb)^t,\quad \delta_{\widehat{r}}:= r - \Pi_{S} r.
\end{align*}
Immediately, on $\Cf_h$ we have
\[
\delta_{\Bu} \cdot \Bn = \delta_{\widehat{\Bu}} \cdot \Bn, \quad \delta^t_{\Bb} = \delta^t_{\widehat{\Bb}}, \quad \delta_{r} = \delta_{\widehat{r}}.
\]
Now we are ready to derive the error equations for the method. Notice that the exact solution $(\Bu, \RmL, p, \Bb, \Bw, r, 0, \Bu|_{\Cf_h}, \Bb^t|_{\Cf_h}, r|_{\Cf_h})$ also satisfies the discrete system \eqref{HDG_red}, if we subtract these two systems and use the splitting discussed above, we can obtain the error equations as:
\begin{subequations}\label{HDG_err}
\begin{align}
\label{HDG_err1}
&\bint{e^{\RmL}}{\RmG} + B_h(\RmG; (e^{\Bu}, e^{\widehat{\Bu}})) - B_h(\nu e^{\RmL}; (\Bv, \vhat)) + D_h(e^p; (\Bv, \vhat)) \\
\nonumber
&- I_h(\lambda_h; (\Bv, \vhat)) + I_h(\eta; (e^{\Bu}, e^{\widehat{\Bu}}))+ (O_h(\Bu; (\Bu, \Bu), (\Bv, \vhat)) - O_h(\Bu_h; (\Bu_h, \uhat), (\Bv, \vhat))) \\
\nonumber
&+ (C_h(\Bb; \Bv, (\Bb, \Bb^t)) - C_h(\Bb_h; \Bv, (\Bb_h, \bhat))) \\ 
\nonumber
&= - \bint{\delta_{L}}{\RmG} -  B_h(\RmG; (\delta_{\Bu}, \delta_{\widehat{\Bu}})) + B_h(\nu \delta_{\RmL}; (\Bv, \vhat)) - D_h(\delta_p; (\Bv, \vhat)), \\
\label{HDG_err2}
&\bint{e^\Bw}{\Bz} - M_h(\Bz; (e^\Bb, e^{\widehat{\Bb}, t})) + M_h(\kappa \nu_m e^\Bw; (\Bc, \chat)) - N_h(\Bc; (e^r, e^{\widehat{r}}))\\
\nonumber
&  + \frac{1}{h}\bintEh{\kappa\nu_{m}((e^\Bb)^t - e^{\widehat{\Bb},t})}{\Bc^t - \chat}
-(C_h(\Bb; \Bu, (\Bc, \chat)) - C_h(\Bb_h; \Bu_h, (\Bc, \chat)))  \\
\nonumber
&= -\bint{\delta_\Bw}{\Bz} + M_h(\Bz; (\delta_\Bb, \delta^t_{\widehat{\Bb}})) - M_h(\kappa \nu_m \delta_\Bw; (\Bc, \chat)) + N_h(\Bc; (\delta_r, \delta_{\widehat{r}})) \\
\label{HDG_err3}
&-D_h(q;(e^\Bu, e^{\widehat{\Bu}})) = D_h(q;(\delta_{\Bu}, \delta_{\widehat{\Bu}})),\\
\label{HDG_err4}
&N_h(e^\Bb; (s, \shat)) + \frac{1}{h}\bintEh{e^r - e^{\widehat{r}}}{s - \shat} = - N_c(\delta_{\Bb}; (s, \shat)),
\end{align}
\end{subequations}
for all $(\RmG,\Bv,q, \Bc, \Bz, s, \eta, \vhat, \chat, \shat) \in \mathrm{G}_h \times \BV_h \times Q_h \times \BC_h \times \BW_h \times S_h \times \Lambda_h \times \BM_h \times \BM^T_h \times N_h$. 

\subsubsection{Energy identity}
Similar as the analysis for the mixed DG method in Section 8, we have the following energy identity and several auxiliary estimates for the final error estimates:

\begin{lemma}\label{HDG_energy}
The projection of the errors satisfy:
\begin{itemize}
\item[(1)] $e^{\Bu} \in H(\text{div}^0; \Omega) \cap \BV_h$ and $(e^{\Bu} - e^{\widehat{\Bu}}) \cdot \Bn = 0$ on $\partial \Ct_h$.
\item[(2)] we have
\[
\|(e^{\Bu}, e^{\widehat{\Bu}})\|_{1,h} \le C \|e^{\RmL}\|_{L^2(\Omega)}.
\]
\item[(3)] $\bint{e^{\Bb}}{\nabla s} = 0$ for all $s \in H_{0}^{1}(\Omega) \cap S_h$, and it holds:
\[
\|e^{\Bb}\|_{L^2(\Omega)} \le C \|e^{\Bb}\|_{L^3(\Omega)} \le C \|(e^{\Bb}, e^{\widehat{\Bb}, t})\|_C.
\] 
\item[(4)] we have
\[
\|e^{\Bw}\|_{L^2(\Omega)} \le C (\|\nabla \times \delta_{\Bb}\|_{\Ct_h} + \|(e^{\Bb}, e^{\widehat{\Bb}, t})\|_C).
\] 
\item[(5)] In addition, we have the energy identity:
\begin{subequations}\label{energy_idHDG}
\begin{align}
\nonumber
& \nu \|e^{\RmL}\|^2_{L^2(\Omega)} + \kappa \nu_m \|e^{\Bw}\|^2_{L^2(\Omega)}
 + \frac{1}{h}\kappa\nu_{m}\|(e^{\Bb})^t - e^{\widehat{\Bb}, t}\|^2_{L^2(\partial \Ct_h)} 
 + \frac{1}{h}\|e^r - e^{\widehat{r}}\|^2_{L^2(\partial \Ct_h)} \\ 
\nonumber
&+ O_h(\Bu_h; (e^{\Bu}, e^{\widehat{\Bu}}), (e^{\Bu}, e^{\widehat{\Bu}})) \\
\nonumber
=& \big( -B_h(\nu e^{\RmL}; (\delta_{\Bu}, \delta_{\widehat{\Bu}})) + B_h(\nu \delta_{\RmL}; (e^{\Bu}, e^{\widehat{\Bu}}))\big) \\
\nonumber
& + \big( M_h(\kappa \nu_m e^{\Bw}; (\delta_{\Bb}, \delta^t_{\widehat{\Bb}})) - M_h(\kappa \nu_m \delta_{\Bw}; (e^{\Bb}, e^{\widehat{\Bb}, t})) \big) \\
\nonumber
&+ \big( N_h(e^{\Bb}; (\delta_{r}, \delta_{\widehat{r}})) - N_h(\delta_{\Bb}; (e^{r}, e^{\widehat{r}}))\big) \\
\nonumber
& - \big( O_h(\delta_{\Bu}; (\Bu, \Bu); (e^{\Bu}, e^{\widehat{\Bu}}) + O_h(e_{\Bu}; (\Bu, \Bu); (e^{\Bu}, e^{\widehat{\Bu}}) + O_h({\Bu_h}; (\delta_\Bu, \delta_{\widehat{\Bu}}); (e^{\Bu}, e^{\widehat{\Bu}}) ) \big) \\
\nonumber
& + \big( C_h(\Bb; \Bu, (e^{\Bb}, e^{\widehat{\Bb}, t})) - C_h(\Bb_h; \Bu_h, (e^{\Bb}, e^{\widehat{\Bb}, t})) + C_h(\Bb_h; e^\Bu, (\Bb_h, \bhat)) - C_h(\Bb; e^\Bu, (\Bb, \Bb^t)) \big) \\
\nonumber
:=& T_1 + T_2 + \dots + T_5.
\end{align}
\end{subequations}
\end{itemize}
\end{lemma}

\begin{proof}
(1) is direct consequence of the fact that $\Bu_h \in H(\text{div}^0; \Omega)$ and $(\Bu_h - \uhat) \cdot \Bn = 0$ on $\partial \Ct_h$ and the conforming property of the Raviart-Thomas projection. (2) is a direct consequence of Theorem 2.1 in \cite{FuJinQiu2017} and (1) in this lemma.

For (3), in \eqref{HDG_red4} if we restrict $s \in H_{0}^1(\Omega) \cap S_h$ and $\shat = s$ on $\Cf_h$, we have
\[
\bint{\Bb_h}{\nabla s} = 0, \quad \forall s \in H_{0}^{1}(\Omega) \cap S_h,
\]
then we also have 
\[
\bint{e^\Bb}{\nabla s} = 0, \quad \forall s \in H_{0}^{1}(\Omega) \cap S_h,
\]
by the property of projection $\boldsymbol{\Pi}_C$ in Lemma \ref{proj_C_props}. Further, this means that $e^{\Bb}$ satisfies the condition in Theorem \ref{theorem_stability}, therefore, we have
\begin{align*}
\|e^{\Bb}\|_{L^2(\Omega)} \le C \|e^{\Bb}\|_{L^3(\Omega)} &\le C (\|\nabla \times e^{\Bb}\|_{L^2(\Ct_h)} + \|  h^{-1/2}\llbracket (e^\Bb)^t \rrbracket \|_{L^2(\Cf_h)}),\\
& \le C \|(e^{\Bb}, e^{\widehat{\Bb}, t})\|_{C}.
\end{align*}
The last inequality is due to the fact that $e^{\widehat{\Bb}, t}$ is single valued on $\partial \Ct_h$ and a triangle inequality. 

To prove (4), we take $(\Bz, \Bc, \chat) = (e^{\Bw}, {\bf 0}, {\bf 0}^t)$ in error equation \eqref{HDG_err2} we have
\[
\|e^{\Bw}\|^2_{L^2(\Omega)} = M_h(e^{\Bw}; (e^{\Bb}, e^{\widehat{\Bb}, t})) + M_h(e^{\Bw}; (\delta_{\Bb}, \delta_{\widehat{\Bb}, t})).
\]
The estimate follows from Cauchy-Schwartz inequality and inverse inequality. 

Finally, to establish (5), we take $$(\RmG,\Bv,q, \Bc, \Bz, s, \eta, \vhat, \chat, \shat) = (\nu e^{\RmL}, e^{\Bu}, e^p, e^{\Bb}, \kappa \nu_m e^{\Bw}, e^r, \lambda_h, e^{\widehat{\Bu}}, e^{\widehat{\Bb},t}, e^{\widehat{r}})$$ in the error equations \eqref{HDG_err} and adding all the error equations, after some algebraic rearrangement of the terms we have
\begin{align*}
& \nu \|e^{\RmL}\|^2_{L^2(\Omega)} + \kappa \nu_m \|e^{\Bw}\|^2_{L^2(\Omega)}
 + \frac{1}{h}\kappa\nu_{m}\|(e^{\Bb})^t - e^{\widehat{\Bb}, t}\|^2_{L^2(\partial \Ct_h)} 
 + \frac{1}{h}\|e^r - e^{\widehat{r}}\|^2_{L^2(\partial \Ct_h)} \\ 
&+ O_h(\Bu_h; (e^{\Bu}, e^{\widehat{\Bu}}), (e^{\Bu}, e^{\widehat{\Bu}})) \\
=& - D_h(\delta_p; (e^{\Bu}, e^{\widehat{\Bu}})) + D_h(e_p; (\delta_{\Bu}, \delta_{\widehat{\Bu}})) + T_1 + T_2 + \dots + T_5.
\end{align*}
We only need to show that the first two terms vanish in the last equality. By (1), we have $D_h(\delta_p; (e^{\Bu}, e^{\widehat{\Bu}})) = 0$ and 
\[
D_h(e_p; (\delta_{\Bu}, \delta_{\widehat{\Bu}})) = \bint{\nabla \cdot \delta_{\Bu}}{e^p} = 0,
\]
by the commuting property: $\nabla \cdot \boldsymbol{\Pi}^{RT} \Bu = \Pi_{Q} (\nabla \cdot \Bu)$ and $e^p \in Q_h$. This completes the proof.
\end{proof}

\subsubsection{Error estimates}
Comparing energy identity for the HDG method in \ref{energy_idHDG} and the one for the mixed DG method in \eqref{energy_identity}, we can see the the right hand side of \ref{energy_idHDG} is independent of pressure $p$. This is due to the fact that the HDG method provides exactly divergence-free velocity. Thanks to this feature, the error estimate for the energy norm is independent of the regularity of the pressure. Consequently, for the error estimates, we can further relax the regularity assumption \eqref{regularity} to be:
\begin{subequations}\label{regularity_HDG}
\begin{align}
&(\Bu, p) \; \in \; H^{\sigma + 1}(\Omega; \mathbb{R}^{3}) \times H^{\sigma_p}(\Omega), \\
& (\Bb, \nabla \times \Bb, r) \; \in \; H^{\sigma_m}(\Omega; \mathbb{R}^{3}) 
\times H^{\sigma_m}(\Omega; \mathbb{R}^{3}) \times H^{\sigma_m + 1}(\Omega),
\end{align}
\end{subequations}
for $\sigma, \sigma_m > \frac12, \sigma_p > 0$.

Now we are ready to state our main convergence results for the HDG method:
\begin{theorem}\label{errors_ub_hdg}
Let $(\Bu, \RmL, p,  \Bb, \Bw, r)$ be the exact solution of the system \eqref{mhd_v2}, and $(\mathrm{L}_h, \Bu_h, p_h,$    $\Bb_h, \Bw_h, r_h, \lambda_h, \uhat, \bhat, \rhat)$ be the solution of the HDG method (\ref{HDG_eqs}). 
With the same assumption as in Theorem \ref{theorem_wellposed}, in addition with the regularity assumption \eqref{regularity_HDG} 
and that $\dfrac{1}{\min(\nu, \nu_{m})}\Vert \Bu\Vert_{H^{1}(\Omega)}$ and 
$\dfrac{1}{\sqrt{\nu \kappa \nu_{m}}}\Vert \nabla\times \Bb \Vert_{L^{2}(\Omega)}$  
are small enough, then we have
\begin{align*}
\nu^{\frac12}\|(\Bu - \Bu_h, \Bu - \uhat)\|_{1, h} &+ \kappa^{\frac12} \nu_m^{\frac12}\|(\Bb - \Bb_h, \Bb^t - \bhat)\|_C + \|(r -r_h, r - \rhat)\|_{1, h} \\
& \le \mathcal{C} h^{\text{min}\{k, \sigma, \sigma_m\}}\Big(\|\Bu\|_{H^{\sigma+1}(\Omega)}  
+ \|\Bb\|_{H^{\sigma_m}(\Omega)} + \|\nabla \times \Bb\|_{H^{\sigma_m}(\Omega)} \\
&+ \|r\|_{H^{\sigma_m + 1}(\Omega)} + (\|\Bu\|_{H^{\sigma+1}(\Omega)} 
 + \|\nabla \times \Bb\|_{H^{\sigma_m}(\Omega)}) \|\Bb\|_{H^{\sigma_m}(\Omega)} \Big), \\
 \|p - p_h\|_{L^2(\Omega)} & \le C h^{\min{\sigma_p, k}} \|p\|_{H^{\sigma_p}(\Omega)}  \\
 & + \mathcal{C} h^{\text{min}\{k, \sigma, \sigma_m\}}\Big(\|\Bu\|_{H^{\sigma+1}(\Omega)}  
+ \|\Bb\|_{H^{\sigma_m}(\Omega)} + \|\nabla \times \Bb\|_{H^{\sigma_m}(\Omega)} \\
&+ \|r\|_{H^{\sigma_m + 1}(\Omega)} + (\|\Bu\|_{H^{\sigma+1}(\Omega)} 
 + \|\nabla \times \Bb\|_{H^{\sigma_m}(\Omega)}) \|\Bb\|_{H^{\sigma_m}(\Omega)} \Big)
\end{align*}
here $\mathcal{C}$ depends on the physical parameters $\kappa, \nu, \nu_m$ and the external forces $\Bf, \Bg$ 
but is independent of mesh size $h$.
\end{theorem}
The proof of the result is based on Lemma \ref{HDG_energy} and mimicking the proofs for the mixed DG method in Section 8. We leave the details for readers who are interested. 

\section{Concluding Remarks}

In this paper we rigorously analyzed a DG scheme for MHD problem. With standard regularity assumption on the exact solution, we proved that the numerical solution converges to the exact solution optimally for all unknowns in the energy norm. To the best of our knowledge, it is the first analysis dedicated to DG methods for nonlinear MHD problems. In order to make the method more attractive and competitive, we also derive and analyze the first HDG scheme for the problem with several unique features in addition to those for the mixed DG method, including but not limited to (1) It reduces the size of the global system significantly by the hybridization technique, (2) It provides exactly divergence-free velocity fields, (3) The errors for the velocity and magnetic fields are independent of the regularity of the pressure. The issues related with implementation of the mixed DG and HDG method are subjected to ongoing work.

\appendix

\section{Proof of Lemma~\ref{lemma_conforming_stability2}}
\label{appendix1}
In this section, we give the proof for Lemma~\ref{lemma_conforming_stability2}. 

Since $\nabla\cdot (\nabla\times \tilde{\Bb}_{h}) = 0$, there is a unique 
$\boldsymbol{\sigma} \in H_{0}(\text{curl},\Omega)$ satisfying 
\begin{align*}
\nabla\times (\nabla\times \boldsymbol{\sigma}) & = \nabla\times \tilde{\Bb}_{h}  \quad \text{ in } \Omega,\\
\nabla \cdot \boldsymbol{\sigma} & = 0  \quad \text{ in } \Omega.
\end{align*}
It is well known that 
\begin{align*}
\Vert \boldsymbol{\sigma}\Vert_{L^{2}(\Omega)} + \Vert \nabla\times \boldsymbol{\sigma} \Vert_{L^{2}(\Omega)} 
\leq C \Vert \nabla\times \tilde{\Bb}_{h} \Vert_{L^{2}(\Omega)}. 
\end{align*}
Obviously, 
\begin{align*}
\nabla\cdot (\nabla\times \boldsymbol{\sigma})=0 \quad \text{ in }\Omega,\qquad 
(\nabla\times \boldsymbol{\sigma})\cdot \Bn = 0 \quad \text{ on } \partial\Omega.  
\end{align*}
So, according to \cite[Theorem~$4.1$]{Hiptmair02}, there is $\delta\in (0,\frac{1}{2}]$ such that 
\begin{align}
\label{curl_sigma_ineq1}
\Vert \nabla\times \boldsymbol{\sigma} \Vert_{H^{1/2+\delta}(\Omega)} 
\leq C \Vert \nabla\times \boldsymbol{\sigma}\Vert_{L^{2}(\Omega)} 
\leq C \Vert \nabla\times \tilde{\Bb}_{h} \Vert_{L^{2}(\Omega)}. 
\end{align}

We recall that $\boldsymbol{\Pi}_{N}$ is the N{\'e}d{\'e}lec projection onto $H(\text{curl},\Omega)\cap \BC_{h}$, 
and denote by $\boldsymbol{\Pi}^{BDM}$ the BDM projection onto $H(\text{div}, \Omega) \cap P_{k-1}(\Ct_{h};\mathbb{R}^{3})$. 
Thus 
\begin{align*}
\nabla\times \boldsymbol{\Pi}_{N} (\nabla\times \boldsymbol{\sigma}) 
= \boldsymbol{\Pi}^{BDM} (\nabla\times (\nabla\times \boldsymbol{\sigma}))
= \boldsymbol{\Pi}^{BDM}(\nabla\times \tilde{\Bb}_{h}) = \nabla\times \tilde{\Bb}_{h}. 
\end{align*}
So, there is $g_{h}\in H^{1}(\Omega) \cap P_{k+1}(\Ct_{h})$ such that 
\begin{align*}
\boldsymbol{\Pi}_{N} (\nabla\times \boldsymbol{\sigma}) - \tilde{\Bb}_{h} = \nabla g_{h}.
\end{align*}
Since $(\nabla\times \boldsymbol{\sigma}, \nabla g_{h})_{\Omega} = (\tilde{\Bb}_{h}, \nabla g_{h})_{\Omega} = 0$, we have that 
\begin{align*}
& \Vert \boldsymbol{\Pi}_{N} (\nabla\times \boldsymbol{\sigma}) - \tilde{\Bb}_{h}\Vert_{L^{2}(\Omega)}^{2}
= (\boldsymbol{\Pi}_{N} (\nabla\times \boldsymbol{\sigma}) - \tilde{\Bb}_{h}, \nabla g_{h})_{\Omega} \\
= & (\boldsymbol{\Pi}_{N} (\nabla\times \boldsymbol{\sigma}) - \nabla\times \boldsymbol{\sigma}, \nabla g_{h})_{\Omega} 
= (\boldsymbol{\Pi}_{N} (\nabla\times \boldsymbol{\sigma}) - \nabla\times \boldsymbol{\sigma}, 
\boldsymbol{\Pi}_{N} (\nabla\times \boldsymbol{\sigma}) - \tilde{\Bb}_{h})_{\Omega} \\
\leq & \Vert \boldsymbol{\Pi}_{N} (\nabla\times \boldsymbol{\sigma}) - \nabla\times \boldsymbol{\sigma}\Vert_{L^{2}(\Omega)} 
\Vert \boldsymbol{\Pi}_{N} (\nabla\times \boldsymbol{\sigma}) - \tilde{\Bb}_{h} \Vert_{L^{2}(\Omega)}. 
\end{align*}
By (\ref{curl_sigma_ineq1}), we have that 
\begin{align}
\label{curl_sigma_ineq2}
\Vert \nabla\times \boldsymbol{\sigma} - \tilde{\Bb}_{h}\Vert_{L^{2}(\Omega)} 
\leq C \Vert \boldsymbol{\Pi}_{N} (\nabla\times \boldsymbol{\sigma}) - \nabla\times \boldsymbol{\sigma}\Vert_{L^{2}(\Omega)} 
\leq C h^{1/2+\delta} \Vert \nabla\times \tilde{\Bb}_{h} \Vert_{L^{2}(\Omega)}. 
\end{align}

We denote by $\boldsymbol{\Pi}_{h}$ the $L^{2}$-orthogonal projection onto $P_{k}(\Ct_{h};\mathbb{R}^{3})$. 
Since $\tilde{\Bb}_{h} \in P_{k}(\Ct_{h};\mathbb{R}^{3})$, then $\boldsymbol{\Pi}_{h} \tilde{\Bb}_{h} = \tilde{\Bb}_{h}$. 
So, we have that 
\begin{align*}
\Vert \tilde{\Bb}_{h}\Vert_{L^{3}(\Omega)} = \Vert \boldsymbol{\Pi}_{h} \tilde{\Bb}_{h}\Vert_{L^{3}(\Omega)} 
\leq \Vert \boldsymbol{\Pi}_{h} (\tilde{\Bb}_{h} - \nabla\times \boldsymbol{\sigma})\Vert_{L^{3}(\Omega)} 
+ \Vert \boldsymbol{\Pi}_{h} \nabla\times \boldsymbol{\sigma} \Vert_{L^{3}(\Omega)}. 
\end{align*}
By scaling argument and (\ref{curl_sigma_ineq2}), we have that 
\begin{align*}
& \Vert \boldsymbol{\Pi}_{h} (\tilde{\Bb}_{h} - \nabla\times \boldsymbol{\sigma})\Vert_{L^{3}(\Omega)}  
\leq C h^{-1/2} \Vert \boldsymbol{\Pi}_{h} (\tilde{\Bb}_{h} - \nabla\times \boldsymbol{\sigma})\Vert_{L^{2}(\Omega)} \\
\leq & C h^{-1/2}\Vert \tilde{\Bb}_{h} - \nabla\times \boldsymbol{\sigma} \Vert_{L^{2}(\Omega)} 
\leq C h^{\delta} \Vert \nabla \times \tilde{\Bb}_{h} \Vert_{L^{2}(\Omega)}. 
\end{align*}
Again by scaling argument, It is easy to see there is a constant $C>0$ such that 
\begin{align*}
\Vert \boldsymbol{\Pi}_{h} \Bv\Vert_{L^{3}(\Omega)} \leq C \Vert \Bv\Vert_{L^{3}(\Omega)}, 
\qquad \forall \Bv \in L^{3}(\Omega).
\end{align*}
Thus, by (\ref{curl_sigma_ineq1}), we have that 
\begin{align*}
\Vert \boldsymbol{\Pi}_{h} \nabla\times \boldsymbol{\sigma}\Vert_{L^{3}(\Omega)} 
\leq C \Vert \nabla\times \boldsymbol{\sigma} \Vert_{L^{3}(\Omega)} 
\leq C \Vert \nabla\times \boldsymbol{\sigma} \Vert_{H^{1/2+\delta}(\Omega)} 
\leq C \Vert \nabla \times \tilde{\Bb}_{h} \Vert_{L^{2}(\Omega)}.
\end{align*}
So, we can conclude that 
\begin{align*}
\Vert \tilde{\Bb}_{h}\Vert_{L^{3}(\Omega)} \leq C \Vert \nabla \times \tilde{\Bb}_{h} \Vert_{L^{2}(\Omega)}. 
\end{align*}

\section{Proof of Lemma~\ref{lemma_curl_interpolation}}
\label{appendix2}
In this section, we give the proof for Lemma~\ref{lemma_curl_interpolation}, which 
highly mimics that of \cite[Proposition~$4.5$]{Houston}. 

For any element $K\in \Ct_{h}$, we use $\Ce(K)$ and $\Cf(K)$ to denote the sets of edges and faces of $K$.  
For any face $f\in \Cf_{h}$, we denote by $\Ce(f)$ the set of the edges of $f$. 
The functions $\{ \boldsymbol{\varphi}_{K, e}^{i} \}_{i=1}^{N_{e}}$, $\{ \boldsymbol{\varphi}_{K,f}^{i}\}_{i=1}^{N_{f}}$ 
and $\{ \boldsymbol{\varphi}_{K,b}^{i} \}_{i=1}^{N_{b}}$ are Lagrange basis functions on $P_{k}(K;\mathbb{R}^{3})$,  
which are introduced in step $2$ of the proof of \cite[Proposition~$4.5$]{Houston}. In fact, 
$\{ \boldsymbol{\varphi}_{K, e}^{i} \}_{i=1}^{N_{e}}$ are basis functions associated with edge degrees of freedom 
for any edge edge $e\in \Ce(K)$; $\{ \boldsymbol{\varphi}_{K,f}^{i}\}_{i=1}^{N_{f}}$ are basis functions 
associated with face degrees of freedom for any face $f \in \Cf(K)$; and $\{ \boldsymbol{\varphi}_{K,b}^{i} \}_{i=1}^{N_{b}}$ 
are basis functions associated with volume degrees of freedom of $K$. 

Like $(A.1)$ in step $2$ of the proof of \cite[Proposition~$4.5$]{Houston}, for any $K \in \Ct_{h}$, we have that 
\begin{align*}
\Bb_{h} |_{K} = \Sigma_{e\in \mathbf{E}(K)} \Sigma_{i=1}^{N_{e}} b_{K,e}^{i} \boldsymbol{\varphi}_{K, e}^{i} 
+ \Sigma_{f\in \Ce (K)} \Sigma_{i=1}^{N_{f}} b_{K,f}^{i} \boldsymbol{\varphi}_{K, f}^{i} 
+ \Sigma_{i=1}^{N_{b}} b_{K,b}^{i} \boldsymbol{\varphi}_{K, b}^{i}. 
\end{align*}
Similar to step $5$ of the proof of \cite[Proposition~$4.5$]{Houston}, We construct $\tilde{\Bb}_{h} 
\in H(\text{curl}, \Omega)\cap \BC_{h}$ to be the unique function whose edge degrees freedom are 
\begin{align*}
\tilde{b}_{K, e}^{i} = \dfrac{1}{\vert N(e)\vert} \Sigma_{K^{\prime}\in N(e)} b_{K^{\prime}, e}^{i},
\end{align*}
$i=1,\cdots, N_{e}$ where $N(e)$ is the set of all elements sharing the edge $e$; 
whose face degrees of freedom are 
\begin{align*}
\tilde{b}_{K, f}^{i} = \dfrac{1}{\vert N(f) \vert} \Sigma_{K^{\prime}\in N(f)} b_{K^{\prime}, f} ^{i}, 
\end{align*}
$i=1,\cdots, N_{f}$ where $N(f)$ is the set of all elements sharing the face $f$; 
and whose volume degrees of freedom are 
\begin{align*}
\tilde{b}_{K,b}^{i} = b_{K, b}^{i},\qquad i = 1,\cdots, N_{b}.
\end{align*}

From the bound in $(A.2)$ in step $3$ of the proof of \cite[Proposition~$4.5$]{Houston}, we have that 
\begin{align*}
\Vert \Bb_{h} - \tilde{\Bb}_{h} \Vert_{L^{2}(K)}^{2} 
\leq C h_{K} \big[ \Sigma_{e\in \Ce(K)} \Sigma_{i=1}^{N_{e}} (b_{K,e}^{i} - \tilde{b}_{K,e}^{i})^{2} 
+ \Sigma_{f\in \Cf(K)}\Sigma_{i=1}^{N_{f}} (b_{k,f}^{i} - \tilde{b}_{K,f}^{i})^{2}\big].
\end{align*}
For any edge $e\in \Ce(K)$, we denote by $\Cf(e)$ the set of the faces sharing the edge $e$. 
For $f\in \Cf(e)$, we denote by $K_{f}$ and $K^{\prime}_{f}$ the elements sharing the face $f$. 
It is easy to see that $K_{f} = K^{\prime}_{f}$ if $f$ is a face on $\partial\Omega$. 
%If $e$ is an interior edge, then due to step $5$ of the proof of \cite[Proposition~$4.5$]{Houston}, we have that 
%\begin{align*}
%\Sigma_{i=1}^{N_{e}} (b_{K,e}^{i} - \tilde{b}_{K,e}^{i})^{2} \leq C 
%\Sigma_{f\in \Cf(e)} \int_{f} \vert  \llbracket\Bb_h \rrbracket_{T} \vert^{2} ds.
%\end{align*}
%If $e$ is an edge on $\partial\Omega$, 
By the construction of $\tilde{\Bb}_{h}$,  
Cauchy-Schwarz inequality and shape-regularity assumption, we have that 
\begin{align*}
\Sigma_{i=1}^{N_{e}} (b_{K,e}^{i} - \tilde{b}_{K,e}^{i})^{2} 
\leq C \Sigma_{K^{\prime}\in N(e)} \Sigma_{i=1}^{N_{e}} (b_{K,e}^{i} - b_{K^{\prime},e}^{i})^{2}.
\end{align*}
We give an order of elements in $N(e)$ by $\{K^{j}\}_{j=1}^{\vert N(e)\vert}$, such that 
\begin{align*}
K^{j}\cap K^{j+1} \text{ is a face,}\quad \forall 1\leq j\leq \vert N(e) \vert -1.
\end{align*} 
Then, for any $1\leq i \leq N_{e}$, we have that 
\begin{align*}
\Sigma_{K^{\prime}\in N(e)} (b_{K,e}^{i} - b_{K^{\prime},e}^{i})^{2} 
\leq C \Sigma_{j=1}^{\vert N(e)\vert - 1} (b_{K^{j},e}^{i} - b_{K^{j+1},e}^{i})^{2}.
\end{align*}
So, we have that 
\begin{align*}
\Sigma_{i=1}^{N_{e}} (b_{K,e}^{i} - \tilde{b}_{K,e}^{i})^{2} 
\leq C \Sigma_{f\in \Cf(e)}\Sigma_{i=1}^{N_{e}} (b_{K_{f},e}^{i} - b_{K^{\prime}_{f},e}^{i})^{2}.
\end{align*}
Since $K_{f} = K^{\prime}_{f}$ if the face $f$ is on $\partial\Omega$, then 
\begin{align*}
\Sigma_{i=1}^{N_{e}} (b_{K,e}^{i} - \tilde{b}_{K,e}^{i})^{2} 
\leq C \Sigma_{f\in \Cf(e), f \nsubseteq \partial\Omega}
\Sigma_{i=1}^{N_{e}} (b_{K_{f},e}^{i} - b_{K^{\prime}_{f},e}^{i})^{2}.
\end{align*}
Then, by $(A.3)$ in step $4$ of the proof of \cite[Proposition~$4.5$]{Houston}, we have that 
\begin{align*}
\Sigma_{i=1}^{N_{e}} (b_{K,e}^{i} - \tilde{b}_{K,e}^{i})^{2}  
\leq C \Sigma_{f\in \Cf(e), f \nsubseteq \partial\Omega} \int_{f} 
\vert \llbracket\Bb_h \rrbracket_{T} \vert^{2} ds. 
\end{align*}

It is easy to see that for any face $f\in \Cf(K)$, we have that 
\begin{align*}
\Sigma_{i=1}^{N_{f}} (b_{K,f}^{i} - \tilde{b}_{K,f}^{i})^{2} 
\leq C \Sigma_{K^{\prime}\in N(f)} \Sigma_{i=1}^{N_{f}} (b_{K,f}^{i} - b_{K^{\prime},f}^{i})^{2} 
\leq C \int_{f\backslash \partial\Omega} \vert \llbracket\Bb_h \rrbracket_{T}\vert^{2} ds. 
\end{align*}

Combing the above estimates yields 
\begin{align*}
\Vert \Bb_{h} -\tilde{\Bb}_{h} \Vert_{L^{2}(K)}^{2} 
\leq C h_{K} \big[ \Sigma_{e\in \Ce(K)} \Sigma_{f\in \Cf(e), f\nsubseteq \partial\Omega}
\int_{f} \vert \vert \llbracket\Bb_h \rrbracket_{T} \vert^{2} ds
+\Sigma_{f\in \Cf(K), f\nsubseteq \partial\Omega}\int_{f} 
\vert \vert \llbracket\Bb_h \rrbracket_{T}\vert^{2} ds\big].
\end{align*}
Summing over all elements, taking into account the shape-regularity of the mesh, we finish the proof. 
 
%%%%%%%%%%%%%%%%%%%%%%%%%%%%%%%%%%%%%%%%%%%%%%%%%%%%%%%%%%%%%%%%%%%%%%

\end{document}